\documentclass[10pt]{article}

\usepackage{amsfonts}
\usepackage{amssymb}
\usepackage{amsthm}
\usepackage{amsmath}
\usepackage{bm}
\usepackage{amscd} 
\usepackage{graphicx}
\usepackage{mathrsfs}
\usepackage{relsize}
\usepackage{amsrefs}
\usepackage{comment}
\usepackage[all]{xy}
\usepackage{MnSymbol}
\usepackage{tikz-cd}

\usepackage{fullpage}
\usepackage{hyperref}

\numberwithin{equation}{section} \DeclareMathSizes{2}{10}{12}{13}
\newtheorem{theorem}{Theorem}[section]
\newtheorem{lemma}[theorem]{Lemma}

\newtheorem{prop}[theorem]{Proposition}
\newtheorem{corollary}[theorem]{Corollary}
\newtheorem{remark}[theorem]{Remark}
\newtheorem{definition}[theorem]{Definition}

\numberwithin{equation}{section}

\title{Fredholm modules over categories, Connes periodicity and classes in cyclic cohomology}

\author{Mamta Balodi\footnote{Email: mamta.balodi@gmail.com} \footnote{MB was supported by 
SERB
Fellowship PDF/2017/000229}  $\qquad$ Abhishek Banerjee \footnote{Email: abhishekbanerjee1313@gmail.com} \footnote{AB was partially supported by SERB Matrics fellowship MTR/2017/000112}}

\date{}

\begin{document}

\maketitle

\centerline{\emph{Department of Mathematics, Indian Institute of Science, Bangalore - 560012, India.}}

\begin{abstract}
We replace a ring with a small $\mathbb C$-linear category $\mathcal{C}$, seen as a ring with several objects in the sense of Mitchell. We introduce Fredholm modules over this category and construct a Chern character taking values in the cyclic cohomology of $\mathcal C$. We show that this categorified Chern character is homotopy invariant and is well-behaved with respect to the periodicity operator in cyclic cohomology. For this, we also obtain a description of cocycles and coboundaries in the cyclic cohomology of $\mathcal C$ (and more generally, in the Hopf-cyclic cohomology of a Hopf module category) by means of  DG-semicategories equipped with a trace on endomorphism spaces. 

\end{abstract}

\medskip
{\bf MSC(2010) Subject Classification:}  18E05,   47A53, 53C99, 58B34 

\medskip
{\bf Keywords: } Hopf-module categories, Differential graded categories, Fredholm modules, cyclic cohomology

\section{Introduction}

In his celebrated work \cite{C2}, Connes extended differential calculus beyond the  framework of manifolds to include noncommutative spaces such as that of leaves of a foliation or the orbit space of the action of a group on a manifold. For this, he began by considering Fredholm modules over an algebra $A$ which could in general be noncommutative. When $A$ is commutative, such as the space of smooth functions on a manifold $M$, examples of Fredholm modules over $A$ may be obtained by considering elliptic operators on $M$. More generally, by considering Schatten classes inside the collection of bounded operators on a Hilbert space, Connes studied the notion of $p$-summable Fredholm modules over $A$ in \cite{C2}. The Fredholm modules over $A$ lead to Chern characters taking values in the cyclic cohomology of $A$. Moreover, these cohomology classes are related by means of Connes' periodicity operator.

\smallskip

In this paper, we study Fredholm modules over linear categories, along with their Chern characters taking values in cyclic cohomology. Our idea is to have a counterpart of the algebraic notion of modules over a category, a subject which has been highly developed in the literature (see, for instance,  \cite{AB}, \cite{EV}, \cite{LV2}, \cite{Low2}, \cite{Xu1}, \cite{Xu2}). A small preadditive category is treated as a ring with several objects, following an idea first advanced by Mitchell \cite{Mit1}. We note that there is also a well-developed study of spaces in algebraic geometry over categories (see, for instance, \cite{Del}, \cite{EV}, \cite{TV}).   It is also important to mention here the work of Baez \cite{Baez} with the category of Hilbert spaces as well as the  recent work of Henriques \cite{h}, Henriques and Penneys \cite{hp} with  fusion categories with potential applications to physics. 

\smallskip
Let $\mathcal C$ be a small linear category.  We consider pairs $(\mathscr H,\mathcal F)$, where $\mathscr H$ is a functor
\begin{equation}\label{1/1wo}
\mathscr H:\mathcal C\longrightarrow SHilb_{\mathbb Z_2}
\end{equation}  taking values in $\mathbb Z_2$-graded separable Hilbert spaces and $\mathcal F=\{\mathcal F_X:\mathscr H(X)
\longrightarrow \mathscr H(X)\}_{X\in Ob(\mathcal C)}$ is a family of bounded and involutive linear operators each of degree $1$. When the elements of $\mathcal F$ satisfy certain commutator conditions with respect to the operators $\{\mathscr H(f)\}_{f\in Mor(\mathcal C)}$, we say that the pair $(\mathscr H,\mathcal F)$ is a Fredholm module over the category $\mathcal C$. Following the methods of Connes \cite{C2}, we construct Chern characters of these Fredholm modules taking values in the cyclic cohomology of $\mathcal C$ and study how they are related by means of the periodicity operator. We hope this is the first step towards  a larger program which mixes together the techniques in categorical algebra with those in differential geometry. 

\smallskip
The paper consists of two parts. In the first part, we study cyclic cohomology. We work more generally with a small  linear category $\mathcal D_H$ whose morphism spaces carry a well-behaved action of a Hopf algebra $H$. In other words, $\mathcal D_H$ is a small Hopf-module category (or $H$-category) in the sense of Cibils and Solotar \cite{CiSo}. We recall that in \cite{CM0}, \cite{CM1}, \cite{CM2}, Connes and Moscovici introduced Hopf-cyclic cohomology as a generalization of Lie algebra
cohomology adapted to noncommutative geometry.   For an $H$-category $\mathcal D_H$, we  describe  the cocycles and coboundaries that determine
its Hopf cyclic cohomology groups by extending Connes' original construction of cyclic cohomology from
\cite{C1} and \cite{C2} in terms of cycles and closed graded traces on differential graded algebras. An important role
in our paper is played by ``semicategories,'' which are categories that may not contain identity maps. This
notion, introduced by Mitchell \cite{Mit2}, is precisely what we need in order to categorify non-unital algebras. We
work with the Hopf cyclic cohomology groups $HC^\bullet_H(\mathcal D_H,M)$ having coefficients in $M$, where $M$ is a stable
anti-Yetter Drinfeld module in the sense of \cite{hkrs}.

\smallskip
Let $k$ be a field. After collecting some preliminaries in Section 2, we begin in Section 3 by considering the universal differential graded Hopf module semicategory (or DGH-semicategory) associated to the $H$-category $\mathcal D_H$. For a DGH-semicategory $(\mathcal S_H,\hat\partial_H)$ and $n\geq 0$, we let an $n$-dimensional closed graded $(H,M)$-trace on $\mathcal S_H$ be a collection of maps 
\begin{equation}\label{1.2eo} \hat{\mathscr{T}}^H:=\{\hat{\mathscr{T}}_X^H:M \otimes Hom^n_{\mathcal S_H}(X,X) \longrightarrow k\}_{X \in Ob(\mathcal S_H)}
\end{equation}
satisfying certain conditions (see Definition \ref{gradedtrace}). A cycle over $\mathcal D_H$ consists of a triple $(\mathcal S_H,\hat\partial_H,\hat{\mathscr T}^H)$ along with an $H$-linear semifunctor $\rho:\mathcal D_H\longrightarrow \mathcal S_H^0$. In Theorem \ref{charcycl}, we provide a description of the cocycles $Z^\bullet_H(\mathcal D_H,M)$ in Hopf cyclic cohomology in terms of characters of cycles over $\mathcal D_H$. This result is an $H$-linear categorical
version of Connes' \cite[Proposition 1, p. 98]{C2}.  It also follows from Theorem \ref{charcycl} that there is a one-one correspondence between $Z^\bullet_H(\mathcal D_H,M)$ and the collection of $n$-dimensional closed graded $(H,M)$-traces
on the universal DGH-semicategory $\Omega(\mathcal D_H)$ associated to $\mathcal D_H$. 

\smallskip
In Sections 4 and 5, we provide a description of the space $B^\bullet_H(\mathcal D_H,M)$ of coboundaries. Throughout, we take $k=\mathbb C$. We consider families $\eta$ 
of automorphisms $\eta = \{\eta (X)\in  Aut_{\mathcal D_H}(X)\}_{X\in Ob(\mathcal D_H)}$ such that
$
h(\eta(X))=\epsilon(h)\eta(X)
$ for all $h\in H$ and $X\in Ob(\mathcal D_H)$. 
We show that these families form a group, which we denote by $\mathbb U_H(\mathcal D_H)$. Further, we show that the inner
automorphism of $\mathcal D_H$ induced by conjugating with an element $\eta\in \mathbb U_H(\mathcal D_H)$ induces the identity functor on
$HC^\bullet_H(\mathcal D_H,M)$. Using this, we obtain in Proposition \ref{vanishing} a set of sufficient
conditions for the Hopf cyclic cohomology of an $H$-category to be zero. 

\smallskip
We say that a cycle $(\mathcal S_H,\hat\partial_H,\hat{\mathscr T}^H)$ is vanishing if  $\mathcal S_H^0$ is an $H$-category and $\mathcal S_H^0$ satisfies the assumptions in
Proposition \ref{vanishing}. We describe the elements of $B^\bullet_H(\mathcal D_H,M)$  in Theorem \ref{cob1} as the characters of vanishing
cycles over $\mathcal D_H$. Finally, in Theorem \ref{Thmfin}, we use categorified cycles and vanishing cycles to construct a
product in Hopf-cyclic cohomologies
\begin{equation}\label{1.3ep} HC^p_H(\mathcal{D}_H,M) \otimes HC^q_H(\mathcal{D}'_H,M') \longrightarrow HC^{p+q}_H(\mathcal{D}_H \otimes \mathcal{D}'_H,M \square_H M')\qquad p,q\geq 0 
\end{equation}
where $\mathcal D_H$ and $\mathcal D'_H$ are $H$-linear categories and $M$ and $M'$ are stable anti-Yetter Drinfeld modules over $H$ satisfying certain conditions.

\smallskip
In the second part of the paper, we study Fredholm modules and Chern classes. For this, we assume $H=\mathbb C=M$ and consider a small $\mathbb C$-linear category $\mathcal C$. Let $p\geq 1$ be an integer. We will say that a  pair $(\mathscr H,
\mathcal F)$  over $\mathcal C$ as in \eqref{1/1wo} is a $p$-summable Fredholm module if it satisfies
\begin{equation}\label{1.4jh}
[\mathcal{F},f]:= \left(\mathcal{F}_Y \circ \mathscr{H}(f) - \mathscr{H}(f) \circ \mathcal{F}_X \right)     \in \mathcal{B}^p\left(\mathscr{H}(X),\mathscr{H}(Y)\right)
\end{equation} for any morphism $f:X\longrightarrow Y$ in $\mathcal C$ (see Definition \ref{dD3.1}). Here, $\mathcal{B}^p\left(\mathscr{H}(X),\mathscr{H}(Y)\right)$ is the $p$-th Schatten class inside the space of bounded linear operators from $\mathscr H(X)$
to $\mathscr H(Y)$. We mention here that in this paper, we will consider only even Fredholm modules. We hope to tackle the case of odd Fredholm modules over linear categories in a future paper \cite{AMf}. 

\smallskip Let $H^\bullet_\lambda(\mathcal C):=HC^\bullet_{\mathbb C}(\mathcal C,\mathbb C)$ denote the cyclic cohomology groups of $\mathcal C$. Corresponding to a $p$-summable Fredholm module $(\mathscr H,
\mathcal F)$  and any $2m\geq p-1$, we construct a DG-semicategory $(\Omega_{(\mathscr H,\mathcal F)}\mathcal C,\partial')$ along with a closed graded trace $\hat{Tr}_s=\{Tr_s:Hom^{2m}_{\Omega_{(\mathscr H,\mathcal F)}}(X,X) \longrightarrow \mathbb C\}_{X\in Ob(\mathcal C)}$ of dimension $2m$. Let $CN_\bullet(\mathcal C)$ denote the cyclic nerve of $\mathcal C$ and $CN^\bullet(\mathcal C)$ its linear dual. By taking the character of the cycle $(\Omega_{(\mathscr H,\mathcal F)}\mathcal C,\partial',\hat{Tr}_s)$ over $\mathcal C$, we obtain $\phi^{2m}\in CN^{2m}(\mathcal C)$ which is given by (see Theorem \ref{evencyc})
\begin{equation}
\phi^{2m}(f^0 \otimes f^1 \otimes \ldots \otimes f^{2m}):= Tr_s\left(\mathscr H({f}^0)[\mathcal F,f^1][\mathcal F,f^2] \ldots [\mathcal F,f^{2m}]\right)
\end{equation} for any $f^0 \otimes f^1 \otimes \ldots \otimes f^{2m} \in CN_{2m}(\mathcal C)$. Then, $\phi^{2m}$ lies in the space $Z^{2m}_\lambda(\mathcal C)$ of cocycles for the cyclic cohomology of $\mathcal C$. The Chern character $ch^{2m}(\mathscr H,\mathcal F)$ of the Fredholm module $(\mathscr H,\mathcal F)$ will be the class of $\phi^{2m}$ in the cyclic cohomology $H^{2m}_\lambda(\mathcal C)$ of $\mathcal C$.

\smallskip
We relate the Chern characters by means of the periodicity operator in Section 7. We know that the action of the periodicity operator $S:H^\bullet_\lambda(\mathcal C)\longrightarrow H^{\bullet+2}_\lambda(\mathcal C)$ is given by taking the product as in 
\eqref{1.3ep} with a certain class in the cohomology $H^2_{\lambda}(\mathbb C)$. If $(\mathscr H,\mathcal F)$ is a $p$-summable Fredholm module over $\mathcal C$ and $2m\geq p-1$, we show in Theorem \ref{Periodthmo} that
\begin{equation}
S(\phi^{2m})=-(m+1)\phi^{2m+2} \qquad \text{in}~ H^{2m+2}_\lambda(\mathcal{C})
\end{equation}
Finally, in Section 8, we describe the homotopy invariance of the Chern character. For this, we consider a family $\{(\rho_t,\mathcal F_t)\}_{t\in [0,1]}$ of $p$-summable Fredholm modules
\begin{equation}
\{\rho_t:\mathcal{C} \longrightarrow SHilb_{\mathbb Z_2}\}_{t \in [0,1]}\qquad \mathcal F_t(X):\rho_t(X)\longrightarrow  \rho_t(X)
\end{equation}
each having the same underlying Hilbert space and satisfying some conditions. Then, if the $\rho_t$ and $\mathcal F_t$ vary in a strongly continuous manner with respect to $t\in [0,1]$, we show in Theorem \ref{TfinalT} that the $(p+2)$-dimensional character $ch^{p+2}(\mathscr H_t,\mathcal F_t)\in H^{p+2}_\lambda(\mathcal C)$ is independent of $t\in [0,1]$.

\medskip

\textbf{Notations:}
Throughout the paper, $H$ is a Hopf algebra over the field  $k$ of characteristic zero, with comultiplication $\Delta$, counit $\varepsilon$ and bijective antipode $S$. We will use  Sweedler's notation for the coproduct $\Delta(h)= h_1 \otimes h_2$ and for a left $H$-coaction $\rho:M \longrightarrow H \otimes M$, $\rho(m)= m_{(-1)} \otimes m_{(0)}$ (with the summation sign suppressed). The small cyclic category of Connes \cite{C1} will be denoted by $\Lambda$. The Hochschild differential will always  be denoted by $b$ and the modified Hochschild differential (with the last face operator missing) will be denoted by $b'$. 

\smallskip
On any cocyclic module $\mathscr C$, we will denote by $\tau_n$ the unsigned cyclic operator on $C^n(\mathscr C)$ and by $\lambda_n$ the signed cyclic operator $(-1)^n\tau_n$ on $C^n(\mathscr C)$. The complex computing cyclic cohomology of $\mathscr C$ will be denoted by $C^\bullet_\lambda(\mathscr C) := Ker(1-\lambda)$. Accordingly, the cyclic cocycles and cyclic coboundaries will be denoted by $Z^\bullet_\lambda(\mathscr C)$ and $B^\bullet_\lambda(\mathscr C)$ respectively.

\medskip

\textbf{Acknowledgements:} We are grateful to Gadadhar Misra for several useful discussions.

\section{Preliminaries on $H$-categories and Hopf-cyclic cohomology}
A small Hopf module category may be treated as a ``Hopf algebra with several objects.''  In this section, we will collect some preliminaries on Hopf module categories and on Hopf cyclic cohomology. We note that the Hopf cyclic cohomology introduced by Connes and Moscovici (\cite{CM0}, \cite{CM1}, \cite{CM2}) has been developed extensively by a number of authors (see, for instance, \cite{AK}, \cite{MB}, 
\cite{hkrs2}, \cite{Hask}, \cite{Hask1}, \cite{kyg}, \cite{kyg1}, \cite{kayx}, \cite{kr}, \cite{KoSh}, \cite{Rangipu}). 

\begin{definition}\label{defH-cat} (see Cibils and Solotar \cite{CiSo}) 
Let $H$ be a Hopf algebra over a field $k$. A $k$-linear category $\mathcal{D}_H$ is said to be a left $H$-module category if
\begin{itemize}
\item[(i)] $Hom_{\mathcal{D}_H}(X,Y)$ is a left $H$-module for all $X,Y \in Ob(\mathcal{D}_H)$ 
\item[(ii)] $h(\text{id}_X)=\varepsilon(h)\text{id}_X$ for all $X \in Ob(\mathcal{D}_H)$ and $h \in H$
\item[(iii)] the composition map is a morphism of $H$-modules, i.e., 
$h(gf)=(h_1g)(h_2f)$
for any $h \in H$, $f \in Hom_{\mathcal{D}_H}(X,Y)$ and  $g \in Hom_{\mathcal{D}_H}(Y,Z)$. 
\end{itemize} A small left $H$-module category will be called a left $H$-category. We will denote by $Cat_H$ the category of all left $H$-categories with $H$-linear functors between them.
\end{definition}

 For more on Hopf-module categories, we refer the reader, for instance, to \cite{BBR1}, \cite{BBR2}, \cite{HS}, \cite{kk}. Let $\mathcal{D}_H$ be a left $H$-category. We set
\begin{equation}\label{cnerf}CN_n(\mathcal{D}_H):=\bigoplus Hom_{\mathcal{D}_H}(X_1,X_0) \otimes Hom_{\mathcal{D}_H}(X_2,X_1) \otimes \ldots \otimes Hom_{\mathcal{D}_H}(X_0,X_n)
\end{equation}
where the direct sum runs over all $(X_0,X_1,\ldots,X_n) \in Ob(\mathcal{D}_H)^{n+1}$. 

\begin{lemma}
Let $M$ be a right $H$-module. For each $n\geq 0$, $M \otimes CN_n(\mathcal{D}_H)$ is a right $H$-module with action determined by
$$(m \otimes f^0 \otimes \ldots \otimes f^{n})h:= mh_1 \otimes S(h_2)(f^0 \otimes \ldots \otimes f^{n})$$
for any $m \in M$, $f^0 \otimes \ldots \otimes f^{n} \in CN_n(\mathcal D_H)$ and $h \in H$.
\end{lemma}

We now recall the notion of a stable anti-Yetter-Drinfeld module (SAYD) module from \cite[Definition 2.1]{hkrs2}.
\begin{definition}
Let $H$ be a Hopf algebra with a bijective antipode $S$. A $k$-vector space $M$ is said to be a right-left anti-Yetter-Drinfeld module over $H$ if $M$ is a right $H$-module and a left $H$-comodule such that
\begin{equation}\label{SAYDcondi}
\rho(mh)=(mh)_{(-1)} \otimes (mh)_{(0)}= S(h_3)m_{(-1)}h_1 \otimes m_{(0)}h_2
\end{equation}
for all $m \in M$ and $h \in H$, where $\rho: M \longrightarrow H \otimes M,~ m \mapsto m_{(-1)} \otimes m_{(0)}$ is the coaction. Moreover, $M$ is said to be stable if $m_{(0)}m_{(-1)}=m$.
\end{definition}

\begin{comment}
For an SAYD module $M$ over $H$, we  note that
\begin{equation}\label{2obsv}m_{(0)}S^{-1}(m_{(-1)})=m_{(0)(0)}m_{(0)(-1)}S^{-1}(m_{(-1)})=m_{(0)}m_{(-1)2}S^{-1}(m_{(-1)1})=m_{(0)}\varepsilon(m_{(-1)})=m
\end{equation}
\end{comment}

We now take the Hopf-cyclic cohomology $HC^\bullet_H(\mathcal{D}_H,M)$ of an $H$-category $\mathcal{D}_H$ with coefficients in a SAYD module $M$ (see also \cite{kk}). This generalizes the construction of the Hopf-cyclic cohomology for $H$-module algebras with coefficients in an SAYD module (see \cite{hkrs} and also \cite{carthy}).  For each $n \geq 0$, we set 
\begin{equation*}
\begin{array}{ll}
C^n(\mathcal{D}_H,M):=Hom_k(M \otimes CN_n(\mathcal{D}_H),k) \qquad \qquad C^n_H(\mathcal{D}_H,M):=Hom_H(M \otimes CN_n(\mathcal{D}_H),k)
\end{array}
\end{equation*}
where $k$ is considered as a right $H$-module via the counit. It is clear from the definition that an element in $C^n_H(\mathcal{D}_H,M)$ is a $k$-linear map $\phi:M \otimes CN_n(\mathcal{D}_H) \longrightarrow k$ satisfying
\begin{equation}\label{new3.2}
\phi\left(mh_1 \otimes S(h_2)(f^0 \otimes \ldots \otimes f^{n})\right)=\varepsilon(h)\phi(m \otimes f^0 \otimes \ldots \otimes f^{n})
\end{equation}

We recall that a (co)simplicial module is said to be para-(co)cyclic if all the relations for a (co)cyclic module are satisfied except $\tau_n^{n+1}=id$ (see, for instance \cite{kk}). The following may be verified directly.

\begin{prop}\label{prop2.3} Let $\mathcal{D}_H$ be a left $H$-category and let $M$ be a right-left SAYD module over $H$. Then, 

\smallskip
(1) we have a para-cocyclic module $C^\bullet(\mathcal{D}_H,M):=\{C^n(\mathcal{D}_H,M)\}_{n \geq 0}$ with the following structure maps  
\begin{align*}
(\delta_i\phi)(m \otimes f^0 \otimes \ldots \otimes f^{n})&= \begin{cases}
\phi(m \otimes f^0 \otimes \ldots \otimes f^if^{i+1} \otimes \ldots \otimes f^{n}) \quad~~~~~~~~~~~~~~~~ 0 \leq i \leq n-1\\
\phi\big(m_{(0)} \otimes \big(S^{-1}(m_{(-1)})f^n\big)f^0 \otimes \ldots \otimes f^{n-1}\big) \quad~~~~~~~~~~~~i=n\\
\end{cases}\\
\vspace{.2cm}
(\sigma_i\psi)(m \otimes f^0 \otimes \ldots \otimes f^{n})&= \begin{cases} \psi(m \otimes f^0 \otimes \ldots \otimes f^i \otimes id_{X_{i+1}}\otimes f^{i+1} \otimes \ldots \otimes f^{n}) \quad~~~~~ 0 \leq i \leq n-1\\
\psi(m \otimes f^0 \otimes \ldots \otimes f^{n} \otimes id_{X_{0}}) \quad~~~~~~~~~~~~~~~~~~~~~~~~~~ i=n\\
\end{cases}\\ 
\vspace{.2cm}
(\tau_n\varphi)(m \otimes f^0 \otimes \ldots \otimes f^{n})&=\varphi\big(m_{(0)} \otimes S^{-1}(m_{(-1)})f^n \otimes f^0 \otimes \ldots \otimes f^{n-1}\big)
\end{align*}
for any $\phi \in C^{n-1}(\mathcal{D}_H,M)$, $\psi \in C^{n+1}(\mathcal{D}_H,M)$, $\varphi \in C^{n}(\mathcal{D}_H,M)$, $m \in M$ and $ f^0 \otimes \ldots \otimes f^n \in Hom_{\mathcal{D}_H}(X_1,X_0) \otimes Hom_{\mathcal{D}_H}(X_2,X_1) \otimes \ldots \otimes Hom_{\mathcal{D}_H}(X_0,X_n)$.

\smallskip
(2) by restricting to right $H$-linear morphisms $C^n_H(\mathcal{D}_H,M)=Hom_H(M \otimes CN_n(\mathcal{D}_H),k),$ we obtain a cocyclic module $C^\bullet_H(\mathcal{D}_H,M):=\{C^n_H(\mathcal{D}_H,M)\}_{n \geq 0}$. 
\end{prop}

The cohomology of the cocyclic module $C^\bullet_H(\mathcal{D}_H,M)$ is referred to as the Hopf-cyclic cohomology of the $H$-category $\mathcal{D}_H$ with coefficients in the SAYD module $M$. The corresponding cohomology groups are denoted by $HC_H^\bullet(\mathcal{D}_H,M)$.

\begin{remark}
\emph{
As $k$ contains $\mathbb{Q}$, we recall that the cohomology of a cocyclic module $\mathscr{C}$ can be expresed alternatively as the cohomology of the following complex (see, for instance \cite[2.5.9]{Loday}):}
\begin{equation}\label{R2.5c}
\begin{CD}
C^0_\lambda(\mathscr{C}) @> b>> \dots @ > b>> C^n_\lambda(\mathscr{C}) @ > b>> C^{n+1}_\lambda(\mathscr{C})@> b>>\dots
\end{CD}
\end{equation}
\emph{where $C^n_\lambda(\mathscr{C})=Ker(1-\lambda)\subseteq C^n(\mathscr{C})$, $b=\sum_{i=0}^{n+1} (-1)^i \delta_i$ and 
$\lambda=(-1)^n \tau_n$. In particular, an element $\phi \in C^n_H(\mathcal{D}_H,M)$ is a cyclic cocycle if and only if}
\begin{equation}\label{3.3y}
b(\phi)=0 \quad  \text{and} \quad (1-\lambda)(\phi)=0
\end{equation}\emph{ In this paper, the cocycles and coboundaries of a cocyclic module will always refer to this complex. }
\end{remark}

\begin{prop}\label{P6.2wp} Let $\mathcal{D}_H$ be a left $H$-category and let $M$ be a right-left SAYD module. Then:

\smallskip
(1) We obtain a para-cyclic module $C_\bullet(\mathcal{D}_H,M):=\{C_n(\mathcal{D}_H,M):=M \otimes CN_n(\mathcal{D}_H)\}_{n \geq 0}$ with the following structure maps  
\begin{align*}
d_i(m \otimes f^0 \otimes \ldots \otimes f^{n})&= \begin{cases}
m \otimes f^0 \otimes f^1 \otimes \ldots \otimes f^if^{i+1} \otimes \ldots \otimes f^{n} \quad~~~~~~~~~~~~~~~~ 0 \leq i \leq n-1\\
m_{(0)} \otimes \big(S^{-1}(m_{(-1)})f^n\big)f^0 \otimes f^1 \otimes \ldots \otimes f^{n-1} \quad~~~~~~~~~~~~i=n\\
\end{cases}\\
\vspace{.2cm}
s_i(m \otimes f^0 \otimes \ldots \otimes f^{n})&= \begin{cases} m \otimes f^0 \otimes f^1 \otimes \ldots f^i \otimes id_{X_{i+1}}\otimes f^{i+1} \otimes \ldots \otimes f^{n} \quad~~~~~ 0 \leq i \leq n-1\\
m \otimes f^0 \otimes f^1 \otimes \ldots \otimes f^{n} \otimes id_{X_{0}} \quad~~~~~~~~~~~~~~~~~~~~~~~~~~ i=n\\
\end{cases}\\ 
\vspace{.2cm}
t_n(m \otimes f^0 \otimes \ldots \otimes f^{n})&=m_{(0)} \otimes S^{-1}(m_{(-1)})f^n \otimes f^0 \otimes \ldots \otimes f^{n-1}
\end{align*}
for any $m \in M$ and $ f^0 \otimes f^1 \otimes \ldots \otimes f^n \in Hom_{\mathcal{D}_H}(X_1,X_0) \otimes Hom_{\mathcal{D}_H}(X_2,X_1) \otimes \ldots \otimes Hom_{\mathcal{D}_H}(X_0,X_n)$.

\smallskip
(2) By passing to the tensor product over $H$,   we obtain a cyclic module $C_\bullet^H(\mathcal{D}_H,M):=\{C_n^H(\mathcal{D}_H,M)=M \otimes_H CN_n(\mathcal{D}_H)\}_{n \geq 0}$. 
\end{prop}

The cyclic homology groups corresponding to the cyclic module $C_\bullet^H(\mathcal{D}_H,M)$ will be denoted by $HC^H_\bullet(\mathcal{D}_H,M)$. 

\section{Traces, cocycles and DGH-semicategories}\label{SectionDG}

We continue with $\mathcal D_H$ being a left $H$-category and $M$ a right-left SAYD module over $H$. 
Our purpose  is to develop a formalism analogous to that of Connes \cite{C2} in order to interpret the cocycles $Z^\bullet_H(\mathcal D_H,M)$ of the complex  $C^\bullet_H(\mathcal D_H,M)$  and its coboundaries $B^\bullet_H(\mathcal D_H,M)$  as characters of differential graded  semicategories. 
 In this section, we  will describe $Z^\bullet_H(\mathcal{D}_H,M)$, for which we will need  the framework of DG-semicategories. Let us first recall the notion of a semicategory introduced by Mitchell in \cite{Mit} (for more on semicategories, see, for instance, \cite{BoLS}). 

\begin{definition}(see \cite[Section 4]{Mit})
A semicategory $\mathcal C$ consists of a collection $Ob(\mathcal{C})$ of objects together with a set of morphisms $Hom_{\mathcal{C}}(X,Y)$ for each $X,Y \in Ob(\mathcal{C})$ and an associative composition. A semifunctor $F:\mathcal{C} \longrightarrow \mathcal{C}'$ between semicategories assigns an object $F(X) \in Ob(\mathcal{C}')$ to each $X \in Ob(\mathcal{C})$
and a morphism $F(f) \in Hom_{\mathcal C'}(F(X),F(Y))$ to each $f \in Hom_\mathcal{C}(X,Y)$ and preserves composition. 

\smallskip
A left $H$-semicategory is a small $k$-linear semicategory $\mathcal{S}_H$ such that
\begin{itemize}
\item[(i)] $Hom_{\mathcal{S}_H}(X,Y)$ is a left $H$-module for all $X,Y \in Ob(\mathcal{S}_H)$ 
\item[(ii)] $h(gf)=(h_1g)(h_2f)$
for any $h \in H$, $f \in Hom_{\mathcal{S}_H}(X,Y)$ and  $g \in Hom_{\mathcal{S}_H}(Y,Z)$. 
\end{itemize}
\end{definition}

It is clear that any ordinary category may be treated as a semicategory. Conversely, to any $k$-semicategory $\mathcal{C}$, we can associate an ordinary $k$-category
$\tilde{\mathcal{C}}$ by setting $Ob(\tilde{\mathcal{C}})=Ob(\mathcal{C})$ and adjoining  unit morphisms as follows:
\begin{align*}
Hom_{\tilde{\mathcal{C}}}(X,Y):&=\left\{\begin{array}{ll} Hom_\mathcal{C}(X,X) \bigoplus k &  \mbox{if $X=Y$}\\
Hom_\mathcal{C}(X,Y) & \mbox{if $X \neq Y$} \\ \end{array}\right.
\end{align*}

A  morphism in $Hom_{\tilde{\mathcal{C}}}(X,Y)$ will be denoted by $\tilde{f}=f+\mu$, where $f \in Hom_{\mathcal{C}}(X,Y)$ and $\mu \in k$. It is understood that $\mu=0$ whenever $X \neq Y$. Any semifunctor $F:\mathcal{C} \longrightarrow \mathcal{D}$ where $\mathcal D$ is an ordinary category may  be extended to an ordinary functor $\tilde{F}:\tilde{\mathcal{C}} \longrightarrow \mathcal{D}$. If $\mathcal S_H$ is a left $H$-semicategory, we note that $\tilde{\mathcal S}_H$ is a left $H$-category in the sense of Definition \ref{defH-cat}. 

\begin{definition}
A differential graded semicategory (DG-semicategory) $(\mathcal{S},\hat\partial)$ is a $k$-linear semicategory $\mathcal{S}$ such that
\begin{itemize}
\item[(i)] $Hom^\bullet_\mathcal{S}(X,Y)=\big(Hom^n_\mathcal{S}(X,Y),\hat\partial^n_{XY}\big)_{n \geq 0}$ is a cochain complex of $k$-spaces for each $X,Y \in Ob(\mathcal{S})$.

\item[(ii)] the composition map
$Hom^\bullet_\mathcal{S}(Y,Z) \otimes Hom^\bullet_\mathcal{S}(X,Y) \longrightarrow Hom^\bullet_\mathcal{S}(X,Z)$
is a morphism of complexes. Equivalently, we have $
\hat\partial^n_{XZ}(gf)=\hat\partial^{n-r}_{YZ}(g)f+(-1)^{n-r}g\hat\partial^{r}_{XY}(f) \label{comp}
$
for any $f \in Hom_\mathcal{S}(X,Y)^r$ and $g \in Hom_\mathcal{S}(Y,Z)^{n-r}$.
\end{itemize} Whenever the meaning is clear from context, we will drop the subscript and simply write $\hat\partial^\bullet$ for the differential
on any $Hom^\bullet_{\mathcal S}(X,Y)$. 
\end{definition}

A  small DG-semicategory may be treated as a differential 
graded (but  not necessarily unital) $k$-algebra with several objects. The DG-semicategories may be treated in a manner similar to DG-categories (see, for instance, \cite{Ke1}, \cite{Ke2}). For instance, there is an obvious notion of DG-semifunctor between DG-semicategories. We also note that if $\mathcal S$ is a DG-semicategory, the morphisms in degree $0$ determine a 
semicategory $\mathcal S^0$.

\begin{comment}
\begin{definition}
A DG-semifunctor $\alpha:(\mathcal{S},\hat\partial)\longrightarrow (\mathcal{S}',\hat\partial')$ between two DG-semicategories  is a $k$-linear semifunctor $\alpha:\mathcal{S} \longrightarrow \mathcal{S}'$ such that the induced map $Hom^\bullet_\mathcal{S}(X,Y) \longrightarrow Hom^\bullet_{\mathcal{S}'}(\alpha X,\alpha Y)$, $f \mapsto \alpha (f)$, is a morphism of complexes for each $X,Y \in Ob(\mathcal{S})$. 
\end{definition}

\begin{remark}\label{0dg}
We observe that corresponding to any DG-semicategory $\mathcal{S}$, there is a semicategory $\mathcal{S}^0$ defined as:
\begin{align*}
Ob(\mathcal{S}^0):&=Ob(\mathcal{S})\\
Hom_{\mathcal{S}^0}(X,Y):&=Hom^0_\mathcal{S}(X,Y)
\end{align*}
The composition   in $\mathcal{S}$ induces a well-defined composition $Hom_{\mathcal{S}^0}(Y,Z) \otimes Hom_{\mathcal{S}^0}(X,Y) \longrightarrow Hom_{\mathcal{S}^0}(X,Z)$.
\end{remark}
\end{comment}

We now construct a ``universal DG-semicategory'' associated to a given $k$-linear semicategory, similar to the construction of the universal differential graded  algebra
associated to a  (not necessarily unital) $k$-algebra (see, for instance, \cite[p. 315]{C2}). 

\smallskip
Let $\Omega\mathcal{C}$ be the semicategory with $Ob(\Omega\mathcal{C}):=Ob(\mathcal{C})$ and  $Hom_{\Omega \mathcal{C}}(X,Y)=\bigoplus\limits_{n \geq 0}Hom^n_{\Omega \mathcal{C}}(X,Y)$, where
\begin{equation}\label{xudga} Hom^n_{\Omega\mathcal{C}}(X,Y):=
\left\{
\begin{array}{ll}
Hom_{\mathcal C}(X,Y) & \mbox{if $n=0$} \\
\\ \underset{(X_1,...,X_n)\in Ob(\mathcal C)^n}{\mbox{\Large $\bigoplus$}}Hom_{\tilde{\mathcal C}}(X_1,Y)\otimes Hom_{\mathcal C}
(X_2,X_1)\otimes \dots \otimes Hom_{\mathcal C}(X,X_n) & \mbox{if $n\geq 1$} \\
\end{array} \right.
\end{equation}
Here the sum runs over the ordered tuples $(X_1,...,X_n)\in Ob(\mathcal C)^n$.
 In particular, $(\Omega\mathcal{C})^0={\mathcal{C}}$.  For $n\geq 1$, an element of the form $\tilde{f}^0\otimes f^1\otimes ...\otimes f^n$ in $Hom^n_{\Omega\mathcal{C}}(X,Y)$ will be denoted by  $\tilde{f}^0df^1 \ldots df^n=(f^0+\mu) df^1\dots df^n$   
and said to be homogeneous of degree $n$. By abuse of notation, we will continue to use $\tilde{f}^0df^1 \ldots df^n=(f^0+\mu) df^1\dots df^n$ to denote an element 
of $Hom^n_{\Omega\mathcal{C}}(X,Y)$ even when $n=0$. In that case, it will be understood that $\mu=0$. 

\smallskip
The composition in $\Omega\mathcal{C}$ is determined by
\begin{equation}\label{symb}
f^0\circ df^1\circ \dots \circ df^n= f^0df^1\dots df^n \qquad (df^0)\circ f^1=d(f^0f^1)-f^0(df^1) \qquad  df^1\circ \dots \circ df^n= df^1\dots df^n
\end{equation}
In particular, it follows that
\begin{equation}\label{compoDG}
 \begin{array}{l}
((f^0+\mu)df^1...df^i)\cdot ((g^0+\mu')dg^1...dg^j)\\
= (f^0+\mu)\left(df^1....df^{i-1}d(f^ig^0)dg^1...dg^j+\underset{l=1}{\overset{i-1}{\sum}}(-1)^{i-l}df^1...d(f^lf^{l+1})...df^idg^0dg^1...dg^j\right)\\
\quad + (-1)^i(f^0+\mu)f^1df^2...df^idg^0dg^1...dg^j+\mu' (f^0+\mu)df^1...df^idg^1...dg^j\\
\end{array}
\end{equation}
For each $X,Y \in Ob(\Omega\mathcal{C})$,  the differential $\partial^n_{XY}:Hom^n_{\Omega\mathcal{C}}(X,Y) \longrightarrow Hom^{n+1}_{\Omega\mathcal{C}}(X,Y)$ is determined by setting
$$\partial^n_{XY}((f^0+\mu)df^1 \ldots df^n):=df^0df^1 \ldots df^n$$ It follows from definition that $\partial^{n+1}_{XY}\circ \partial^n_{XY}=0$. Therefore, $Hom^\bullet_{\Omega\mathcal{C}}(X,Y):=\big(Hom^n_{\Omega\mathcal{C}}(X,Y),\partial^n_{XY}\big)_{n \geq 0}$ is a cochain complex for each $X, Y \in Ob(\Omega\mathcal{C})$. It may also be verified  that the composition in $\Omega\mathcal{C}$ is a morphism of complexes. Thus, $\Omega\mathcal{C}$ is a DG-semicategory.
 
\begin{prop}\label{construniv} Let $\mathcal C$ be a small $k$-linear semicategory. Then, the associated DG-semicategory $(\Omega\mathcal C,\partial)$  is universal in the following sense: given
\begin{itemize}
\item[(i)] any DG-semicategory $(\mathcal{S},\hat\partial)$ and 
\item[(ii)] a $k$-linear semifunctor $\rho:\mathcal{C} \longrightarrow \mathcal{S}^0$,
\end{itemize}
there exists a unique DG-semifunctor $\hat{\rho}:(\Omega\mathcal{C},\partial) \longrightarrow (\mathcal{S},\hat\partial)$ such that the restriction of $\hat{\rho}$ to  the semicategory $\mathcal{C}$ is identical to $\rho: \mathcal{C} \longrightarrow \mathcal{S}^0$.
\end{prop}
\begin{proof}
We extend $\rho$ to obtain a DG-semifunctor $\hat{\rho}:(\Omega\mathcal{C},\partial) \longrightarrow (\mathcal{S},\hat{\partial})$ as follows:
\begin{equation}\label{cat1}
\begin{array}{c}
\hat{\rho}(X):=\rho(X)\\
\hat{\rho}((f^0+\mu)df^1\ldots df^n):=\rho(f^0)\circ\hat{\partial}^0(\rho(f^1))\circ \ldots\circ \hat{\partial}^0(\rho(f^n))+\mu \hat{\partial}^0(\rho(f^1)) \circ \ldots \circ \hat{\partial}^0(\rho(f^n))
\end{array}
\end{equation}
for all $X \in Ob(\Omega\mathcal{C})=Ob(\mathcal{C})$ and $(f^0+\mu)df^1\ldots df^n \in Hom^n_{\Omega\mathcal{C}}(X,Y)$, $n\geq 1$. Since each $\rho(f^i)$ is a morphism of degree $0$ in $\mathcal{S}$, it follows from \eqref{comp} and \eqref{compoDG} that 
\begin{equation} \hat{\rho}(((f^0+\mu)df^1...df^n)\circ ((f^{n+1}+\mu')df^{n+2}... df^m))=\hat{\rho}((f^0+\mu)df^1...df^n)\circ \hat{\rho}((f^{n+1}+\mu')df^{n+2}... df^m)
\end{equation} It is also clear by construction that $\hat{\rho}|_{\mathcal{C}}=\rho$. Moreover, we have
\begin{equation*}
\begin{array}{ll}
\hat{\partial}^n\left(\hat{\rho}((f^0+\mu)df^1\ldots df^n)\right)&= \hat{\partial}^n\left(\rho(f^0)\hat{\partial}^0(\rho(f^1)) \ldots \hat{\partial}^0(\rho(f^n))\right)+
\mu \hat{\partial}^n\left( \hat{\partial}^0(\rho(f^1)) \ldots \hat{\partial}^0(\rho(f^n))\right)\\
&= {\hat{\partial}}^0(\rho(f^0)) {\hat{\partial}}^0(\rho(f^1)) \ldots  {\hat{\partial}}^0(\rho(f^n))+\rho(f^0) {\hat{\partial}}^n\left({ \hat{\partial}}^0(\rho(f^1)) \ldots  {\hat{\partial}}^0(\rho(f^n))\right)\\
&= \hat{\partial}^0(\rho(f^0)) \hat{\partial}^0(\rho(f^1)) \ldots  \hat{\partial}^0(\rho(f^n))=\hat{\rho}\left(\partial^n((f^0+\mu)df^1\ldots df^n)\right)
\end{array}
\end{equation*}
The uniqueness of $\hat\rho$ is also clear from \eqref{symb} and \eqref{compoDG}. 
\end{proof}

\begin{definition}\label{DGH}
A left DGH-semicategory   is a left $H$-semicategory $\mathcal S_H$  equipped with a DG-semicategory $(\mathcal S_H,\hat\partial_H)$ structure such that for all $n\geq 0$:

\smallskip
(a)  $Hom^n_{\mathcal S_H}(X,Y)$ is a left $H$-module for $X,Y \in Ob(\mathcal S_H)$. 

(b) $\hat\partial^n_{H}:Hom^n_{\mathcal S_H}(X,Y) \longrightarrow Hom^{n+1}_{\mathcal S_H}(X,Y)$ is $H$-linear for $X,Y \in Ob(\mathcal S_H)$. 
\end{definition}

We can similarly define the notion of a DGH-semifunctor between DGH-semicategories. If $(\mathcal S_H,\hat\partial_H)$  is a left DGH-semicategory, we note that  $\mathcal S_H^0$  is a left $H$-semicategory.

\begin{prop}
Let $\mathcal{D}_H$ be a left $H$-category. Then, the universal DG-semicategory $(\Omega(\mathcal{D}_H),\partial_H)$ associated to $\mathcal{D}_H$ is a left DGH-semicategory with the $H$-action determined by
\begin{equation}\label{comp4.8}
h \cdot \left((f^0+\mu)df^1 \ldots df^n\right):=(h_1f^0+\mu \varepsilon(h_1))d(h_2f^1) \ldots d(h_{n+1}f^n) 
\end{equation}
for all $h \in H$ and $(f^0+\mu)df^1 \ldots df^n\in Hom_{\Omega(\mathcal D_H)}(X,Y)$.
\end{prop}
\begin{proof} This  is immediate from the definitions in \eqref{compoDG} and \eqref{comp4.8}.
\end{proof} 

\begin{definition}\label{gradedtrace}
Let $(\mathcal S_H,\hat\partial_H)$ be a left DGH-semicategory and $M$ be a right-left SAYD module over $H$. A closed graded $(H,M)$-trace of dimension $n$ on $\mathcal S_H$ is a collection of $k$-linear maps $$\hat{\mathscr{T}}^H:=\{\hat{\mathscr{T}}_X^H:M \otimes Hom^n_{\mathcal S_H}(X,X) \longrightarrow k\}_{X \in Ob(\mathcal S_H)}$$ such that
\begin{align}
&\hat{\mathscr{T}}_X^H\big(mh_1 \otimes S(h_2)f\big)=\varepsilon(h)\hat{\mathscr{T}}_X^H(m \otimes f)\label{gt0}\\
&\hat{\mathscr{T}}_X^H\big(m \otimes \hat\partial_{H}^{n-1}(f')\big)=0 \label{gt1}\\
&\hat{\mathscr{T}}_X^H\big(m \otimes g'g)=(-1)^{ij}~\hat{\mathscr{T}}_Y^{H}\big(m_{(0)} \otimes \left(S^{-1}(m_{(-1)})g\right)g'\big)\label{gt2}
\end{align}
for all $h \in H$, $m \in M$, $f \in Hom^n_{\mathcal S_H}(X,X)$, $f' \in Hom^{n-1}_{\mathcal S_H}(X,X)$, $g \in Hom^i_{\mathcal S_H}(X,Y)$, $g' \in Hom^j_{\mathcal S_H}(Y,X)$ and $i+j=n$.
\end{definition}

\begin{definition}\label{cycle} 
An $n$-dimensional $\mathcal S_H$-cycle with coefficients in a SAYD module $M$ is a triple 
$(\mathcal{S}_H,\hat{\partial}_H,\hat{\mathscr T}^H)$ such that
\begin{itemize}
\item[(i)] $(\mathcal S_H,\hat\partial_H)$ is a left DGH-semicategory.
\item[(ii)]  $\hat{\mathscr T}^H$ is a closed graded $(H,M)$-trace of dimension $n$ on $\mathcal{S}_H$.
\end{itemize} Let $\mathcal{D}_H$ be a left $H$-category. By an $n$-dimensional cycle over $\mathcal D_H$, we mean a tuple $(\mathcal{S}_H,\hat{\partial}_H, \hat{\mathscr T}^H,\rho)$ such that 
\begin{itemize}
\item[(i)] $(\mathcal{S}_H,\hat{\partial}_H, \hat{\mathscr T}^H)$ is an $n$-dimensional $\mathcal S_H$-cycle with coefficients in a SAYD module $M$.
\item[(ii)] $\rho:\mathcal{D}_H \longrightarrow \mathcal{S}_H^0$ is an $H$-linear semifunctor.
\end{itemize}
\end{definition}

 We fix a left $H$-category $\mathcal D_H$. Given an $n$-dimensional cycle $(\mathcal{S}_H,\hat{\partial}_H, \hat{\mathscr T}^H,\rho)$  over $\mathcal{D}_H$, we define its character $\phi\in C^n_H(\mathcal D_H,M)$ by setting
\begin{equation*}
\phi:M\otimes CN_n(\mathcal D_H)\longrightarrow k\qquad \phi(m \otimes f^0 \otimes \ldots \otimes f^n):=\hat{\mathscr T}^H_{X_0}\big(m \otimes \rho(f^0)\hat{\partial}^0_H\left(\rho(f^1)\right) \ldots \hat{\partial}^0_H\left(\rho(f^n)\right)\big)
\end{equation*} 
for $m \in M$ and $f^0 \otimes \ldots \otimes f^n \in Hom_{\mathcal{D}_H}(X_1,X_0) \otimes Hom_{\mathcal{D}_H}(X_2,X_1) \otimes \ldots \otimes Hom_{\mathcal{D}_H}(X_0,X_n)$. We will often suppress the semifunctor $\rho$ and refer to
$\phi$ simply as the character of the $n$-dimensional cycle   $(\mathcal{S}_H,\hat{\partial}_H, \hat{\mathscr T}^H)$. 

\smallskip
We now have a characterization of the space $Z^n_H(\mathcal{D}_H,M)$ of $n$-cocycles in the Hopf-cyclic cohomology of the category $\mathcal{D}_H$ with coefficients in the SAYD module $M$.

\begin{theorem}\label{charcycl}
Let $\mathcal{D}_H$ be a left $H$-category and $M$ be a right-left SAYD module over $H$. Let $\phi \in C^n_H(\mathcal{D}_H,M)$. Then, the following conditions are equivalent:
\begin{itemize}
\item[(1)] $\phi$ is the character of an $n$-dimensional cycle over $\mathcal D_H$, i.e., there is an $n$-dimensional cycle $(\mathcal{S}_H,\hat{\partial}_H, \hat{\mathscr T}^H)$ with coefficients in $M$ and an $H$-linear semifunctor $\rho:\mathcal{D}_H \longrightarrow \mathcal{S}_H^0$ such that
\begin{equation}\label{eq1}
\begin{array}{ll}
\phi(m \otimes f^0 \otimes \ldots \otimes f^n)
&=\hat{\mathscr T}^H_{X_0}((id_M \otimes \hat{\rho})(m \otimes f^0df^1\ldots df^n))\vspace{0.02in}\\ 
&=\hat{\mathscr T}^H_{X_0}\big(m \otimes \rho(f^0)\hat{\partial}_H^0\left(\rho(f^1)\right) \ldots \hat{\partial}_H^0\left(\rho(f^n)\right)\big)\\
\end{array}
\end{equation}
for any $m \in M$ and $f^0 \otimes \ldots \otimes f^n \in Hom_{\mathcal{D}_H}(X_1,X_0) \otimes Hom_{\mathcal{D}_H}(X_2,X_1) \otimes \ldots \otimes Hom_{\mathcal{D}_H}(X_0,X_n)$.
\item[(2)] There exists a closed graded $(H,M)$-trace $\mathscr{T}^H$ of dimension $n$ on $\left(\Omega(\mathcal{D}_H),\partial_H\right)$ such that
\begin{equation}\label{eq2}
\phi(m \otimes f^0 \otimes \ldots \otimes f^n)=\mathscr{T}^H_{X_0}(m \otimes f^0df^1 \ldots df^n)
\end{equation}
for any $m \in M$ and $f^0 \otimes \ldots \otimes f^n \in Hom_{\mathcal{D}_H}(X_1,X_0) \otimes Hom_{\mathcal{D}_H}(X_2,X_1) \otimes \ldots \otimes Hom_{\mathcal{D}_H}(X_0,X_n)$.
\item[(3)] $\phi \in Z^n_H(\mathcal{D}_H,M)$. 
\end{itemize}
\end{theorem}
\begin{proof}
(1) $\Rightarrow$ (2). By the universal property of $\Omega(\mathcal{D}_H)$, the $H$-linear semifunctor $\rho:\mathcal{D}_H \longrightarrow \mathcal{S}_H^0$ can be extended to a DGH-semifunctor $\hat{\rho}:\Omega(\mathcal{D}_H) \longrightarrow \mathcal{S}_H$ as in \eqref{cat1}.  We define a collection $\mathscr{T}^H:=\{\mathscr{T}^H_X:M \otimes Hom^n_{\Omega(\mathcal{D}_H)}(X,X) \longrightarrow k\}_{X \in Ob(\Omega(\mathcal{D}_H))}$  of $k$-linear maps  given by
\begin{equation}\label{ver1x}\mathscr{T}_X^H(m \otimes (f^0+\mu)df^1 \ldots df^n):=\hat{\mathscr T}^H_{X}\big(m \otimes \hat\rho((f^0+\mu)df^1 \ldots df^n)\big)
\end{equation} for any $m \in M$ and $f^0 \otimes \ldots \otimes f^n \in Hom_{\mathcal{D}_H}(X_1,X) \otimes Hom_{\mathcal{D}_H}(X_2,X_1) \otimes \ldots \otimes Hom_{\mathcal{D}_H}(X,X_n)$. In particular, it follows from \eqref{ver1x} that
\begin{equation}
\phi(m \otimes f^0 \otimes \ldots \otimes f^n)=\hat{\mathscr T}^H_{X}\big(m \otimes \rho(f^0)\hat{\partial}_H^0\left(\rho(f^1)\right) \ldots \hat{\partial}_H^0\left(\rho(f^n)\right)\big)=\mathscr{T}_X^H(m \otimes f^0df^1 \ldots df^n)
\end{equation} 
It may be verified that the collection $\mathscr{T}^H$ is an $n$-dimensional closed graded $(H,M)$-trace on $\Omega(\mathcal{D}_H)$. 

\smallskip
(2) $\Rightarrow$ (1). Suppose that we have a closed graded $(H,M)$-trace $\mathscr{T}^H$ of dimension $n$ on $\Omega(\mathcal{D}_H)$ satisfying \eqref{eq2}.  Then, the triple $(\Omega(\mathcal D_H), \partial_H, \mathscr{T}^H)$ forms an $n$-dimensional cycle over $\mathcal D_H$ with coefficients in $M$. Further, by observing that $\partial^0_{H}(f)=df$ for any $f \in Hom_{\mathcal{D}_H}(X,Y)$, we get \eqref{eq1}.

\smallskip
(1) $\Rightarrow$ (3). Let $(\mathcal{S}_H,\hat{\partial}_H, \hat{\mathscr T}^H)$ be an $n$-dimensional cycle over $\mathcal D_H$ with coefficients in $M$ and $\rho:\mathcal{D}_H \longrightarrow \mathcal{S}^0_H$ be an $H$-linear semifunctor satisfying 
$$\phi(m \otimes f^0 \otimes \ldots \otimes f^n)=\hat{\mathscr T}^H_{X_0}\big(m \otimes \rho(f^0)\hat{\partial}_H^0\left(\rho(f^1)\right) \ldots \hat{\partial}_H^0\left(\rho(f^n)\right)\big)$$
for any $m \in M$ and $f^0 \otimes \ldots \otimes f^n \in Hom_{\mathcal{D}_H}(X_1,X_0) \otimes Hom_{\mathcal{D}_H}(X_2,X_1) \otimes \ldots \otimes Hom_{\mathcal{D}_H}(X_0,X_n)$.
For simplicity of  notation, we will drop the functor $\rho$.
To show that $\phi$ is an $n$-cocycle, it suffices to check that (see \eqref{3.3y})
$$b(\phi)=0\quad \text{and} \quad (1-\lambda)(\phi)=0$$ where $b=\sum\limits_{i=0}^{n+1}(-1)^i\delta_i$ and $\lambda=(-1)^n\tau_n$. For any $p^0 \otimes \ldots \otimes p^{n+1} \in Hom_{\mathcal{D}_H}(X_1,X_0) \otimes Hom_{\mathcal{D}_H}(X_2,X_1) \otimes \ldots \otimes Hom_{\mathcal{D}_H}(X_0,X_{n+1})$, we have
\begin{align*}
&\sum\limits_{i=0}^{n+1}(-1)^i\delta_i(\phi)(m \otimes p^0 \otimes \ldots \otimes p^{n+1})\\
&= \sum\limits_{i=0}^{n}(-1)^i\phi(m \otimes p^0 \otimes \ldots \otimes p^ip^{i+1} \otimes \ldots \otimes p^{n+1})~ + (-1)^{n+1}\phi\big(m_{(0)} \otimes \big(S^{-1}(m_{(-1)})p^{n+1}\big)p^0 \otimes p^1 \otimes \ldots \otimes p^{n}\big)\\
&= \hat{\mathscr T}^H_{X_0}\big(m \otimes p^0p^1\hat{\partial}_H^0(p^2) \ldots \hat{\partial}_H^0(p^{n+1})\big) ~+ \sum\limits_{i=1}^{n}(-1)^i \hat{\mathscr T}^H_{X_0}\big(m \otimes p^0\hat{\partial}_H^0(p^1) \ldots \hat{\partial}_H^0(p^ip^{i+1})\ldots \hat{\partial}_H^0(p^{n+1})\big)~ +\\
& \quad (-1)^{n+1}\hat{\mathscr T}^H_{X_{n+1}}\big(m_{(0)} \otimes \big(S^{-1}(m_{(-1)})p^{n+1}\big)p^0 \hat{\partial}_H^0(p^1)\ldots \otimes \hat{\partial}_H^0(p^n)\big)
\end{align*}
Now using the equality $\hat{\partial}_H^0(fg)=\hat{\partial}_H^0(f)g+f\hat{\partial}_H^0(g)$ for any $f$ and $g$ of degree $0$, we have
 \begin{align*}
&\big(p^0\hat{\partial}_H^0(p^1) \ldots \hat{\partial}_H^0(p^n)\big)p^{n+1}\\
&=\sum\limits_{i=1}^n (-1)^{n-i} p^0\hat{\partial}_H^0(p^1) \ldots \hat{\partial}_H^0(p^ip^{i+1}) \ldots \hat{\partial}_H^0(p^{n+1}) + (-1)^n p^0p^1\hat{\partial}_H^0(p^2) \ldots \hat{\partial}_H^0(p^{n+1})
\end{align*}
Thus, using the condition in \eqref{gt2},  we obtain
\begin{align*}
&\sum\limits_{i=0}^{n+1}(-1)^i\delta_i(\phi)(m \otimes p^0 \otimes \ldots \otimes p^{n+1})\\
&= (-1)^n \hat{\mathscr T}^H_{X_0}\big(m \otimes \big( p^0\hat{\partial}_H^0(p^1) \ldots \hat{\partial}_H^0(p^n)\big)p^{n+1}\big) + (-1)^{n+1}\hat{\mathscr T}^H_{X_{n+1}}\big(m_{(0)} \otimes \big(S^{-1}(m_{(-1)})p^{n+1}\big)p^0 \hat{\partial}_H^0(p^1)\ldots  \hat{\partial}_H^0(p^n)\big)=0
\end{align*}

Next, using \eqref{gt1}, \eqref{gt2}, and the $H$-linearity of $\hat{\partial}_H$,  we have
\begin{align*}
&\big(\left(1-(-1)^n\tau_n\right)\phi\big)(m \otimes f^0 \otimes \ldots \otimes f^{n})\\
&=\phi(m \otimes f^0 \otimes \ldots \otimes f^{n})- (-1)^n\phi\big(m_{(0)} \otimes S^{-1}(m_{(-1)})f^n \otimes f^0 \otimes \ldots \otimes f^{n-1}\big)\\
&=\hat{\mathscr T}^H_{X_0}(m \otimes f^0\hat{\partial}_H^0(f^1) \ldots \hat{\partial}_H^0(f^n))-(-1)^n \hat{\mathscr T}^H_{X_n}\big(m_{(0)} \otimes \big(S^{-1}(m_{(-1)})f^n\big)\hat{\partial}_H^0(f^0)\hat{\partial}_H^0(f^1) \ldots \hat{\partial}_H^0(f^{n-1})\big)\\
&= (-1)^{n-1} \hat{\mathscr T}^H_{X_n}\big (m_{(0)} \otimes \big(S^{-1}(m_{(-1)})\hat{\partial}_H^0(f^n)\big)f^0\hat{\partial}_H^0(f^1) \ldots \hat{\partial}^0_H(f^{n-1}) \big)+\\
& \quad (-1)^{n-1} \hat{\mathscr T}^H_{X_n}\big(m_{(0)} \otimes \big(S^{-1}(m_{(-1)})f^n\big)\hat{\partial}^0_H(f^0) \hat{\partial}_H^0(f^1)\ldots \hat{\partial}^0_H(f^{n-1})\big)\\
&= (-1)^{n-1} \hat{\mathscr T}^H_{X_n}\big(m_{(0)} \otimes \hat{\partial}_H^{n-1}\big((S^{-1}(m_{(-1)})f^n)f^0\hat{\partial}_H^0(f^1) \ldots \hat{\partial}_H^0(f^{n-1})\big)\big)=0
\end{align*}

(3) $\Rightarrow$ (2). Let $\phi \in Z^n_H(\mathcal{D}_H,M)$. For each $X \in Ob(\Omega(\mathcal{D}_H))$, we define an $H$-linear map
$\mathscr{T}^H_X:M \otimes Hom^n_{\Omega(\mathcal{D}_H)}(X,X) \longrightarrow k$ given by
$$\mathscr{T}^H_X(m \otimes (f^0+\mu)df^1\ldots df^{n}):= \phi(m \otimes f^0 \otimes \ldots \otimes f^{n})$$
for $f^0 \otimes \ldots \otimes f^{n} \in Hom_{\mathcal{D}_H}(X_1,X) \otimes Hom_{\mathcal{D}_H}(X_2,X_1) \otimes \ldots \otimes Hom_{\mathcal{D}_H}(X,X_{n})$. We now verify that the collection $\{\mathscr{T}^n_X:M \otimes Hom^n_{\Omega(\mathcal{D}_H)}(X,X) \longrightarrow k\}_{X \in Ob(\Omega(\mathcal{D}_H))}$ is a closed graded $(H,M)$-trace on $(\Omega(\mathcal{D}_H),\partial_H)$. For any $(p^0+\mu)dp^1\ldots dp^{n-1} \in Hom^{n-1}_{\Omega(\mathcal{D}_H)}(X,X)$, we have
\begin{align*}
\mathscr{T}_X^{H}\big(m \otimes \partial_H^{n-1}((p^0+\mu)dp^1\ldots dp^{n-1})\big)&=\mathscr{T}_X^{H}\big(m \otimes 1dp^0dp^1\ldots dp^{n-1}\big)=\phi(m \otimes 0 \otimes p^0 \otimes \ldots \otimes p^{n-1})=0
\end{align*}
This proves the condition in \eqref{gt1}. Using \eqref{new3.2}, it is also clear that $\{\mathscr{T}^n_X:M \otimes Hom^n_{\Omega(\mathcal{D}_H)}(X,X) \longrightarrow k\}_{X \in Ob(\Omega(\mathcal{D}_H))}$ satisfies condition \eqref{gt0}. Finally, for any $g'=(g^0+\mu')dg^1\ldots dg^{r} \in Hom^r_{\Omega(\mathcal{D}_H)}(Y,X)$ and $g=(g^{r+1}+\mu)dg^{r+2}\ldots dg^{n+1} \in Hom^{n-r}_{\Omega(\mathcal{D}_H)}(X,Y)$, we have
\begin{align*}
&\mathscr{T}_X^{H}\big(m \otimes g'g\big)\\&=\sum\limits_{j=1}^r (-1)^{r-j}~ \mathscr{T}_X^{H}\big(m \otimes (g^0+\mu')dg^1 \ldots d(g^jg^{j+1}) \ldots dg^{n+1}\big) + (-1)^r~ \mathscr{T}_X^{H}\big(m \otimes (g^0+\mu')g^1dg^2 \ldots dg^{n+1}\big)\\
&\textrm{ }+\mathscr{T}_X^{H}\big(m \otimes \mu(g^0+\mu')dg^1 \ldots dg^rdg^{r+2}\ldots dg^{n+1}\big) \\
&= \sum\limits_{j=1}^r (-1)^{r-j} \phi(m \otimes g^0 \otimes \ldots \otimes g^jg^{j+1} \otimes \ldots \otimes g^{n+1}) + (-1)^r~ \phi(m \otimes g^0g^1 \otimes g^2 \otimes \ldots \otimes g^{n+1})\\
& \textrm{ }+ (-1)^r~ \mu'\phi(m \otimes g^1 \otimes g^2 \otimes \ldots \otimes g^{n+1})+\mu \phi(m \otimes g^0\otimes g^1 \otimes...\otimes g^r\otimes g^{r+2} \otimes \ldots \otimes g^{n+1})\\
&= \sum\limits_{j=0}^r (-1)^{r+j} \phi(m \otimes g^0 \otimes \ldots \otimes g^jg^{j+1} \otimes \ldots \otimes g^{n+1})+ (-1)^r~ \mu'\phi(m \otimes g^1 \otimes g^2 \otimes \ldots \otimes g^{n+1})\\ &\textrm{ }+\mu \phi(m \otimes g^0\otimes g^1 \otimes...\otimes g^r\otimes g^{r+2} \otimes \ldots \otimes g^{n+1})
\end{align*}

On the other hand, we have
\begin{equation*}
\begin{array}{ll}
&(-1)^{r(n-r)}~\mathscr{T}_Y^{H}\Big(m_{(0)} \otimes \big(S^{-1}(m_{(-1)})g\big)g'\Big)\\
&=(-1)^{r(n-r)}~\mathscr{T}_Y^{H}\Big(m_{(0)} \otimes \left([S^{-1}\left((m_{(-1)})_{n-r+1}\right)(g^{r+1}+\mu)][d\left(S^{-1}\left((m_{(-1)})_{n-r}\right)g^{r+2}\right)] \ldots [d\left(S^{-1}\left((m_{(-1)})_{1}\right)g^{n+1}\right)]\right)\circ\\
&\qquad((g^0+\mu')dg^1\ldots dg^{r}) \Big)\\
&=(-1)^{r(n-r)} \sum\limits_{j=r+2}^{n} (-1)^{n-j+1} ~\mathscr{T}_Y^{H}\Big(m_{(0)} \otimes [S^{-1}\left((m_{(-1)})_{n-r}\right)(g^{r+1}+\mu)]\ldots d\big[\big(S^{-1}((m_{(-1)})_{n-j+1})(g^{j}g^{j+1})\big]\ldots dg^r)+\\
&\quad (-1)^{r(n-r)}~\mathscr{T}_Y^{H}\Big(m_{(0)} \otimes [S^{-1}\left((m_{(-1)})_{n-r+1}\right)(g^{r+1}+\mu)]\ldots d[\big(S^{-1}((m_{(-1)})_1)g^{n+1}\big)g^{0}] \ldots dg^r\Big)+\\
& \quad (-1)^{r(n-r)} (-1)^{n-r} \mathscr{T}_Y^{H}\Big(m_{(0)} \otimes \left([S^{-1}\left((m_{(-1)})_{n-r}\right)((g^{r+1}+\mu)g^{r+2})]\right) \ldots [d\left(S^{-1}\left((m_{(-1)})_{1}\right)g^{n+1}\right)]
(dg^0dg^1\ldots dg^{r}) \Big)\\
&\textrm{ }+(-1)^{r(n-r)}\mu'\mathscr{T}_Y^{H}\Big(m_{(0)} \otimes \left. [S^{-1}\left((m_{(-1)})_{n-r+1}\right)(g^{r+1}+\mu)][d\left(S^{-1}\left((m_{(-1)})_{n-r}\right)g^{r+2}\right)] \ldots [d\left(S^{-1}\left((m_{(-1)})_{1}\right)g^{n+1}\right)]\right.\\
&\qquad dg^1\ldots dg^{r} \Big)\\
&=(-1)^{r(n-r)} \sum\limits_{j=r+2}^{n}(-1)^{n-j+1} ~\phi\Big(m_{(0)} \otimes S^{-1}\left((m_{(-1)})_{n-r}\right)g^{r+1} \otimes \ldots \otimes \big(S^{-1}((m_{(-1)})_{n-j+1})(g^{j}g^{j+1}) \otimes \ldots \otimes g^r)+\\
&\quad (-1)^{r(n-r)}~\phi\Big(m_{(0)} \otimes S^{-1}\left((m_{(-1)})_{n-r+1}\right)g^{r+1} \otimes \ldots \otimes \big(S^{-1}((m_{(-1)})_1)g^{n+1}\big)g^{0} \otimes \ldots \otimes g^r\Big)+\\
& \quad (-1)^{r(n-r)} (-1)^{n-r} \phi\Big(m_{(0)} \otimes S^{-1}\left((m_{(-1)})_{n-r}\right)(g^{r+1}g^{r+2}) \otimes \ldots \otimes \left(S^{-1}\left((m_{(-1)})_{1}\right)g^{n+1}\right) \otimes  g^0 \otimes g^1 \otimes \ldots \otimes g^{r} \Big)\\
& \quad (-1)^{r(n-r)} (-1)^{n-r} \mu \phi\Big(m_{(0)} \otimes  S^{-1}\left((m_{(-1)})_{n-r}\right)g^{r+2} \otimes \ldots \otimes \left(S^{-1}\left((m_{(-1)})_{1}\right)g^{n+1}\right) \otimes  g^0 \otimes g^1 \otimes \ldots \otimes g^{r} \Big)\\
&\textrm{ }+(-1)^{r(n-r)}\mu'\phi\Big(m_{(0)} \otimes \left. S^{-1}\left((m_{(-1)})_{n-r+1}\right)g^{r+1}\otimes S^{-1}\left((m_{(-1)})_{n-r}\right)g^{r+2}\otimes \ldots \otimes (S^{-1}\left((m_{(-1)})_{1}\right)g^{n+1}\otimes \right. g^1\otimes \ldots \otimes g^{r} \Big) \\
\end{array}
\end{equation*}

Using repeatedly the fact that  $\phi=(-1)^n\tau_n\phi$, we get

\begin{align*}
&(-1)^{r(n-r)}~\mathscr T_Y^{n}\Big(m_{(0)} \otimes \big(S^{-1}(m_{(-1)})g\big)g'\Big)\\
&=-\sum\limits_{j=r+1}^{n} (-1)^{r+j} \phi(m \otimes g^0 \otimes \ldots \otimes g^jg^{j+1} \otimes \ldots \otimes g^{n+1})-(-1)^{n+r+1}\phi\Big(m_{(0)}  \otimes \big(S^{-1}(m_{(-1)})g^{n+1}\big)g^0 \otimes g^1 \otimes \ldots \otimes g^n\Big)\\
&\textrm{ } + (-1)^r~ \mu'\phi(m \otimes g^1 \otimes g^2 \otimes \ldots \otimes g^{n+1}) +\mu \phi(m \otimes g^0\otimes g^1 \otimes...\otimes g^r\otimes g^{r+2} \otimes \ldots \otimes g^{n+1})
\end{align*}

\smallskip
The condition \eqref{gt2} now follows using the fact that $b(\phi)=0$. This proves the result.
\end{proof}

\begin{remark}\label{rem4.11}
From the statement and proof of Theorem \ref{charcycl}, it is clear that there is a one to one correspondence between $n$-dimensional closed graded $(H,M)$-traces on $\Omega(\mathcal{D}_H)$ and $Z^n_H(\mathcal{D}_H,M)$.
\end{remark}

\section{Linearization by matrices and Hopf-cyclic cohomology}\label{Morita}
 In the previous section, we described the spaces $Z^\bullet_H(\mathcal{D}_H,M)$. The next aim is to find a characterization of  $B^\bullet_H(\mathcal{D}_H,M)$  which will be done in several steps. For this, we will show in this section that the Hopf-cyclic cohomology of an $H$-category  $\mathcal D_H$ is the same as that of its linearization $\mathcal{D}_H \otimes M_r(k)$ by the algebra of $r\times r$-matrices. We observe that $\mathcal{D}_H \otimes M_r(k)$ is also a left $H$-category. We denote by $\overline{Cat}_H$ the category whose objects are left $H$-categories and whose morphisms are $H$-linear semifunctors.

\smallskip
We denote by $Vect_k$ the category of all $k$-vector spaces and by $H\text{-}Mod$ the category of all left $H$-modules. Let
$Hom_H(-,k):H\text{-}Mod \longrightarrow Vect_k$ be the functor that takes $N \mapsto Hom_H(N,k)$. 

\smallskip
We fix $r\geq 1$. For $1\leq i,j\leq r$ and $\alpha\in k$, we let $E_{ij}(\alpha)$ denote the elementary matrix in $M_r(k)$ having $\alpha$
at $(i,j)$-th position and $0$ everywhere else. We will often use $E_{ij}$ for $E_{ij}(1)$. For each $1\leq p\leq r$, we have an inclusion $inc_p:\mathcal{D}_H \longrightarrow \mathcal{D}_H \otimes M_r(k)$ in $\overline{Cat}_H$ which fixes the objects and 
 $inc_p(f)=f\otimes E_{pp}=f \otimes E_{pp}(1)$ for any morphism $f \in \mathcal{D}_H$. 

\smallskip
For any right-left SAYD-module $M$, the inclusion $inc_p:\mathcal{D}_H \longrightarrow \mathcal{D}_H  \otimes M_r(k)$ induces an inclusion map 
$({inc}_p,M):M \otimes  CN_n(\mathcal{D}_H)  \longrightarrow M \otimes CN_n\left(\mathcal{D}_H  \otimes M_r(k)\right) $ which takes 
$
m \otimes f^0 \otimes \ldots \otimes f^n \mapsto m \otimes (f^0 \otimes E_{pp}) \otimes \ldots \otimes (f^n \otimes E_{pp})
$. This induces a morphism of Hochschild complexes
$C_\bullet(inc_p,M)^{hoc}:C_\bullet(\mathcal{D}_H,M)^{hoc} \longrightarrow C_\bullet\left(\mathcal{D}_H \otimes M_r(k),M\right)^{hoc}$.
Applying the functor $Hom_H(-,k)$, we obtain morphisms of Hochschild complexes
$
C^{\bullet}_H(inc_1,M)^{hoc}:C^{\bullet}_H\left(\mathcal{D}_H \otimes M_r(k),M\right)^{hoc} \longrightarrow C^{\bullet}_H(\mathcal{D}_H,M)^{hoc}
$.

\smallskip
We also have the induced morphism of double complexes computing cyclic homology
$C_{\bullet\bullet}(inc_p,M)^{cy}:C_{\bullet\bullet}(\mathcal{D}_H,M)^{cy} \longrightarrow C_{\bullet\bullet}\left(\mathcal{D}_H \otimes M_r(k),M\right)^{cy}$.
Applying the functor $Hom_H(-,k)$, we obtain a morphism of double complexes computing cyclic cohomology
$
C^{\bullet\bullet}_H(inc_1,M)^{cy}:C^{\bullet\bullet}_H(\mathcal{D}_H\otimes M_r(k),M)^{cy} \longrightarrow C_H^{\bullet\bullet}\left(\mathcal{D}_H ,M\right)^{cy}
$.

\smallskip
For each $n \geq 0$, there is  an $H$-linear trace map 
$tr^M:M \otimes CN_n\left(\mathcal{D}_H \otimes M_r(k)\right) \longrightarrow M \otimes CN_n(\mathcal{D}_H) $ given by 
\begin{equation}\label{tracemap}
tr^M\left(m \otimes (f^0 \otimes B^0) \otimes \ldots \otimes (f^n \otimes B^n)\right):=(m \otimes f^0 \otimes \ldots \otimes f^n)\text{trace}(B^0\ldots B^n)
\end{equation}
for any $m \in M$ and $(f^0 \otimes B^0) \otimes \ldots \otimes (f^n \otimes B^n) \in CN_n\left(\mathcal{D}_H \otimes M_r(k)\right)$.
It may be verified easily that the trace map as in \eqref{tracemap} defines a morphism $C_\bullet(tr^M):C_\bullet\left(\mathcal{D}_H \otimes M_r(k),M\right) \longrightarrow C_\bullet(\mathcal{D}_H,M)$ of para-cyclic modules.
In particular, we have an induced morphism between underlying Hochschild complexes
\begin{equation*}
C_\bullet({tr^M})^{hoc}:C_\bullet\left(\mathcal{D}_H \otimes M_r(k),M\right)^{hoc} \longrightarrow  C_\bullet(\mathcal{D}_H,M)^{hoc}
\end{equation*}

\begin{prop}\label{homotopy}
The maps $C_\bullet({inc}_1,M)^{hoc}$ and $C_\bullet(tr^M)^{hoc}$ are homotopy inverses of each other.
\end{prop}
\begin{proof}
It may be easily verified that $C_\bullet(tr^M)^{hoc} \circ C_\bullet(inc_1,M)^{hoc}=id$. 
To show that $C_\bullet(inc_1,M)^{hoc} \circ C_\bullet(tr^M)^{hoc} \sim id$, we  define  $k$-linear maps
$\{\hbar_i: C_{n}\left(\mathcal{D}_H \otimes M_r(k),M\right) \longrightarrow C_{n+1}\left(\mathcal{D}_H \otimes M_r(k),M\right)\}_{0\leq i\leq n}$  by setting:
\begin{equation*}
\begin{array}{ll}
\hbar_i\left(m \otimes (f^0 \otimes B^0) \otimes \ldots \otimes (f^n \otimes B^n)\right):=& m \otimes \sum_{1 \leq j,k,l,\ldots,p,q\leq r}  (f^0 \otimes E_{j1}(B^0_{jk})) \otimes (f^1 \otimes E_{11}(B^1_{kl})) \otimes \ldots\\ & \quad  \otimes (f^i \otimes E_{11}(B^i_{pq})) \otimes (id_{X_{i+1}} \otimes E_{1q}(1)) \otimes (f^{i+1} \otimes B^{i+1}) \otimes \ldots\\
& \quad \ldots \otimes (f^n \otimes B^n)
\end{array}
\end{equation*}
for $0\leq i<n$ and 
\begin{equation*}
\begin{array}{ll}
\hbar_n\left(m \otimes (f^0 \otimes B^0) \otimes \ldots \otimes (f^n \otimes B^n)\right):= & m \otimes \sum_{1 \leq j,k,m,\ldots,p,q \leq r}  (f^0 \otimes E_{j1}(B^0_{jk})) \otimes (f^1 \otimes E_{11}(B^1_{km})) \otimes \ldots\\ & \quad \ldots \otimes (f^n \otimes E_{11}(B^n_{pq})) \otimes (id_{X_0} \otimes E_{1q}(1))
\end{array}
\end{equation*}
We now verify that $\hbar^n:=\sum_{i=0}^n (-1)^i\hbar_i$ is a pre-simplicial homotopy (see, for instance, \cite[$\S$ 1.0.8]{Loday}) between $C_\bullet(inc_1,M)^{hoc} \circ C_\bullet(tr^M)^{hoc}$ and $id_{C_\bullet\left(\mathcal{D}_H \otimes M_r(k),M\right)}$. For this, we need to verify the following identities:
\begin{equation}\label{relations}
\begin{array}{lll}
d_i\hbar_{i'}=\hbar_{i'-1}d_i & \text{for}~ i<i'\\
d_i\hbar_i=d_i\hbar_{i-1} & \text{for}~ 0< i \leq n\\
d_i\hbar_{i'}= \hbar_{i'}d_{i-1} & \text{for}~ i>i'+1\\
d_0\hbar_0= id_{C_\bullet\left(\mathcal{D}_H \otimes M_r(k),M\right)^{hoc}}& \text{and}~ d_{n+1}\hbar_n=C_\bullet(inc_1,M)^{hoc} \circ C_\bullet(tr^M)^{hoc}
\end{array}
\end{equation}
where $d_i:C_{n+1}\left(\mathcal{D}_H \otimes M_r(k),M\right) \longrightarrow C_{n}\left(\mathcal{D}_H \otimes M_r(k),M\right)$, $0 \leq i \leq n+1$ are the face maps.  We only verify the last one in \eqref{relations} because the others follow similarly. 
Using the fact that $E_{1q}(1)E_{j1}(B_{jk})=0$ unless $q=j$, we have 
\begin{equation*}
\begin{array}{ll}
&d_{n+1}\hbar_n\left(m \otimes (f^0 \otimes B^0) \otimes \ldots \otimes (f^n \otimes B^n)\right)\\
& \quad =d_{n+1}\big(m \otimes \sum_{1 \leq j,k,l,\ldots,p,q \leq r}  (f^0 \otimes E_{j1}(B^0_{jk})) \otimes (f^1 \otimes E_{11}(B^1_{kl})) \otimes \ldots \otimes (f^n \otimes E_{11}(B^n_{pq})) \otimes 
 (id_{X_{0}} \otimes E_{1q}(1))\big)\\
& \quad =m_{(0)} \otimes \sum_{1 \leq j,k,l,\ldots,p,q \leq r} \left(S^{-1}(m_{(-1)})(id_{X_{0}} \otimes E_{1q}(1))\right)(f^0 \otimes E_{j1}(B^0_{jk}))\otimes (f^1 \otimes E_{11}(B^1_{kl})) \otimes \ldots \\
&\qquad \ldots \otimes (f^n \otimes E_{11}(B^n_{pq}))\\
& \quad= m \otimes \sum_{1 \leq j,k,l,\ldots,p,q \leq r}\left(f^0 \otimes E_{1q}(1)E_{j1}(B^0_{jk})\right) \otimes (f^1 \otimes E_{11}(B^1_{kl})) \otimes \ldots \otimes (f^n \otimes E_{11}(B^n_{pq}))\\
& \quad= m \otimes \sum_{1 \leq j,k,l,\ldots,p\leq r}\left(f^0 \otimes E_{1j}(1)E_{j1}(B^0_{jk})\right) \otimes (f^1 \otimes E_{11}(B^1_{kl})) \otimes \ldots \otimes (f^n \otimes E_{11}(B^n_{pj}))\\
& \quad= m \otimes \sum_{1 \leq j,k,l,\ldots,p\leq r}(f^0 \otimes E_{11}(B^0_{jk})) \otimes (f^1 \otimes E_{11}(B^1_{kl})) \otimes \ldots \otimes (f^n \otimes E_{11}(B^n_{pj}))\\
& \quad = \left(m \otimes (f^0 \otimes E_{11}) \otimes \ldots \otimes (f^n \otimes E_{11})\right) \sum_{1 \leq j,k,l,\ldots,p\leq r}(B^0_{jk}B^1_{kl}\ldots B^n_{pj})\\
& \quad = \left(m \otimes (f^0 \otimes E_{11}) \otimes \ldots \otimes (f^n \otimes E_{11})\right) \sum_{1 \leq j \leq r} 
(B^0B^1 \ldots B^n)_{jj}\\
& \quad = \left(m \otimes (f^0 \otimes E_{11}) \otimes \ldots \otimes (f^n \otimes E_{11})\right)trace(B^0B^1 \ldots B^n)\\
& \quad = \left(C_\bullet(inc_1,M)^{hoc} \circ C_\bullet(tr^M)^{hoc}\right)\left(m \otimes (f^0 \otimes B^0) \otimes \ldots \otimes (f^n \otimes B^n)\right)
\end{array}
\end{equation*} This proves the result.
\end{proof}

\begin{prop}\label{moritainvcyc} Let $\mathcal{D}_H$ be a left $H$-category and $M$ be a right-left SAYD module. Then,

\smallskip
(1) The morphisms
\begin{equation*}
\begin{array}{ll}
HC^{\bullet}_H(inc_1,M)^{hoc}:HC^{\bullet}_H\left(\mathcal{D}_H \otimes M_r(k),M\right)^{hoc} \longrightarrow HC^{\bullet}_H(\mathcal{D}_H,M)^{hoc}\\
HC^{\bullet}_H(tr^M)^{hoc}:HC^{\bullet}_H(\mathcal{D}_H,M)^{hoc} \longrightarrow HC^{\bullet}_H\left(\mathcal{D}_H \otimes M_r(k),M\right)^{hoc}
\end{array}
\end{equation*}
induced by $C^{\bullet}_H(inc_1,M)^{hoc}$ and $C^{\bullet}_H(tr^M)^{hoc}$ are mutually inverse isomorphisms of Hochschild cohomologies.

\smallskip
(2)  We have  isomorphisms 
 \begin{equation*} \xy (10,0)*{HC^\bullet_H(\mathcal{D}_H,M)}; (40,3)*{HC^{\bullet}_H(tr^M)}; (40,-3.3)*{HC^{\bullet}_H(inc_1,M)};  {\ar
(25,1)*{}; (60,1)*{}}; {\ar (60,-1)*{};(25,-1)*{}};
\endxy \quad HC^\bullet_H\left(\mathcal{D}_H \otimes M_r(k),M\right)
\end{equation*}
\end{prop}
\begin{proof}
(1) By Proposition \ref{homotopy}, we know that $C_\bullet(tr^M)^{hoc} \circ C_\bullet(inc_1,M)^{hoc}=id_{C_\bullet(\mathcal{D}_H,M)^{hoc}}$ and
$C_\bullet(inc_1,M)^{hoc} \circ C_\bullet(tr^M)^{hoc} \sim id_{C_\bullet\left(\mathcal{D}_H \otimes M_r(k),M\right)^{hoc}}$. Thus, applying the functor $Hom_H(-,k)$, we obtain 
\begin{equation*}
C^{\bullet}_H(inc_1,M)^{hoc} \circ C^{\bullet}_H(tr^M)^{hoc}=id_{C^{\bullet}_H(\mathcal{D}_H,M)^{hoc}} \qquad 
C^{\bullet }_H(tr^M)^{hoc} \circ C^{\bullet}_H(inc_1,M)^{hoc} \sim id_{C^{\bullet}_H(\mathcal{D}_H \otimes M_r(k),M)^{hoc}}
\end{equation*}
Therefore, $C^{\bullet}_H(inc_1,M)^{hoc}$ and $C^{\bullet}_H(tr^M)^{hoc}$ are homotopy inverses of each other. 

\smallskip
(2) This follows immediately from (1) and the Hochschild to cyclic spectral sequence.
\end{proof}

\begin{corollary}\label{5.8a}
For an  $n$-cocycle $\phi \in Z^n_H(\mathcal{D}_H,M)$, the $n$-cocycle $\tilde{\phi}=Hom_H(tr^M,k)(\phi) =\phi\circ tr^M\in Z^n_H(\mathcal{D}_H \otimes M_r(k),M)$  may be described as follows
\begin{equation*}
\tilde{\phi}\left(m \otimes (f^0 \otimes B^0) \otimes \ldots \otimes (f^n \otimes B^n)\right)=\phi(m \otimes f^0 \otimes \ldots \otimes f^n)trace(B^0\ldots B^n)
\end{equation*}
\end{corollary}

\section{Vanishing cycles on an $H$-category and coboundaries}
From now onwards, we will always assume that $k=\mathbb C$. In this section, we will describe the spaces $B^\bullet_H(\mathcal{D}_H,M)$. 
We will then use the formalism of categorified cycles and vanishing cycles developed in this paper to obtain a product on Hopf cyclic cohomologies of $H$-categories. We begin by recalling the notion of an inner automorphism of a category.

\begin{definition}\label{indefr} [see \cite{Sch}, p 24]
Let $\mathcal{D}_H$ be a left $H$-category. An automorphism $\Phi \in Hom_{Cat_H}(\mathcal{D}_H,\mathcal{D}_H)$ is said to be inner if $\Phi$ is isomorphic to the identity functor $id_{\mathcal{D}_H}$. In particular, there exist isomorphisms $\{\eta(X):X\longrightarrow \Phi(X)\}_{X\in Ob(\mathcal 
D_H)}$ such that   $\Phi(f)=\eta(Y)\circ f \circ (\eta(X))^{-1}$ for any $f \in Hom_{\mathcal{D}_H}(X,Y)$.
\end{definition}

We now set
\begin{equation}
\mathbb G(\mathcal D_H):=\underset{X\in Ob(\mathcal D_H)}{\prod}\textrm{ }Aut_{\mathcal D_H}(X)
\end{equation}
By definition, an element $\eta\in \mathbb G(\mathcal D_H)$ corresponds to a family of automorphisms $\{\eta(X):X\longrightarrow X\}_{X\in Ob(\mathcal 
D_H)}$. We now set
\begin{equation}
\mathbb U_H(\mathcal D_H):=\{\mbox{$\eta\in \mathbb G(\mathcal D_H)$ $\vert$ $h(\eta(X))=\varepsilon(h)\eta(X)$ for every $h\in H$ and $X\in Ob(\mathcal D_H)$}\}
\end{equation}

\begin{lemma} $\mathbb U_H(\mathcal D_H)$ is a subgroup of $\mathbb G(\mathcal D_H)$. 
\end{lemma}
\begin{proof}
The element $\mathbf{e}=\underset{X\in Ob(\mathcal D_H)}{\prod} id_X$ is the identity of the group $\mathbb G(\mathcal D_H)$.
By definition of an $H$-category, we know that $h\cdot id_X=\varepsilon(h)\cdot id_X$ for each $X \in Ob(\mathcal{D}_H)$ and $h \in H$. Thus, $\mathbf{e} \in \mathbb U_H(\mathcal D_H)$. Now, suppose that $\eta, \eta' \in \mathbb U_H(\mathcal D_H)$. Then, for each $X \in Ob(\mathcal{D}_H)$ and $h \in H$,
\begin{equation*}
h\left((\eta \circ \eta')(X)\right)=h(\eta(X) \circ \eta'(X))=(h_1\eta(X)) \circ (h_2\eta'(X))=(\varepsilon(h_1)\eta(X)) \circ (\varepsilon(h_2)\eta'(X))=\varepsilon(h)(\eta(X) \circ \eta'(X))
\end{equation*}
Hence, $\eta \circ \eta' \in \mathbb U_H(\mathcal D_H)$.  Also, $\eta^{-1} \in \mathbb G(\mathcal D_H)$ corresponds to a family of morphisms $\{\eta^{-1}(X):=\eta(X)^{-1}:X \longrightarrow X\}_{X \in Ob(\mathcal{D}_H)}$. Then, for each $h \in H$ and $X \in Ob(\mathcal{D}_H)$,
\begin{equation*}
\varepsilon(h)id_X=h(\eta(X) \circ \eta^{-1}(X))=(\varepsilon(h_1)\eta(X)) \circ (h_2 \eta^{-1}(X))=\eta(X) \circ (h \eta^{-1}(X))
\end{equation*}
which gives $\varepsilon(h)\eta^{-1}(X)=h \eta^{-1}(X)$. Therefore, $\eta^{-1} \in  \mathbb U_H(\mathcal D_H)$.
\end{proof}

\begin{lemma}\label{innerautonew}
 Let $\mathcal{D}_H$ be a left $H$-category and let $\eta\in \mathbb U_H(\mathcal D_H)$.

\smallskip
(1) Consider $\Phi_\eta:\mathcal D_H\longrightarrow \mathcal D_H$ defined by 
\begin{equation*}\Phi_\eta(X)=X\qquad \Phi_\eta(f):=\eta(Y)\circ f\circ\eta(X)^{-1}
\end{equation*} for every $X\in Ob(\mathcal D_H)$ and $f\in Hom_{\mathcal D_H}(X,Y)$. Then, $\Phi_\eta: \mathcal D_H\longrightarrow \mathcal D_H$
is an  inner automorphism of $\mathcal{D}_H$.

\smallskip
(2) Consider $\tilde\Phi_\eta: \mathcal D_H\otimes M_2(k)\longrightarrow \mathcal D_H\otimes M_2(k)$ defined by
\begin{equation*}
\tilde\Phi_\eta(X)=X \qquad \tilde\Phi_\eta(f\otimes B)=(id_Y\otimes E_{11}+  \eta(Y)\otimes E_{22})\circ (f\otimes B)\circ (id_X\otimes E_{11}+  \eta(X)^{-1}\otimes E_{22})
\end{equation*}  for every $X\in Ob(\mathcal D_H\otimes M_2(k))=Ob(\mathcal D_H)$ and $f\otimes B \in Hom_{\mathcal D_H\otimes M_2(k)}(X,Y)$. Then,
$\tilde\Phi_\eta: \mathcal D_H\otimes M_2(k)\longrightarrow \mathcal D_H\otimes M_2(k)$ is an  inner automorphism.
\end{lemma}
\begin{proof}
{\it (1)}  Using the fact that $\eta, \eta^{-1} \in \mathbb U_H(\mathcal D_H)$, we have
\begin{equation*}
\begin{array}{ll}
h(\Phi_\eta(f))=(h_1\eta(Y)) \circ (h_2f) \circ (h_3\eta(X)^{-1})&=(\varepsilon(h_1)\eta(Y)) \circ (h_2f) \circ (\varepsilon(h_3)\eta(X)^{-1})\\
&=\eta(Y)  \circ (h_1f) \circ (\varepsilon(h_2)\eta(X)^{-1})\\
&=\eta(Y)\circ (hf)\circ\eta(X)^{-1}
\end{array}
\end{equation*}
for any $h \in H$ and $f\in Hom_{\mathcal D_H}(X,Y)$. By Definition \ref{indefr}, we now see that $\Phi_\eta$ is an inner automorphism.

\smallskip
{\it (2)}  Setting $\tilde{\eta}(X):X \longrightarrow X$ in $\mathcal D_H\otimes M_2(k)$ as $\tilde{\eta}(X)=id_X\otimes E_{11}+  \eta(X)\otimes E_{22}$, we see that
\begin{equation*}
\begin{array}{ll}
\tilde\Phi_\eta(f\otimes B)&=(id_Y\otimes E_{11}+  \eta(Y)\otimes E_{22})\circ (f\otimes B)\circ (id_X\otimes E_{11}+  \eta(X)^{-1}\otimes E_{22})\\
&=\tilde{\eta}(Y) \circ (f\otimes B)\circ \tilde{\eta}(X)^{-1}
\end{array}
\end{equation*}
for any  $f\otimes B \in Hom_{\mathcal D_H\otimes M_2(k)}(X,Y)$.
Considering the $H$-action on the category $\mathcal D_H\otimes M_2(k)$, we have
\begin{equation*}
\begin{array}{ll}
h\left(\tilde{\Phi}_\eta((f \otimes B)\right)&=h_1(id_Y\otimes E_{11}+  \eta(Y)\otimes E_{22})\circ h_2(f\otimes B)\circ h_3(id_X\otimes E_{11}+  \eta(X)^{-1}\otimes E_{22})\\
&=(h_1id_Y\otimes E_{11}+  h_1\eta(Y)\otimes E_{22})\circ h_2(f\otimes B)\circ (h_3id_X\otimes E_{11}+  h_3\eta(X)^{-1}\otimes E_{22})\\
&=\varepsilon(h_1)(id_Y\otimes E_{11}+  \eta(Y)\otimes E_{22})\circ h_2(f\otimes B)\circ \varepsilon(h_3)(id_X\otimes E_{11}+  \eta(X)^{-1}\otimes E_{22})\\
&=(id_Y\otimes E_{11}+  \eta(Y)\otimes E_{22})\circ h(f\otimes B)\circ (id_X\otimes E_{11}+  \eta(X)^{-1}\otimes E_{22})\\
&=\tilde{\Phi}_\eta(h(f \otimes B))
\end{array}
\end{equation*}
for any $h \in H$ and $f\otimes B \in Hom_{\mathcal D_H\otimes M_2(k)}(X,Y)$. By Definition \ref{indefr}, we now see that $\Phi_{\tilde\eta}$ is an inner automorphism.
\end{proof}

For  any $\eta\in \mathbb U_H(\mathcal D_H)$, we will always denote by $\Phi_\eta$ and $\tilde{\Phi}_\eta$ the inner automorphisms defined in Lemma \ref{innerautonew}.

\begin{lemma}\label{6.3a} Let $M$ be a right-left SAYD module over $H$. Then, 

\smallskip
(1) A semifunctor $\alpha \in Hom_{\overline{Cat}_H}(\mathcal{D}_H,\mathcal{D}_H')$ induces a morphism (for all $n\geq 0$)
\begin{equation*}
C^n_H(\alpha,M):C^n_H(\mathcal{D}_H',M)=Hom_H(M \otimes CN_n(\mathcal{D}_H'),k)\longrightarrow C^n_H(\mathcal{D}_H,M)=Hom_H(M \otimes CN_n(\mathcal{D}_H),k)
\end{equation*}  determined by
\begin{equation*}
C^{n}_H(\alpha,M)(\phi)(m \otimes f^0 \otimes \ldots \otimes f^n)=\phi\left(m \otimes \alpha(f^0) \otimes \ldots \otimes \alpha(f^n)\right)
\end{equation*}
for any $\phi \in C^n_H(\mathcal{D}_H',M)$, $m \in M$ and $f^0 \otimes \ldots \otimes f^n \in CN_n(\mathcal{D}_H)$. This leads to a morphism  $C^{\bullet\bullet}_H(\alpha,M)^{cy}:C^{\bullet\bullet}_H(\mathcal{D}_H',M)^{cy} \longrightarrow C^{\bullet\bullet}_H(\mathcal{D}_H,M)^{cy}$ of double complexes and induces  a functor $HC^\bullet_H(-,M):\overline{Cat}_H^{op} \longrightarrow Vect_k$.

\smallskip
(2) Let $\eta\in \mathbb U_H(\mathcal D_H)$. Then, $\Phi_\eta$  induces the identity map on $HC^\bullet_H(\mathcal{D}_H,M)$.
\end{lemma}
\begin{proof}
(1) Since $\phi$ and $\alpha$ are $H$-linear, the morphisms $C^n_H(\alpha,M)$ are well-defined and well behaved with respect to the maps appearing in the Hochschild
and cyclic complexes. The result follows.

\smallskip
(2) Let $\eta\in \mathbb U_H(\mathcal D_H)$ and $\Phi_\eta \in Hom_{Cat_H}(\mathcal{D}_H,\mathcal{D}_H)$ be the corresponding inner automorphism. By Proposition \ref{moritainvcyc}, the maps $HC^\bullet_H(inc_1,M)$ and $HC^\bullet_H(tr^M)$ are mutually inverse isomorphisms of Hopf-cyclic cohomology groups. Thus, we have
\begin{equation}\label{dd6.1}
HC^\bullet_H(inc_2,M) \circ \left(HC^\bullet_H(inc_1,M)\right)^{-1} =HC^\bullet_H(inc_2,M)  \circ HC^\bullet_H(tr^M)= HC^\bullet_H\left(tr^M \circ (inc_2,M)\right)=id
\end{equation} 
Further, we have the following commutative diagram in the category $\overline{Cat}_H$:
\begin{equation}\label{d6.1}
\begin{CD}
\mathcal{D}_H @>inc_1>>  \mathcal{D}_H \otimes M_2(k) @<inc_2<< \mathcal{D}_H\\
@Vid_{\mathcal{D}_H}VV        @VV \tilde\Phi_\eta V @VV \Phi_\eta V\\
\mathcal{D}_H    @> inc_1>>\mathcal{D}_H \otimes M_2(k)@<inc_2<< \mathcal{D}_H\\
\end{CD}
\end{equation}
Thus, by applying the functor $HC^\bullet_H(-,M)$ to the commutative diagram \eqref{d6.1} and using \eqref{dd6.1}, we obtain
\begin{align*}
HC^\bullet_H(\Phi_\eta,M)&=  \left(HC^\bullet_H(inc_2,M)\right)\circ HC^\bullet_H(inc_1,M)^{-1} \circ HC^\bullet_H(id_{\mathcal{D}_H},M)\circ \left(HC^\bullet_H(inc_1,M)\right) \circ HC^\bullet_H(inc_2,M)^{-1}\\
&=id_{HC^\bullet_H(\mathcal{D}_H,M)}
\end{align*} 
\end{proof}

\begin{prop}\label{vanishing}
Let $\mathcal{D}_H$ be a left $H$-category. Suppose that there is a semifunctor $\upsilon \in Hom_{\overline{Cat}_H}(\mathcal{D}_H,\mathcal{D}_H)$ and an $\eta \in \mathbb U_H\left(\mathcal{D}_H \otimes M_2(k)\right)$ such that 
\begin{itemize}
\item[(1)] $\upsilon(X)=X \qquad \forall X \in Ob(\mathcal{D}_H)$
\item[(2)] $\Phi_\eta(f \otimes E_{11}+\upsilon(f) \otimes E_{22})=\upsilon(f) \otimes E_{22}$
\end{itemize}
for all $f \in Hom_{\mathcal{D}_H}(X,Y)$ and  $X,Y \in Ob(\mathcal{D}_H)$. Then, $HC^\bullet_H(\mathcal{D}_H,M)=0$.
\end{prop}
\begin{proof}
Let $\alpha, \alpha' \in Hom_{\overline{Cat}_H}\left(\mathcal{D}_H, \mathcal{D}_H \otimes M_2(k)\right)$ be the semifunctors defined by
\begin{equation*}
\begin{array}{c}
\alpha(X):=X \qquad \alpha(f):=f \otimes E_{11} + \upsilon(f) \otimes E_{22}\\
\alpha'(X):=X \qquad \alpha'(f):=\upsilon(f) \otimes E_{22}
\end{array}
\end{equation*}
for all $X \in Ob(\mathcal{D}_H)$ and $f \in Hom_ {\mathcal{D}_H}(X,Y)$. Then, by assumption, $\alpha'=\Phi_\eta \circ \alpha$. Therefore, applying the functor $HC^\bullet_H(-,M)$ and using Lemma \ref{6.3a} (2), we get
\begin{equation}\label{cohomclass}
HC^\bullet_H(\alpha',M)=HC^\bullet_H(\alpha,M) \circ HC^\bullet_H(\Phi_\eta,M)=HC^\bullet_H(\alpha,M):HC^\bullet_H(\mathcal D_H\otimes M_2(k),M)\longrightarrow HC^\bullet_H(\mathcal D_H,M)
\end{equation}
Let $\phi \in Z^n_H(\mathcal{D}_H,M)$ and $\tilde{\phi}=Hom_H(tr^M,k)(\phi) =\phi\circ tr^M\in Z^n_H(\mathcal{D}_H \otimes M_2(k),M)$ as in Corollary \ref{5.8a}. Let $[\tilde{\phi}]$ denote the cohomology class of $\tilde{\phi}$. Then, by \eqref{cohomclass}, we have $HC^\bullet_H(\alpha,M)([\tilde{\phi}])=HC^\bullet_H(\alpha',M)([\tilde{\phi}])$, i.e.,
\begin{equation}
\tilde{\phi}\circ (id_M\otimes CN_n(\alpha)) + B^n_H(\mathcal{D}_H,M)=\tilde{\phi}\circ  (id_M\otimes CN_n(\alpha'))  + B^n_H(\mathcal{D}_H,M)
\end{equation}
so that $\tilde{\phi}\circ (id_M\otimes CN_n(\alpha)) - \tilde{\phi}\circ (id_M\otimes CN_n(\alpha')) \in B^n_H(\mathcal{D}_H,M)$.
Applying the definition of $\tilde{\phi}$, we now have
\begin{equation*}
\begin{array}{ll}
&(\tilde{\phi}\circ (id_M\otimes CN_n(\alpha))) (m \otimes f^0 \otimes \ldots \otimes f^n)\\
&\quad =\tilde{\phi}\left(m \otimes \alpha(f^0) \otimes \ldots \otimes \alpha(f^n)\right)\\
& \quad = \tilde{\phi}\left(m \otimes (f^0 \otimes E_{11} + \upsilon(f^0) \otimes E_{22}) \otimes \ldots \otimes (f^n \otimes E_{11} + \upsilon(f^n) \otimes E_{22})\right)\\
& \quad = \phi(m \otimes f^0 \otimes \ldots \otimes f^n)+\phi(m \otimes \upsilon(f^0) \otimes \ldots \otimes \upsilon(f^n))
\end{array}
\end{equation*}
Similarly, $(\tilde{\phi}\circ (id_M\otimes CN_n(\alpha'))) (m \otimes f^0 \otimes \ldots \otimes f^n)=\phi(m \otimes \upsilon(f^0) \otimes \ldots \otimes \upsilon(f^n))$. Thus, $\phi=\tilde{\phi}\circ (id_M\otimes CN_n(\alpha)) - \tilde{\phi}\circ (id_M\otimes CN_n(\alpha'))\in B^n_H(\mathcal{D}_H,M)$. This proves the result.
\end{proof}

\begin{definition}\label{def6.5}
Let $(\mathcal{S}_H,\hat{\partial}_H, \hat{\mathscr{T}}^H)$ be an $n$-dimensional $\mathcal{S}_H$-cycle with coefficients in a SAYD module  $M$ over $H$ (see, Definition \ref{cycle}). Then, we say that the cycle $(\mathcal{S}_H,\hat{\partial}_H, \hat{\mathscr{T}}^H)$ is vanishing if  $\mathcal{S}^0_H$ is a left $H$-category and $\mathcal S^0_H$ satisfies the assumptions in Proposition \ref{vanishing}.
\end{definition}

\smallskip
We now recall from \cite[p103]{C2}  the algebra $\mathbf C$ of infinite matrices $(a_{ij})_{i,j\in \mathbb N}$ with entries from $\mathbb C$ satisfying the following conditions
 (see also \cite{kv})
\begin{itemize}
\item[(i)] the set $\{\mbox{$a_{ij}$ $\vert$ $i,j\in \mathbb N$}\}$ is finite,
\item[(ii)] the number of non-zero entries in each row or each column is bounded.
\end{itemize}

Identifying $M_2(\mathbf{C})=\mathbf{C} \otimes M_2(\mathbb C)$, we recall the following result from \cite[p104]{C2}:
\begin{lemma}\label{6.6r}
There exists an algebra homomorphism $\omega:\mathbf{C} \longrightarrow \mathbf{C}$ and an invertible element $\tilde U \in M_2(\mathbf{C})$ such that the corresponding inner automorphism $\Xi:M_2(\mathbf{C}) \longrightarrow M_2(\mathbf{C})$ satisfies
\begin{equation}\label{6.6pqe}
\Xi(B \otimes E_{11}+\omega(B) \otimes E_{22})=\omega(B) \otimes E_{22} \qquad \forall B \in \mathbf{C}
\end{equation}
Then, $HC^\bullet(\mathbf{C})=0$.
\end{lemma}

\begin{remark}
We note that the condition in \eqref{6.6pqe} ensures that $\omega(\mathbf 1)\neq \mathbf 1$, where $\mathbf 1$ is the unit element of $\mathbf C$.
\end{remark}

For any $k$-algebra $\mathcal A$, we  may define a $k$-linear category $\mathcal{A} \otimes \mathcal{D}_H$ by setting 
$ 
Ob(\mathcal{A} \otimes \mathcal{D}_H)=Ob(\mathcal{D}_H)$ and $
Hom_{\mathcal{A} \otimes \mathcal{D}_H}(X,Y)=\mathcal{A} \otimes Hom_{\mathcal{D}_H}(X,Y)
$.  The category $\mathcal{A} \otimes \mathcal{D}_H$ is a left $H$-category via the action $h(a \otimes f):=a \otimes hf$ for any $h \in H$, $a \otimes f \in \mathcal{A} \otimes Hom_{\mathcal{D}_H}(X,Y)$.

\begin{lemma}\label{lem6.7x}
We have  $HC^\bullet_H(\mathbf{C} \otimes \mathcal{D}_H,M)=0$.
\end{lemma}
\begin{proof}
We will verify that the category $\mathbf{C} \otimes \mathcal{D}_H$ satisfies the assumptions of Proposition \ref{vanishing}. Let $\omega$ and $\tilde{U}$ be as in Lemma \ref{6.6r}. We now define $\upsilon:\mathbf{C} \otimes \mathcal{D}_H \longrightarrow \mathbf{C} \otimes \mathcal{D}_H$ given by
\begin{equation*}
\upsilon(X):=X \qquad \upsilon(B \otimes f):=\omega(B) \otimes f
\end{equation*}
for any $X \in Ob(\mathbf{C} \otimes \mathcal{D}_H)$ and $B \otimes f \in Hom_{\mathbf{C} \otimes \mathcal{D}_H}(X,Y)$. Since $\omega:\mathbf{C} \longrightarrow \mathbf{C}$ is an algebra homomorphism, it follows that $\upsilon$ is a semifunctor. By the definition of the $H$-action on $\mathbf{C} \otimes \mathcal{D}_H$, it is also clear that $\upsilon$ is $H$-linear.

\smallskip
Using the identification $\mathbf{C} \otimes \mathcal{D}_H \otimes M_2(\mathbb C) = M_2(\mathbf{C}) \otimes \mathcal{D}_H$, we now define an element $\eta \in \mathbb{G}(\mathbf{C} \otimes \mathcal{D}_H \otimes M_2(\mathbb C))=\mathbb{G}( M_2(\mathbf{C}) \otimes \mathcal{D}_H)$ given by the family of morphims
\begin{equation}\label{tqp}
\{\eta(X):=\tilde U \otimes id_X \in Hom_{M_2(\mathbf{C}) \otimes \mathcal{D}_H}(X,X)=M_2(\mathbf C) \otimes Hom_{\mathcal{D}_H}(X,X)\}_{X \in Ob(\mathcal{D}_H)}
\end{equation}
Since $\tilde U$ is a unit in $M_2(\mathbf C)$, it follows that each $\eta(X)$ in \eqref{tqp} is an automorphism. 
Since $H$ acts trivially on $M_2(\mathbf{C})$, we see that $\eta \in \mathbb{U}_H(\mathbf{C} \otimes \mathcal{D}_H \otimes M_2(\mathbb C))$. Moreover, for any $\tilde{B} \otimes f \in Hom_{M_2(\mathbf{C}) \otimes \mathcal{D}_H}(X,Y)=M_2(\mathbf{C}) \otimes Hom_{\mathcal{D}_H}(X,Y)$, we have
\begin{equation*}
\Phi_\eta(\tilde{B} \otimes f)=\eta(Y) \circ (\tilde{B} \otimes f) \circ \eta(X)^{-1}=(\tilde{U} \otimes id_Y) \circ (\tilde{B} \otimes f) \circ (\tilde{U}^{-1} \otimes id_X)=\tilde{U}\tilde{B}\tilde{U}^{-1} \otimes f=\Xi(\tilde{B}) \otimes f
\end{equation*}

\smallskip
Therefore,  for any $B \otimes f \in \mathbf{C} \otimes Hom_{\mathcal{D}_H}(X,Y)$, we have
\begin{equation*}
\begin{array}{ll}
\Phi_\eta((B \otimes f) \otimes E_{11} + \upsilon(B \otimes f) \otimes E_{22})&=\Phi_\eta(B \otimes f \otimes E_{11} + \omega(B) \otimes f \otimes E_{22})\\
&= \Phi_\eta(B \otimes E_{11} \otimes f + \omega(B) \otimes E_{22} \otimes f)\\
&= \Xi(B \otimes E_{11} + \omega(B) \otimes E_{22}) \otimes f\\
&=\omega(B) \otimes E_{22} \otimes f=\upsilon(B \otimes f) \otimes E_{22}
\end{array}
\end{equation*}
This proves the result.
\end{proof}

We are now ready to describe elements in the space $B^n_H(\mathcal{D}_H,M)$.
\begin{theorem}\label{cob1}
An element $\phi \in C^n_H(\mathcal{D}_H,M)$ is a coboundary iff $\phi$ is the character of an $n$-dimensional vanishing $\mathcal{S}_H$-cycle $(\mathcal{S}_H,\hat{\partial}_H, \hat{\mathscr{T}}^H,\rho)$ over $\mathcal{D}_H$. 
\end{theorem}
\begin{proof}
Let $\phi$ be the character of an $n$-dimensional vanishing $\mathcal{S}_H$-cycle $(\mathcal{S}_H,\hat{\partial}_H, \hat{\mathscr{T}}^H, \rho)$. By definition, $\hat{\mathscr{T}}^H$ is an $n$-dimensional closed graded $(H,M)$-trace on the $H$-semicategory $\mathcal{S}_H$ and that
$\mathcal S_H^0$ is an ordinary $H$-category. We now define $\psi \in C^n_H(\mathcal{S}_H^0,M)$ by setting
\begin{equation*}
\psi(m \otimes g^0 \otimes \ldots \otimes g^n):=\hat{\mathscr{T}}^H_{X_0}\big(m \otimes g^0\hat{\partial}_H^0(g^1) \ldots \hat{\partial}_H^0(g^n)\big)
\end{equation*}
for $m \in M$ and $g^0 \otimes \ldots \otimes g^n \in Hom_{\mathcal{S}^0_H}(X_1,X_0) \otimes Hom_{\mathcal{S}^0_H}(X_2,X_1) \otimes \ldots \otimes Hom_{\mathcal{S}^0_H}(X_0,X_n)$. Then, by the implication (1) $\Rightarrow$ (3) in Theorem \ref{charcycl}, we have that $\psi \in Z^n_H(\mathcal{S}_H^0,M)$. Since $HC^n_H(\mathcal{S}_H^0,M)=0$, we have that $\psi =b\psi'$ for some $\psi' \in C^{n-1}_H(\mathcal{S}_H^0,M)$.

\smallskip
By Lemma \ref{6.3a}, the semifunctor $\rho \in Hom_{\overline{Cat}_H}(\mathcal{D}_H,\mathcal{S}_H^0)$ induces a map $C^{n-1}_H(\rho,M):C^{n-1}_H(\mathcal{S}_H^0,M)
\longrightarrow C_H^{n-1}(\mathcal D_H,M)$. Setting $\psi'':=C^{n-1}_H(\rho,M)(\psi')$,  we have
 \begin{equation*}
\left(\psi''\right)(m \otimes p^0 \otimes \ldots \otimes p^{n-1})=\psi'\left(m \otimes \rho(p^0) \otimes \ldots \otimes \rho(p^{n-1})\right)
\end{equation*}
for any $m \in M$ and $p^0 \otimes \ldots \otimes p^{n-1} \in CN_{n-1}(\mathcal{D}_H)$. Therefore,
\begin{equation*}
\begin{array}{ll}
\phi(m \otimes f^0 \otimes \ldots \otimes f^n)&=\hat{\mathscr{T}}^H_{X_0}\big(m \otimes \rho(f^0)\hat{\partial}_H^0\left(\rho(f^1)\right) \ldots \hat{\partial}_H^0\left(\rho(f^n)\right)\big)=\psi\left(m \otimes \rho(f^0) \otimes \ldots \otimes \rho(f^n)\right)\\
&= (b\psi')\left(m \otimes \rho(f^0) \otimes \ldots \otimes \rho(f^n)\right)=(b\psi'')(m \otimes f^0 \otimes \ldots \otimes f^n)\\
\end{array}
\end{equation*}
for any $m \in M$ and $f^0 \otimes \ldots \otimes f^n \in Hom_{\mathcal{D}_H}(X_1,X_0) \otimes Hom_{\mathcal{D}_H}(X_2,X_1) \otimes \ldots \otimes Hom_{\mathcal{D}_H}(X_0,X_n)$. Thus, $\phi \in B^n_H(\mathcal{D}_H,M)$.

\smallskip
Conversely, suppose that $\phi \in B^n_H(\mathcal{D}_H,M)$. Then, $\phi=b\psi$ for some $\psi \in C^{n-1}_H(\mathcal{D}_H,M)$. We now extend $\psi$ to get an element $\psi' \in C^{n-1}_H(\mathbf{C} \otimes \mathcal{D}_H,M)$ as follows:
\begin{equation*}
\psi'\left(m \otimes (B^0 \otimes f^0) \otimes \ldots \otimes (B^{n-1} \otimes f^{n-1})\right)=\psi(m \otimes B^0_{11}f^0 \otimes \ldots \otimes B^{n-1}_{11}f^{n-1})
\end{equation*}
We now set $\phi'=b\psi' \in Z^n_H(\mathbf{C} \otimes \mathcal{D}_H,M)$. We now consider the $H$-linear semifunctor  $
\rho:\mathcal{D}_H \longrightarrow \mathbf{C} \otimes \mathcal{D}_H$ which fixes objects and takes any morphism $
 f$ to $\mathbf 1 \otimes f$. Then, we have
\begin{equation*}
\begin{array}{ll}
\left(C^n_H(\rho,M)(\phi')\right)(m \otimes f^0 \otimes \ldots \otimes f^n)&=\phi'\left(m \otimes \rho(f^0) \otimes \ldots \otimes \rho(f^n)\right)=(b\psi')\left(m \otimes \rho(f^0) \otimes \ldots \otimes \rho(f^n)\right)\\
&=(b\psi)(m \otimes f^0 \otimes \ldots \otimes f^n)=\phi(m \otimes f^0 \otimes \ldots \otimes f^n)\\
\end{array}
\end{equation*}
Since $\phi'\in Z^n_H(\mathbf{C} \otimes \mathcal{D}_H,M)$, the implication (3) $\Rightarrow$ (2) in Theorem \ref{charcycl} gives us a  closed graded $(H,M)$-trace $\mathscr{T}^H$ of dimension $n$ on the DGH-semicategory $\left(\Omega(\mathbf{C} \otimes \mathcal{D}_H),\partial_H\right)$ such that
\begin{equation}\label{7.6tre}
\mathscr{T}^H_{X_0}\left(m \otimes \rho(f^0)\partial_H^0\left(\rho(f^1)\right) \ldots \partial_H^0\left(\rho(f^n)\right)\right)=
\phi'\left(m \otimes \rho(f^0) \otimes \ldots \otimes \rho(f^n)\right)
=\phi(m \otimes f^0 \otimes \ldots \otimes f^n)
\end{equation}
Since $\left(\Omega\left(\mathbf{C} \otimes \mathcal{D}_H\right)\right)^0=\mathbf C\otimes \mathcal D_H$ is a left $H$-category, we see that $\phi$ is the character associated to the cycle $\left(\Omega\left(\mathbf{C} \otimes \mathcal{D}_H\right),\partial_H,\mathscr{T}^H, {\rho} \right)$ over $\mathcal D_H$.

\smallskip 
 From the proof of Lemma \ref{lem6.7x}, we know that  $\mathbf{C} \otimes \mathcal{D}_H$ satisfies the assumptions in Proposition \ref{vanishing}. Hence, $\left(\Omega\left(\mathbf{C} \otimes \mathcal{D}_H\right),\partial_H,\mathscr{T}^H, \rho \right)$ is a vanishing cycle over $\mathcal D_H$. From this, the result follows.
\end{proof}

For the remaining part of this section, we shall suppose that $H$ is cocommutative. If $\mathcal{D}_H$, $\mathcal{D}_H'$ are left $H$-categories, we observe that  $\mathcal{D}_H \otimes \mathcal{D}_H'$ then becomes a  left $H$-category under the diagonal action of $H$.

\smallskip
Let  $M$, $M'$ be left $H$-comodules equipped respectively with coactions $\rho_M:M\longrightarrow H\otimes M$ and $\rho_{M'}:M'\longrightarrow H\otimes M'$. Since $H$ is cocommutative, $M$ may be treated as a right $H$-comodule and we can form the cotensor product $M \square_H M'$ defined by the kernel
\begin{equation*}
\begin{CD}
M \square_H M' :=Ker\left(M \otimes M'\right. @>{\rho_M\otimes id_{M'}-id_M\otimes \rho_{M'}}>> \left. M \otimes H \otimes M'\right)\\
\end{CD}
\end{equation*}
in $Vect_k$. It follows by \cite[Proposition 7.2.2]{Caen} that the map $\rho_M \otimes id_{M'}$ gives $M \square_H M'$ a left $H$-comodule structure. We also note that $M\otimes M'$ carries a right $H$-module structure via the diagonal action. 

\begin{lemma}
Let $H$ be a cocommutative Hopf algebra and $M$, $M'$ be right-left SAYD modules  over $H$ such that $M \square_H M'$ is a right $H$-submodule of $M \otimes M'$. Then, $M \square_H M'$ is also an SAYD module over $H$.
\end{lemma}
\begin{proof}
For any $m \otimes m' \in M \square_H M'$, we have
\begin{eqnarray*}
&\left((m \otimes m')h\right)_{(-1)} \otimes \left((m \otimes m')h\right)_{(0)}=(mh_1 \otimes m'h_2)_{(-1)} \otimes (mh_1 \otimes m'h_2)_{(0)}= (mh_1)_{(-1)} \otimes (mh_1)_{(0)} \otimes m'h_2\\
&=S(h_{13})m_{(-1)}h_{11} \otimes m_{(0)}h_{12} \otimes m'h_2=S(h_{3})m_{(-1)}h_{1} \otimes m_{(0)}h_{2} \otimes m'h_4
\end{eqnarray*}
On the other hand, we have
\begin{eqnarray*}
&S(h_3)(m \otimes m')_{(-1)}h_1 \otimes (m \otimes m')_{(0)}h_2=S(h_3)m_{(-1)}h_1 \otimes (m_{(0)} \otimes m')h_2=S(h_3)m_{(-1)}h_1 \otimes m_{(0)}h_{21} \otimes m'h_{22}\\
&=S(h_4)m_{(-1)}h_1 \otimes m_{(0)}h_{2} \otimes m'h_{3}
\end{eqnarray*}
Since $H$ is cocommutative, we see that the two expressions are the same. This proves that  $M \square_H M'$ is an anti-Yetter-Drinfeld module. We now check that it is also stable. Using the cocommutativity of $H$ and the stability of $M$, $M'$,  we have
\begin{equation*}
(m \otimes m')_{(0)}(m \otimes m')_{(-1)}=m_0m_1\otimes m'm_2= m_{00}m_{01}\otimes m'm_1=m_0\otimes m'm_1=m\otimes m'_0m'_{-1}=m\otimes m'
\end{equation*}
for any $m \otimes m' \in M \square_H M'$.
\end{proof}

\begin{comment}
Our final aim is to use the method of categorified cycles and categorified vanishing cycles to obtain a pairing \begin{equation*}
HC^p_H(\mathcal{D}_H,M) \otimes HC^q(\mathcal{D}'_H,M) \longrightarrow HC^{p+q}(\mathcal{D}_H \otimes \mathcal{D}'_H,M)
\end{equation*} for $k$-linear categories $\mathcal C$ and $\mathcal C'$. 
\end{comment}

Let $(\mathcal S_H,\hat\partial_H)$ and $(\mathcal S'_H,\hat\partial'_H)$ be DGH-semicategories. Then, their tensor product $\mathcal S_H \otimes \mathcal S'_H$ is the DG-semicategory defined by setting $Ob( \mathcal S_H \otimes \mathcal S'_H)=Ob(\mathcal{S}_H) \times Ob(\mathcal{S}'_H)$ and
$$
Hom_{\mathcal S_H \otimes \mathcal S'_H}^n\left((X,X'),(Y,Y')\right)=\bigoplus\limits_{i+j=n}Hom^i_{\mathcal{S}_H}(X,Y) \otimes_k Hom^j_{\mathcal{S}'_H}(X',Y')$$
The composition in $ \mathcal S_H \otimes  \mathcal S'_H$ is given by the rule:
\begin{equation*}
(g \otimes g')\circ (f \otimes f')=(-1)^{deg(g') deg(f)} (gf \otimes g'f')
\end{equation*}
for homogeneous $f:X \longrightarrow Y$, $g: Y \longrightarrow Z$ in $\mathcal{S}_H$ and $f':X' \longrightarrow Y'$, $g': Y' \longrightarrow Z'$ in $\mathcal{S}'_H$. The differential $(\hat\partial_H\otimes \hat\partial'_H)^n :Hom_{\mathcal S_H \otimes  \mathcal S'_H}^n\left((X,X'),(Y,Y')\right) \longrightarrow Hom_{\mathcal S_H \otimes  \mathcal S'_H}^{n+1}\left((X,X'),(Y,Y')\right)$ is determined by
\begin{equation*}
(\hat\partial_H\otimes \hat\partial'_H)^n(f_i \otimes g_j)=\hat\partial^i_{H}(f_i) \otimes g_j + (-1)^i f_i \otimes {\hat{\partial'}}^{j}_{H}(g_j)
\end{equation*}
for any $f_i \in Hom^i_{\mathcal{S}_H}(X,Y)$ and $g_j \in Hom^j_{\mathcal{S}'_H}(X',Y')$ such that $i+j=n$. Clearly, $(\mathcal{S}_H \otimes \mathcal{S}'_H)^0=\mathcal{S}^0_H \otimes \mathcal{S}'^0_H$.

\begin{theorem}\label{Thmfin}
Let $H$ be a cocommutative Hopf algebra and $M$, $M'$ be right-left SAYD modules over $H$ such that $M \square_H M'$ is a right $H$-submodule of $M \otimes M'$. Let $\mathcal{D}_H$, $\mathcal{D}_H'$ be left $H$-categories. Then, we have a pairing
\begin{equation*}
HC^p_H(\mathcal{D}_H,M) \otimes HC^q_H(\mathcal{D}'_H,M') \longrightarrow HC^{p+q}_H(\mathcal{D}_H \otimes \mathcal{D}'_H,M \square_H M')
\end{equation*} for $p$, $q\geq 0$. 
\end{theorem}
\begin{proof}
Let $\phi \in Z^p_H(\mathcal{D}_H,M)$ and $\phi' \in Z^q_H(\mathcal{D}'_H,M)$. We may express $\phi$ and $\phi'$ respectively as the characters of $p$ and $q$-dimensional cycles $(\mathcal{S}_H,\hat\partial_H, \hat{\mathscr T}^H,\rho)$ and $(\mathcal{S}'_H,\hat\partial'_H,\hat{\mathscr T}'^H,\rho')$ over $\mathcal{D}_H$ and $\mathcal{D}'_H$ with coefficients in $M$ and $M'$ respectively.  
We now consider the collection $\hat{\mathscr T^H} \# \hat{\mathscr T}'^H:=\{ {(\hat{\mathscr T}^H \# \hat{\mathscr T}'^H)}_{(X,X')}:M \square_H M' \otimes Hom^{p+q}_{\mathcal{S}_H\otimes \mathcal{S}'_H}\left((X,X'),(X,X')\right) \longrightarrow \mathbb{C}\}_{(X,X') \in Ob\left(\mathcal{S}_H\otimes  \mathcal{S}'_H\right)}$ of $\mathbb{C}$-linear maps defined by
\begin{equation*}
{(\hat{\mathscr T}^H \# \hat{\mathscr T}'^H)}_{(X,X')}(m \otimes m' \otimes f \otimes f'):=\hat{\mathscr T}^H_X(m \otimes f_p)\hat{\mathscr T}'^H_{X'}(m' \otimes f'_q)
\end{equation*}
for any $m \otimes m' \in M \square_H M'$ and $f \otimes f'= (f_i \otimes f'_j)_{i+j=p+q} \in Hom^{p+q}_{\mathcal{S}_H \otimes \mathcal{S}'_H}\left((X,X'),(X,X')\right)$.
We will now prove that $\hat{\mathscr T}^H \# \hat{\mathscr T}'^H$ is a $(p+q)$-dimensional closed graded trace on the DGH-semicategory $\mathcal{S}_H\otimes \mathcal{S}'_H$ with coefficients in $M \square_H M'$. For any  $m \otimes m' \in M \square_H M'$ and $g \otimes g'= (g_i \otimes g'_j)_{i+j=p+q-1} \in Hom^{p+q-1}_{\mathcal{S}_H \otimes \mathcal{S}'_H}\left((X,X'),(X,Y)\right)$, we have
$$\begin{array}{ll}
&(\hat{\mathscr T}^H \# \hat{\mathscr T}'^H)_{(X,X')}\left(m \otimes m' \otimes (\hat\partial _H \otimes \hat\partial'_H)^{p+q-1}(g \otimes g')\right)\\
&=\underset{i+j=p+q-1}{\sum}(\hat{\mathscr T}^H \# \hat{\mathscr T}'^H)_{(X,X')}\left(m \otimes m' \otimes \hat\partial^i_{H}(g_i) \otimes g'_j + (-1)^{i} m \otimes m' \otimes g_i \otimes \hat{\partial'}^j_H(g'_j) \right)\\
&=\hat{\mathscr T}^H_X(m \otimes \hat\partial^{p-1}_{H}(g_{p-1}))\hat{\mathscr T}'^H_{X'}(m' \otimes g'_q)+(-1)^{p}\hat{\mathscr T}^H_X(m \otimes g_p)\hat{\mathscr T}'^H_{X'}(m' \otimes \hat{\partial'}^{q-1}_{H}(g'_{q-1}))=0
\end{array}$$

This proves the condition in \eqref{gt1}. Next for any homogeneous $f:X \longrightarrow Y$, $g: Y \longrightarrow X$ in $\mathcal{S}_H$ and $f':X' \longrightarrow Y'$, $g': Y' \longrightarrow X'$ in $\mathcal{S}'_H$, we have
$$\begin{array}{ll}
&(\hat{\mathscr T}^H \# \hat{\mathscr T}'^H)_{(X,X')}\left(m \otimes m' \otimes (g \otimes g')(f \otimes f')\right)\\
&=(-1)^{deg(g')deg(f)}(\hat{\mathscr T}^H \# \hat{\mathscr T}'^H)_{(X,X')}(m \otimes m' \otimes gf \otimes g'f')\\
&=(-1)^{deg(g')deg(f)}\hat{\mathscr T}^H_X(m \otimes (gf)_p)\hat{\mathscr T}'^H_{X'}(m' \otimes (g'f')_q)\\
&=(-1)^{deg(g')deg(f)}(-1)^{deg(g)deg(f)}(-1)^{deg(g')deg(f')}\hat{\mathscr T}^H_Y(m \otimes (fg)_p)\hat{\mathscr T}'^H_{Y'}(m' \otimes (f'g')_q)\\
&=(-1)^{deg(g')deg(f)}(-1)^{deg(g)deg(f)}(-1)^{deg(g')deg(f')}(-1)^{deg(g)deg(f')}(\hat{\mathscr T}^H \#\hat{\mathscr T}'^H)_{(Y,Y')}\left(m \otimes m' \otimes (f \otimes f')(g \otimes g')\right)\\
&=(-1)^{deg(g \otimes g')deg(f \otimes f')}(\hat{\mathscr T}^H \# \hat{\mathscr T}'^H)_{(Y,Y')}\left(m \otimes m' \otimes (f \otimes f')(g \otimes g')\right)
\end{array}$$
This proves the condition in \eqref{gt2}. We may similarly verify the condition in \eqref{gt0}. Thus, we get a $(p+q)$-dimensional cycle $\left(\mathcal{S}_H \otimes \mathcal{S}'_H, \hat{\partial}_H\otimes \hat{\partial}'_H, \hat{\mathscr T}^H\#\hat{\mathscr T}'^H, \rho \otimes \rho'\right)$ with coefficients in $M \square_H M'$ over the category $\mathcal D_H \otimes \mathcal{D}'_H$. Then,  the character of this cycle, denoted by $\phi \# \phi' \in Z^{p+q}_H(\mathcal{D}_H \otimes \mathcal{D}'_H, M \square_H M')$,
gives  a well defined map $\gamma:Z^p_H(\mathcal{D}_H,M) \otimes Z^q_H(\mathcal{D}'_H,M') \longrightarrow Z^{p+q}_H(\mathcal{D}_H \otimes \mathcal{D}'_H, M \square_H M')$. 

\smallskip
We now verify that the map $\gamma$ restricts to a pairing
$$\begin{array}{ll}
B^p_H(\mathcal{D}_H,M) \otimes Z^q_H(\mathcal{D}'_H,M') \longrightarrow B^{p+q}_H(\mathcal{D}_H \otimes \mathcal{D}'_H, M \square_H M')
\end{array}$$
For this, we let $\phi \in Z^p_H(\mathcal{D}_H,M)$ be the character of a $p$-dimensional vanishing cycle $\left(\mathcal{S}_H,\hat{\partial}_H,\hat{\mathscr T}^H, \rho \right)$ over $\mathcal D_H$. In particular, it follows from Definition \ref{def6.5} that $\mathcal S^0_H$ is an ordinary left $H$-category. From the implication (1) $\Rightarrow$ (2) in Theorem \ref{charcycl}, it follows that we might as well take $\mathcal S'^0_H$ to be an ordinary left $H$-category. In fact, we could assume that $\mathcal S'_H=\Omega\mathcal D'_H$. Then, $\mathcal S^0_H\otimes \mathcal S'^0_H$ is an ordinary left $H$-category. It suffices to show that the tuple $\left(\mathcal{S}_H \otimes \mathcal{S}'_H, {\hat\partial}_H \otimes {\hat\partial'}_H, \hat{\mathscr T}^H\#\hat{\mathscr T}'^H, \rho \otimes \rho'\right)$ is a vanishing cycle.

\smallskip
Since $\left(\mathcal{S}_H,{\hat\partial}_H,\hat{\mathscr T}^H\right)$ is a vanishing cycle, we have an $H$-linear semifunctor $\upsilon: \mathcal S^0_H \longrightarrow \mathcal{S}^0_H$ and an $\eta \in \mathbb{U}(\mathcal S^0_H \otimes M_2(\mathbb{C}))$ satisfying the conditions in Proposition \ref{vanishing}. Extending  $\upsilon$, we get the the $H$-linear semifunctor $\upsilon \otimes id: \mathcal{S}^0_H \otimes \mathcal S'^0_H \longrightarrow \mathcal{S}^0_H \otimes \mathcal S'^0_H$. Identifying, $\mathcal{S}^0_H \otimes \mathcal S'^0_H \otimes M_2(\mathbb{C})  \cong \mathcal{S}^0_H  \otimes M_2(\mathbb{C}) \otimes \mathcal S'^0_H$, we obtain $\tilde{\eta} \in \mathbb{U}(\mathcal{S}^0_H \otimes M_2(\mathbb{C}) \otimes \mathcal S'^0_H)$ given by
\begin{equation*}
\{\tilde{\eta}(X,X')=\eta(X) \otimes id_{X'} \in Hom_{\mathcal{S}^0_H  \otimes M_2(\mathbb{C}) \otimes \mathcal S'^0_H}((X,X'),(X,X'))=Hom_{\mathcal{S}^0_H  \otimes M_2(\mathbb{C})}(X,X) \otimes Hom_{\mathcal{S}'^0_H}(X',X')\}
\end{equation*}
It may also be easily verified that
\begin{equation*}
\Phi_{\tilde{\eta}}(f \otimes f' \otimes E_{11} + (\upsilon \otimes id)(f \otimes f') \otimes E_{22})=(\upsilon \otimes id)(f \otimes f') \otimes E_{22}
\end{equation*}
Thus, we see that the category $(\mathcal S_H \otimes \mathcal S'_H)^0=\mathcal{S}^0_H \otimes \mathcal{S}'^0_H$   satisfies the conditions in  Proposition \ref{vanishing}. Therefore, the tuple $\left(\mathcal{S}_H \otimes \mathcal{S}'_H, {\hat\partial}_H \otimes {\hat\partial'}_H, \hat{\mathscr T}^H\#\hat{\mathscr T}'^H, \rho \otimes \rho'\right)$ is a vanishing cycle.
This proves the result.
\end{proof}

\section{Characters of Fredholm modules over categories}\label{evencycl}

In the rest of this paper, we will study Fredholm modules and Chern characters. We  fix a small $\mathbb C$-linear category $\mathcal C$. Our categorified Fredholm modules will consist of functors from $\mathcal C$  taking values in 
separable Hilbert spaces.  Let $SHilb$ be the category whose objects are separable Hilbert spaces and  whose morphisms are bounded linear maps.  

\smallskip

 Given separable Hilbert spaces $\mathcal H_1$ and $\mathcal H_2$, let $\mathcal{B}(\mathcal H_1,\mathcal H_2)$ denote the space of all bounded linear operators from $\mathcal H_1$ to $\mathcal H_2$  and $\mathcal{B}^\infty(\mathcal H_1,\mathcal H_2) \subseteq \mathcal{B}(\mathcal H_1,\mathcal H_2)$ be the space of all  compact operators. For any bounded operator  $T \in \mathcal{B}(\mathcal H_1,\mathcal H_2)$, let $\mu_n(T)$ denote the $n$-th singular value of $T$. In other words, $\mu_n(T)$ is the $n$-th (arranged
 in decreasing order) eigenvalue of the positive operator  $|T|:=(T^*T)^{\frac{1}{2}}$. 
For $1 \leq p < \infty$, the $p$-th Schatten class is defined to be the space
\begin{equation*}
\mathcal{B}^p(\mathcal H_1,\mathcal H_2):=\{T \in \mathcal{B}(\mathcal H_1,\mathcal H_2)~|~ \sum \mu_n(T)^p < \infty\}
\end{equation*}
Clearly, $\mathcal{B}^p(\mathcal H_1,\mathcal H_2) \subseteq  \mathcal{B}^q(\mathcal H_1,\mathcal H_2)$  for $p \leq q$. For $p=1$, the space $\mathcal{B}^1(\mathcal H_1,\mathcal H_2)$ is the collection of all trace class operators from $\mathcal H_1$ to $\mathcal H_2$. For  $T \in \mathcal{B}^1(\mathcal H_1,\mathcal H_2)$, we write $Tr(T):=\sum \mu_n(T)$. It is well known that
\begin{equation}\label{commutr}
Tr(T_1T_2)=Tr(T_2T_1) \quad \forall T_1 \in \mathcal{B}^{n_1}(\mathcal{H}, \mathcal{H}'),~ T_2 \in \mathcal{B}^{n_2}(\mathcal{H}', \mathcal{H}), \frac{1}{n_1}+\frac{1}{n_2}=1
\end{equation}

We note that $\mathcal{B}^p(\mathcal H_1,\mathcal H_2)$ is an ``ideal"  in the following sense: consider the functor
\begin{equation*}
\begin{array}{c} \mathcal{B}(-,-):SHilb^{op} \otimes SHilb \longrightarrow Vect_\mathbb{C} \qquad \mathcal B(-,-)(\mathcal H_1,\mathcal H_2):=\mathcal{B}(\mathcal H_1,\mathcal H_2)
\\ \mathcal{B}(-,-)(\phi_1,\phi_2):\mathcal{B}(\mathcal H_1, \mathcal H_2) \longrightarrow \mathcal{B}(\mathcal H_1',\mathcal H_2')\qquad T \mapsto \phi_2T\phi_1\\
\end{array}
\end{equation*}  taking values in the category $Vect_{\mathbb C}$ of $\mathbb C$-vector spaces. Then, $\mathcal{B}^p(-,-)$ is a subfunctor of $\mathcal{B}(-,-)$. In other words, for  morphisms $\phi_1:\mathcal H_1'\longrightarrow \mathcal H_1$, $\phi_2:\mathcal H_2\longrightarrow \mathcal H_2'$  and any $T \in \mathcal{B}^p(\mathcal H_1,\mathcal H_2)$, we have $\phi_2T\phi_1 \in \mathcal{B}^p(\mathcal H_1',\mathcal H_2')$.

\smallskip
We fix the following convention for the commutator notation:  Let $\mathscr{H}:\mathcal{C}  \longrightarrow SHilb$ be a functor and $\mathcal{G}:=\{\mathcal{G}_X:\mathscr{H}(X) \longrightarrow \mathscr{H}(X)\}_{X \in Ob(\mathcal{C})}$ be a collection of bounded linear operators. Then, we set 
\begin{equation*}
[\mathcal G,-]:\mathcal B(\mathscr H(X),\mathscr H(Y))\longrightarrow \mathcal B(\mathscr H(X),\mathscr H(Y)) \qquad [\mathcal G,T]:= \mathcal{G}_Y \circ T - T\circ \mathcal{G}_X\in \mathcal B(\mathscr H(X),\mathscr H(Y))
\end{equation*} 
We now let $SHilb_{\mathbb{Z}_2}$ be the category whose objects are $\mathbb{Z}_2$-graded separable Hilbert spaces and whose morphims are bounded linear maps.   Let $\mathscr{H}:\mathcal{C}  \longrightarrow SHilb_{\mathbb{Z}_2}$ be a functor and $\mathcal{G}:=\{\mathcal{G}_X:\mathscr{H}(X) \longrightarrow \mathscr{H}(X)\}_{X \in Ob(\mathcal{C})}$ be a collection of bounded linear operators of the same degree $|\mathcal G|$. Then, we set 
\begin{equation}\label{gradcomm}
\begin{array}{c}
[\mathcal G,-]:\mathcal B(\mathscr H(X),\mathscr H(Y))\longrightarrow \mathcal B(\mathscr H(X),\mathscr H(Y)) \\  
\mbox{$[\mathcal G,T]:= \mathcal{G}_Y \circ T -(-1)^{|\mathcal G||T|} T\circ \mathcal{G}_X\in \mathcal B(\mathscr H(X),\mathscr H(Y))$}\\
\end{array}
\end{equation}  for each $X$, $Y\in \mathcal C$.

\begin{definition}\label{dD3.1}
Let $\mathcal{C}$ be a small $\mathbb{C}$-category and let $p \in [1,\infty)$. We consider a pair $(\mathscr H,\mathcal F)$ as follows.

\smallskip
\begin{itemize}
\item[(1)] A functor  $\mathscr{H}:\mathcal{C}  \longrightarrow SHilb_{\mathbb{Z}_2}$ such that $\mathscr H(f):\mathscr H(X)\longrightarrow \mathscr H(Y)$ is a linear operator
of degree $0$ for each  $f \in Hom_{\mathcal{C}}(X,Y)$.

\item[(2)] A collection $\mathcal{F}:=\{\mathcal{F}_X:\mathscr{H}(X) \longrightarrow \mathscr{H}(X)\}_{X \in Ob(\mathcal{C})}$  of bounded linear operators of degree $1$ such that $\mathcal{F}_X^2=id_{\mathscr{H}(X)}$ for each $X \in Ob(\mathcal{C})$.
\end{itemize}

The pair $(\mathscr{H},\mathcal F)$ is said to be a $p$-summable even Fredholm module over the category $\mathcal{C}$ if
every $f \in Hom_{\mathcal{C}}(X,Y)$ satisfies
\begin{equation}\label{stneve}
[\mathcal{F},f]:= \left(\mathcal{F}_Y \circ \mathscr{H}(f) - \mathscr{H}(f) \circ \mathcal{F}_X \right)     \in \mathcal{B}^p\left(\mathscr{H}(X),\mathscr{H}(Y)\right)
\end{equation}
\end{definition} 

Taking $H=\mathbb C=M$ in Definition \ref{gradedtrace}, we note that a closed graded trace of dimension $n$ on a DG-semicategory $(\mathcal S,\hat\partial)$ is a collection of $\mathbb{C}$-linear maps $\hat T:=\{\hat T_X: Hom^n_\mathcal{S}(X,X) \longrightarrow \mathbb{C}\}_{X \in Ob(\mathcal{S})}$ satisfying the following two conditions
\begin{equation}\label{clgrfr}
\hat T_X\left(\hat\partial^{n-1}(f)\right)=0\qquad 
\hat T_X(gg')=(-1)^{ij}~ \hat T_Y(g'g)
\end{equation} for all $f \in Hom^{n-1}_{\mathcal S}(X,X)$, $g \in Hom^i_{\mathcal S}(Y,X)$, $g' \in Hom^j_{\mathcal S}(X,Y)$ and $i+j=n$. Accordingly, we will consider cycles  $(\mathcal{S},\hat{\partial},\hat{T},\rho)$  over $\mathcal C$ by setting $H=\mathbb C=M$ in Definition \ref{cycle}.

\smallskip

Let $(\mathscr{H},\mathcal F)$ be a pair that satisfies conditions (1) and (2) in Definition \ref{dD3.1}. We define a graded-semicategory
$\Omega'\mathcal C=\Omega_{(\mathscr H,\mathcal F)}\mathcal C$ as follows: we put $Ob(\Omega'\mathcal C):=Ob(\mathcal C)$ and for any $X$, $Y\in \mathcal C$, $j\geq 0$, we set
$Hom^j_{\Omega'{\mathcal C}}(X,Y)$ to be the linear span in $\mathcal B(\mathscr H(X),\mathscr H(Y))$ of the operators
\begin{equation}
\mathscr H(\tilde{f}^0)[\mathcal F,f^1][\mathcal F,f^2] \ldots [\mathcal F,f^{j}]
\end{equation} where  $\tilde{f}^0\otimes f^1\otimes ...\otimes f^j$ is a homogeneous element of degree $j$ in $Hom_{\Omega\mathcal{C}}(X,Y)$. Here, we write
$\mathscr H(\tilde{f}^0)=\mathscr H({f}^0)+\mu \cdot id$, where $\tilde{f}^0=f^0+\mu$.  Using the fact that
\begin{equation*}
[\mathcal F,f]\mathscr H(f')=[\mathcal F,f\circ f']-\mathscr H(f)[\mathcal F,f']
\end{equation*} for composable morphisms $f$, $f'$ in $\mathcal C$, we observe that $\Omega'\mathcal C$ is closed under composition.
We set 
\begin{equation*}
\begin{array}{c}
\partial':=[\mathcal{F},-]:   \mathcal{B}\left(\mathscr{H}(X),\mathscr{H}(Y)\right) \longrightarrow \mathcal{B}\left(\mathscr{H}(X),\mathscr{H}(Y)\right)\\
\partial' T=[\mathcal{F},T]=\mathcal{F}_Y \circ T -(-1)^{|T|} T\circ \mathcal{F}_X\\
\end{array}
\end{equation*} We now have the following Lemma. 
\begin{comment}
For each $X,Y \in Ob(\Omega'\mathcal{C})$,  the differential $\partial'^j_{XY}:Hom^j_{\Omega'\mathcal{C}}(X,Y) \longrightarrow Hom^{j+1}_{\Omega'\mathcal{C}}(X,Y)$ is determined by setting
\begin{equation}\label{part'}
\begin{array}{ll}
\partial'^j_{XY}(\mathscr H(\tilde{f}^0)[\mathcal F,f^1][\mathcal F,f^2] \ldots [\mathcal F,f^{j}])&=[\mathcal F,\textrm{ }\mathscr H(\tilde{f}^0)[\mathcal F,f^1][\mathcal F,f^2] \ldots [\mathcal F,f^{j}]]\\
&=\mathcal F_Y\circ \left(\mathscr H(\tilde{f}^0)[\mathcal F,f^1][\mathcal F,f^2] \ldots [\mathcal F,f^{j}]\right)\\ &\textrm{ }-(-1)^j \left(\mathscr H(\tilde{f}^0)[\mathcal F,f^1][\mathcal F,f^2] \ldots [\mathcal F,f^{j}]\right)\circ\mathcal F_X\\
\end{array}
\end{equation}
\end{comment}
%%%%%%%%%%%%%%%%%%%%%%%%

\begin{lemma} \label{srLe1} Let $(\mathscr{H},\mathcal F)$ be a pair that satisfies conditions (1) and (2) in Definition \ref{dD3.1}. Then, 

\smallskip
(a) $(\Omega'\mathcal C,\partial')$  is a DG-semicategory and $\Omega'^0\mathcal C$ is an ordinary category. 

\smallskip
(b) There is a canonical semifunctor $\rho'=\rho_{\mathscr H}:\mathcal C\longrightarrow \Omega'^0\mathcal C$ which is identity on objects and takes any $f\in Hom_{\mathcal C}(X,Y)$
to $\mathscr H(f)\in \mathcal B(\mathscr H(X),\mathscr H(Y))$. This extends to  a unique DG-semifunctor  $\hat{\rho}'=\hat{\rho}_{\mathscr H}:(\Omega\mathcal C,\partial)\longrightarrow (\Omega'\mathcal C,\partial')$  such that the restriction of  $\hat{\rho}'$  to $\mathcal C$ is identical to $\rho'$.

\smallskip
(c) Suppose that $(\mathscr{H},\mathcal F)$ is  a $p$-summable  Fredholm module. Choose $n\geq p-1$. Then, for $X$, $Y\in Ob(\mathcal C)$ and $k\geq 0$, we have $Hom^k_{\Omega'{\mathcal C}}(X,Y)
\subseteq \mathcal B^{(n+1)/k}(\mathscr H(X),\mathscr H(Y))$. 
\end{lemma}

\begin{proof}(a) Since each $\mathcal{F}_X$ is a degree 1 operator and $\mathcal F_Y[\mathcal F,f]=-[\mathcal F,f]\mathcal F_X$ for any 
$f\in Hom_{\mathcal C}(X,Y)$, we have $\partial'\left(Hom^j_{\Omega' \mathcal{C}}(X,Y)\right) \subseteq Hom^{j+1}_{\Omega' \mathcal{C}}(X,Y)$.
We now check that $\partial'^2=0$. For any homogeneous element $\mathscr H(\tilde{f}^0)[\mathcal F,f^1][\mathcal F,f^2] \ldots [\mathcal F,f^{j}]$ of degree $j$, we have
\begin{equation*}
\begin{array}{ll}
\partial'^2\left(\mathscr H(\tilde{f}^0)[\mathcal F,f^1][\mathcal F,f^2] \ldots [\mathcal F,f^{j}]\right)&=\partial'\left(\mathcal F_Y\circ \left(\mathscr H(\tilde{f}^0)[\mathcal F,f^1][\mathcal F,f^2] \ldots [\mathcal F,f^{j}]\right)\right)\\ 
&\textrm{ }-(-1)^j \partial' \left(\left(\mathscr H(\tilde{f}^0)[\mathcal F,f^1][\mathcal F,f^2] \ldots [\mathcal F,f^{j}]\right)\circ\mathcal F_X\right)\\
& =\mathcal F_Y^2 \circ \mathscr H(\tilde{f}^0)[\mathcal F,f^1][\mathcal F,f^2] \ldots [\mathcal F,f^{j}]\\
& \quad -(-1)^{j+1} \mathcal{F}_Y \circ \mathscr H(\tilde{f}^0)[\mathcal F,f^1][\mathcal F,f^2] \ldots [\mathcal F,f^{j}]\circ\mathcal F_X \\
& \quad - (-1)^j \big( \mathcal{F}_Y \circ \mathscr H(\tilde{f}^0)[\mathcal F,f^1][\mathcal F,f^2] \ldots [\mathcal F,f^{j}]\circ\mathcal F_X \\
&\quad - (-1)^{j+1}  \mathscr H(\tilde{f}^0)[\mathcal F,f^1][\mathcal F,f^2] \ldots [\mathcal F,f^{j}]\circ\mathcal F_X^2      \big)\\
& =0
\end{array}
\end{equation*}
The fact that $\partial'$ is compatible with composition follows by direct computation. It is also easy to see that $\Omega'^0\mathcal C$ is an ordinary category. 

\smallskip
(b) This is immediate using the universal property in Proposition \ref{construniv}.

\smallskip
(c) This is a consequence of H\"older's inequality and the condition \eqref{stneve} in Definition \ref{dD3.1}.
\end{proof}

For any $\mathbb Z_2$-graded Hilbert space $\mathcal H$, the grading operator on it will be denoted by $\epsilon_{\mathcal H}$ or simply $\epsilon$. 
For any $T \in \mathcal{B}(\mathscr H(X), \mathscr H(Y))$ such that $[\mathcal{F},T] \in \mathcal{B}^1(\mathscr H(X), \mathscr H(Y))$, we define
\begin{equation*}
Tr_s(T):=\frac{1}{2}~Tr\left(\epsilon \mathcal F_Y[\mathcal F,T]\right)=\frac{1}{2}~Tr\left(\epsilon \mathcal F_Y\partial'(T)\right)=\frac{1}{2}~ Tr\left(\epsilon\mathcal{F}_Y(\mathcal{F}_Y \circ T - (-1)^{|T|}~ T \circ \mathcal{F}_X)\right)
\end{equation*}

\begin{prop}\label{srLe2} Let $(\mathscr{H},\mathcal F)$ be  a $p$-summable Fredholm module over $\mathcal C$. Take $ 2m \geq p-1$. Then, the collection
\begin{equation}
\hat{Tr}_s=\{Tr_s:Hom^{2m}_{\Omega'\mathcal{C}}(X,X) \longrightarrow \mathbb C\}_{X\in Ob(\mathcal C)}
\end{equation} defines a closed graded trace of dimension $2m$ on $(\Omega'\mathcal C,\partial')$. 
\end{prop}

\begin{proof} From Lemma \ref{srLe1}(a), it is clear that for any $T\in Hom^{2m}_{\Omega'\mathcal{C}}(X,X)$, we have $[\mathcal F,T]\in Hom^{2m+1}_{\Omega'\mathcal{C}}(X,X)$. Applying Lemma \ref{srLe1}(c), it follows that $[\mathcal F,T]\in  \mathcal B^1(\mathscr H(X),\mathscr H(X))$. Accordingly, each of the maps $Tr_s:Hom^{2m}_{\Omega'\mathcal{C}}(X,X) \longrightarrow \mathbb C$ is well-defined.

\smallskip
For $T'\in Hom^{2m-1}_{\Omega'\mathcal{C}}(X,X)$, we notice that
\begin{equation*}
Tr_s(\partial'T')=\frac{1}{2}~Tr\left(\epsilon \mathcal F_X(\partial'^2T')\right)=0
\end{equation*} We now consider $T_1\in Hom^{i}_{\Omega'\mathcal{C}}(X,Y)$, $T_2\in Hom^{j}_{\Omega'\mathcal{C}}(Y,X)$ such that $i+j=2m$. We notice that
\begin{equation}\label{sre3.4}
\begin{array}{c}
\epsilon \mathcal F_Y\partial'(T_1)=\partial'(T_1)\epsilon \mathcal F_X \qquad \epsilon \mathcal F_X\partial'(T_2)=\partial'(T_2)\epsilon \mathcal F_Y
\end{array}
\end{equation} We note that $i\equiv j \mbox{(mod $2$)}$. Using \eqref{sre3.4} and \eqref{commutr}, we now have
\begin{equation*}
\begin{array}{ll}
2\cdot Tr_s(T_1T_2)=Tr\left(\epsilon \mathcal F_Y\partial'(T_1T_2)\right) & = Tr\left(\epsilon \mathcal F_Y\partial'(T_1)T_2\right)+(-1)^iTr\left(\epsilon \mathcal F_YT_1\partial'(T_2)\right)\\
& = Tr\left(\partial'(T_1)\epsilon \mathcal F_XT_2\right)+(-1)^iTr\left(\partial'(T_2)\epsilon \mathcal F_YT_1\right)\\
& =Tr\left(\epsilon \mathcal F_XT_2\partial'(T_1)\right)+(-1)^iTr\left(\epsilon \mathcal F_X\partial'(T_2)T_1\right)\\
& =  Tr\left(\epsilon \mathcal F_XT_2\partial'(T_1)\right)+(-1)^jTr\left(\epsilon \mathcal F_X\partial'(T_2)T_1\right)\\
&= (-1)^{ij}2\cdot Tr_s(T_2T_1)
\end{array}
\end{equation*}
\end{proof}

\begin{theorem}\label{evencyc}
Let $(\mathscr{H},\mathcal F)$ be  a $p$-summable  Fredholm module over $\mathcal C$. Take $2m \geq p-1$.  Then, the tuple
$(\Omega'\mathcal C,\partial',\hat{Tr}_s,\rho')$ defines a $2m$-dimensional cycle over $\mathcal C$. Then, $\phi^{2m} \in CN^{2m}(\mathcal{C})=C^{2m}_{\mathbb C}(\mathcal C,\mathbb C)=Hom(CN_{2m}(\mathcal C),\mathbb C)$ defined by
\begin{equation*}
\phi^{2m}(f^0 \otimes f^1 \otimes \ldots \otimes f^{2m}):= Tr_s\left(\mathscr H({f}^0)[\mathcal F,f^1][\mathcal F,f^2] \ldots [\mathcal F,f^{2m}]\right)
\end{equation*}
for any $f^0 \otimes f^1 \otimes \ldots \otimes f^{2m} \in Hom_{\mathcal{C}}(X_1,X) \otimes Hom_{\mathcal{C}}(X_2,X_1) \otimes \ldots \otimes Hom_{\mathcal{C}}(X,X_{2m})$ is a cyclic cocycle over $\mathcal C$.
\end{theorem}

\begin{proof} It  follows directly from Lemma \ref{srLe1} and Proposition \ref{srLe2} that  $(\Omega'\mathcal C,\partial',\hat{Tr}_s,\rho')$ is a $2m$-dimensional cycle over $\mathcal C$. 
The rest follows by applying Theorem \ref{charcycl} with $H=\mathbb C=M$.
\end{proof}

We will refer to $\phi^{2m}$ as the $2m$-dimensional character associated with the $p$-summable  even Fredholm module $(\mathscr{H},\mathcal F)$ over the category $\mathcal{C}$.

\begin{remark} 
The appearance of only even cyclic cocycles in Theorem \ref{evencyc} is due to the following fact from \cite[Lemma 2 a)]{C2}: if $T \in \mathcal{B}(\mathscr H(X), \mathscr H(X))$ is homogeneous of odd degree, then $Tr_s(T)=0$.
\end{remark}

\section{Periodicity of Chern character for Fredholm modules}
We continue with $\mathcal{C}$ being a small $\mathbb{C}$-category. Taking $H=\mathbb{C}=M$, we denote the cyclic cohomology groups of $\mathcal C$ by $H^\bullet_\lambda(\mathcal C):=HC^\bullet_{\mathbb C}(\mathcal C,\mathbb C)$. The cyclic complex corresponding to the cocyclic module $\{CN^n(\mathcal C)=Hom_{\mathbb C}(CN_n(\mathcal C),\mathbb C)\}_{n\geq 0}$ as in \eqref{R2.5c} will be denoted by $C^\bullet_\lambda(\mathcal C)$. The  cocycles  of this complex will be denoted by $Z^\bullet_\lambda(\mathcal C):=Z^\bullet_{\mathbb C}(\mathcal C,\mathbb C)$ and the coboundaries by $B^\bullet_\lambda(\mathcal C):=B^\bullet_{\mathbb C}(\mathcal C,\mathbb C)$. 

\smallskip Let $(\mathscr{H},\mathcal{F})$ be a $p$-summable Fredholm module over $\mathcal{C}$. We take  $2m \geq p-1$. Let $\phi^{2m}$ be the $2m$-dimensional character associated to the Fredholm module $(\mathscr{H},\mathcal F)$. We denote by $ch^{2m}(\mathscr H,\mathcal F)\in H^{2m}_\lambda(\mathcal C)$ the cohomology class of $\phi^{2m}$. 
Since $\mathcal{B}^p(\mathscr H(X), \mathscr H(Y)) \subseteq \mathcal{B}^q(\mathscr H(X), \mathscr H(Y))$ for any $p \leq q$, the Fredholm module $(\mathscr{H},\mathcal{F})$ is also $(p+2)$-summable. Using Theorem \ref{evencyc}, we then have the $(2m+2)$-dimensional character $\phi^{2m+2}$ associated to $(\mathscr{H},\mathcal{F})$.
We will show that the cyclic cocycles $\phi^{2m}$ and $\phi^{2m+2}$ are related to each other via the periodicity operator. 

\smallskip
 If  $\mathcal C$ and $\mathcal C'$ are small $\mathbb{C}$-categories, from the proof of Theorem \ref{Thmfin} it follows that there is a pairing on cyclic cocycles
\begin{equation}\label{pairty}
Z^r_\lambda(\mathcal{C}) \otimes Z^s_\lambda(\mathcal{C}') \longrightarrow Z^{r+s}_\lambda(\mathcal{C} \otimes \mathcal{C}') \qquad \phi \otimes \phi' \mapsto \phi \# \phi'
\end{equation} which descends to a pairing on cyclic cohomologies:
\begin{equation}\label{pair}
H^r_\lambda(\mathcal{C}) \otimes H^s_\lambda(\mathcal{C}') \longrightarrow H^{r+s}_\lambda(\mathcal{C} \otimes \mathcal{C}') 
\end{equation}
given by
\begin{equation}
{(\hat{T}^\phi \# \hat{T}^{\phi'})}_{(X,X')}(f \otimes f'):=\hat{T}^\phi_X( f_r)\hat{T}^{\phi'}_{X'}(f'_s)
\end{equation}
for any $f \otimes f'= \underset{i+j=r+s}{\sum}(f_i \otimes f'_j)\in Hom^{r+s}_{\mathcal S \otimes  \mathcal S'}\left((X,X'),(X,X')\right)$. Here $\phi$ and $\phi'$ are expressed respectively as the characters of $r$ and $s$-dimensional cycles $(\mathcal{S},\hat\partial, \hat{T}^\phi,\rho)$ and $(\mathcal{S}',\hat\partial',\hat{T}^{\phi'},\rho')$ over $\mathcal{C}$ and $\mathcal{C}'$. In particular, $\phi \# \phi'$ is the character of  the $(r+s)$-dimensional cycle $\left(\mathcal{S} \otimes \mathcal{S}', \hat{\partial} \otimes \hat{\partial'}, \hat{T}^\phi\# \hat{T}^{\phi'}, \rho \otimes \rho' \right)$ over $\mathcal{C} \otimes \mathcal{C}'$. For a morphism $f$ in $\mathcal C$, we will often suppress the functor $\rho$ and write
the morphism $\rho(f)$ in $\mathcal S^0$ simply as $f$.  Similarly, when there is no danger of confusion, we will often write the morphism $\mathscr H(f)$ simply as $f$.

\smallskip
Now setting $\mathcal{C}'=\mathbb{C}$ (the category with one object) and considering the cyclic cocycle $\psi \in H^2_\lambda(\mathbb{C})$ determined by $\psi(1,1,1)=1$, we obtain the periodicity operator:
\begin{equation*} 
S:Z^r_\lambda(\mathcal{C}) \longrightarrow  Z^{r+2}_\lambda(\mathcal{C}) \qquad S(\phi):=\phi \# \psi 
\end{equation*}
for any $r \geq 0$ and $\phi \in Z^r_\lambda(\mathcal{C})$. 

\begin{lemma}\label{Sform} Let $\phi \in Z^r_\lambda(\mathcal{C})$.
For any $f^0 \otimes f^1 \otimes \ldots \otimes f^{r+2} \in CN_{r+2}(\mathcal{C})$, we have

\smallskip
$\begin{array}{ll}
(S(\phi))(f^0 \otimes f^1 \otimes \ldots \otimes f^{r+2})
&=\hat{T}^\phi_X(f^0f^1f^2\hat\partial f^3 \ldots \hat\partial f^{r+2})+\hat{T}^\phi_X(f^0\hat\partial f^1(f^2f^3) \ldots \hat\partial f^{r+2})+ \ldots \\
&\quad +\hat{T}^\phi_X(f^0\hat\partial f^1 \ldots \hat\partial f^{i-1} (f^if^{i+1})\hat\partial f^{i+2} \ldots \hat\partial f^{r+2}) + \ldots \\
& \quad +\hat{T}^\phi_X(f^0\hat\partial f^1 \ldots \hat\partial f^r (f^{r+1}f^{r+2}))
\end{array}$
\end{lemma}

\begin{proof} We consider the $2$-dimensional trace $\hat{T}^\psi$ on the DG-semicategory $(\Omega\mathbb C,\partial )$ such that $\psi\in Z^2_\lambda(\mathbb C)$ is the character of the corresponding cycle over $\mathbb C$.  We first observe that we have the following equalities in $\Omega \mathbb{C}$:
\begin{equation*}
\partial 1= (\partial1)1+1 (\partial 1), \qquad 1(\partial1)1=0, \qquad 1(\partial 1)^2=(\partial1)^21
\end{equation*}

We illustrate the proof for $r=2$. The general case will follow similarly.  By definition, we have

\smallskip
$\begin{array}{ll}
&(S(\phi))(f^0 \otimes f^1 \otimes f^2 \otimes f^3 \otimes f^{4})\\
& \quad =(\phi \# \psi)(f^0 \otimes f^1 \otimes f^2 \otimes f^3 \otimes f^{4})\\
& \quad ={(\hat{T}^\phi \# \hat{T}^\psi)}\left( (f^0 \otimes 1)(\hat{\partial} \otimes \partial)(f^1 \otimes 1) (\hat{\partial} \otimes \partial)(f^2 \otimes 1)   (\hat{\partial} \otimes \partial)(f^3 \otimes 1) (\hat{\partial} \otimes \partial)(f^{4} \otimes 1) \right)\\
& \quad ={(\hat{T}^\phi \# \hat{T}^\psi)}\left( (f^0 \otimes 1) (\hat{\partial}f^1 \otimes 1 +f^1 \otimes \partial 1) (\hat \partial f^2 \otimes 1 +f^2 \otimes \partial1)(\hat \partial f^3 \otimes 1 +f^3 \otimes \partial1) (\hat \partial f^{4} \otimes 1 +f^{4} \otimes \partial1) \right)\\
& \quad ={(\hat{T}^\phi \# \hat{T}^\psi)} \Big(f^0 \hat \partial f^1 \hat \partial f^2 \hat \partial f^3 \hat \partial f^4 \otimes 1+f^0 \hat \partial f^1 \hat \partial f^2 \hat\partial f^3f^4  \otimes 1 \partial 1 + f^0\hat \partial f^1 \hat \partial f^2f^3f^4 \otimes 1(\partial 1)^2 +\\
&\qquad f^0 \hat \partial f^1f^2f^3 \hat \partial f^4 \otimes 1(\partial 1)^21 + f^0 \hat \partial f^1f^2f^3f^4 \otimes 1(\partial 1)^3 + f^0f^1f^2 \hat \partial f^3\hat \partial f^4 \otimes 1 (\partial 1)^2 \\
&\quad + f^0f^1f^2\hat \partial f^3 f^4 \otimes 1(\partial 1)^3 -
 f^0f^1f^2f^3\hat \partial f^4 \otimes 1(\partial 1)^31 + f^0f^1f^2f^3f^4 \otimes 1(\partial 1)^4\Big)\\
& \quad =\hat{T}^\phi\left(f^0 \hat \partial f^1 \hat \partial f^2f^3f^4\right) \hat{T}^\psi\left(1(\partial 1)^2\right) + \hat{T}^\phi\left(f^0\hat \partial f^1f^2f^3\hat \partial f^4\right) \hat{T}^\psi\left(1(\partial 1)^21\right)\\
& \qquad + \hat{T}^\phi\left(f^0f^1f^2\hat \partial f^3 \hat \partial f^4\right) \hat{T}^\psi\left(1(\partial 1)^2\right)\\
& \quad= \hat{T}^\phi\left(f^0\hat \partial f^1 \hat \partial f^2f^3f^4\right) + \hat{T}^\phi\left(f^0 \hat \partial f^1f^2f^3\hat\partial f^4\right)+ \hat{T}^\phi\left(f^0f^1f^2\hat \partial f^3 \hat \partial f^4\right) 
\end{array}$

\smallskip
The last equality follows by using the fact that $\hat{T}^\psi\left(1(\partial 1)^2\right)=\psi(1,1,1)=1$.
\end{proof}

\begin{prop}\label{Sbound}
Let  $\phi$ be the character of an $r$-dimensional cycle $(\mathcal S,\hat{\partial},\hat{T}^{\phi},\rho)$ over $\mathcal{C}$. Then, $S(\phi)$ is a coboundary.  In particular, we have $S(\phi)=b\psi$, where $\psi \in CN^{r+1}(\mathcal{C})$ is given by
\begin{equation*}
\psi(f^0 \otimes f^1 \otimes \ldots \otimes f^{r+1})=\sum\limits_{j=1}^{r+1} (-1)^{j-1}~\hat{T}^{\phi}\left(f^0\hat{\partial}f^1\ldots  \hat{\partial}f^{j-1} f^j \hat{\partial}f^{j+1} \ldots  \hat{\partial}f^{r+1}\right)
\end{equation*}
\end{prop}
\begin{proof}
Again, we illustrate the case of $r=2$. The general computation is similar.
\begin{equation*}
\begin{array}{ll}
&(b\psi)(f^0\otimes f^1 \otimes f^2 \otimes f^3 \otimes f^4)\\
&=\psi(f^0f^1 \otimes f^2 \otimes f^3 \otimes f^4) - \psi(f^0 \otimes f^1f^2 \otimes f^3 \otimes f^4) + \psi(f^0 \otimes f^1 \otimes f^2f^3 \otimes f^4) - \psi(f^0 \otimes f^1 \otimes f^2 \otimes f^3f^4)\\
&\quad + \psi(f^4f^0 \otimes f^1 \otimes f^2 \otimes f^3)\\
&=\hat{T}^\phi(f^0f^1f^2\hat{\partial}f^3\hat{\partial}f^4)-\hat{T}^\phi(f^0f^1\hat{\partial}f^2f^3\hat{\partial}f^4)+ \hat{T}^\phi(f^0f^1\hat{\partial}f^2\hat{\partial}f^3f^4) \\
&\quad  -\hat{T}^\phi(f^0f^1f^2\hat{\partial}f^3\hat{\partial}f^4)+\hat{T}^\phi(f^0\hat{\partial}(f^1f^2)f^3\hat{\partial} f^4-\hat{T}^\phi(f^0\hat{\partial}(f^1f^2)\hat{\partial}f^3f^4)\\
&\quad +\hat{T}^\phi(f^0f^1\hat{\partial}(f^2f^3)\hat{\partial}f^4)-\hat{T}^\phi(f^0\hat{\partial}f^1f^2f^3\hat{\partial}f^4)+ \hat{T}^\phi(f^0\hat{\partial}f^1\hat{\partial}(f^2f^3)f^4)\\
&\quad -\hat{T}^\phi(f^0f^1\hat{\partial}f^2\hat{\partial}(f^3f^4))+ \hat{T}^\phi(f^0\hat{\partial}f^1f^2\hat{\partial}(f^3f^4))- \hat{T}^\phi(f^0\hat{\partial}f^1\hat{\partial}f^2f^3f^4)\\
&\quad +\hat{T}^\phi(f^4f^0f^1\hat{\partial}f^2\hat{\partial}f^3)- \hat{T}^\phi(f^4f^0\hat{\partial}f^1f^2\hat{\partial}f^3) + \hat{T}^\phi(f^4f^0\hat{\partial}f^1\hat{\partial}f^2f^3)\\
&= \hat{T}^\phi(f^0f^1f^2\hat{\partial}f^3\hat{\partial}f^4) +\hat{T}^\phi(f^0\hat{\partial}f^1(f^2f^3)\hat{\partial}f^4) + \hat{T}^\phi(f^0\hat{\partial}f^1\hat{\partial}f^2f^3f^4) \\
&= (S(\phi))(f^0\otimes f^1 \otimes f^2 \otimes f^3 \otimes f^4)
\end{array}
\end{equation*}
\end{proof}

\begin{theorem}\label{Periodthmo}
Let $\mathcal{C}$ be a small $\mathbb{C}$-category and let $(\mathscr{H},\mathcal{F})$ be a $p$-summable even Fredholm module over $\mathcal{C}$. Take $ 2m \geq p-1$. Then, 
\begin{equation*}
S(\phi^{2m})=-(m+1)\phi^{2m+2} \qquad \text{in}~ H^{2m+2}_\lambda(\mathcal{C})
\end{equation*}
\end{theorem}
\begin{proof}
We will show that $S(\phi^{2m})+(m+1)\phi^{2m+2}=b\psi$ for some $\psi \in Z^{2m+1}_\lambda(\mathcal{C})$.
By Theorem \ref{evencyc}, we know that $\phi^{2m}$ is the character of the $2m$-dimensional cycle $(\Omega' \mathcal{C}, \partial', \hat{Tr}_s,\rho')$ over the category $\mathcal{C}$. Applying Lemma \ref{Sform} and using the fact that $Tr_s(T)=0$ for any homogeneous $T$ of odd degree, we have
\begin{equation*}
\begin{array}{ll}
(S(\phi^{2m}))(f^0 \otimes f^1 \otimes \ldots \otimes f^{2m+2}) &= \sum\limits_{j=0}^{2m+1}  Tr_s\left(f^0[\mathcal F,f^1]\ldots [\mathcal F,f^{j-1}](f^jf^{j+1})[\mathcal F,f^{j+2}]\ldots [\mathcal F,f^{2m+2}] \right)
\end{array}
\end{equation*}
Further,
\begin{equation*}
\begin{array}{ll}
\phi^{2m+2}(f^0 \otimes f^1 \otimes \ldots \otimes f^{2m+2})=  Tr_s\left( f^0[\mathcal F,f^1]\ldots\ldots [\mathcal F,f^{2m+2}] \right)
\end{array}
\end{equation*}
so that
\begin{equation}\label{exp}
\begin{array}{lr}
&\left(S(\phi^{2m})+(m+1)\phi^{2m+2}\right)(f^0 \otimes f^1 \otimes \ldots \otimes f^{2m+2})\qquad \qquad \qquad \qquad \qquad \qquad \\
&= \sum\limits_{j=0}^{2m+1}  Tr_s\left( f^0[\mathcal F,f^1]\ldots [\mathcal F,f^{j-1}](f^jf^{j+1})[\mathcal F,f^{j+2}]\ldots [\mathcal F,f^{2m+2}] \right)\\
&  +(m+1)Tr_s\left(f^0[\mathcal F,f^1]\ldots\ldots [\mathcal F,f^{2m+2}] \right)\qquad \qquad \qquad \qquad \qquad
\end{array}
\end{equation}
We now consider $\psi=\sum\limits_{j=0}^{2m+1} (-1)^{j-1} \psi^j$, where
\begin{equation}\label{defbound}
\psi^j(f^0 \otimes f^1 \otimes \ldots \otimes f^{2m+1})=Tr\left(\epsilon \mathcal  F f^j [\mathcal F,f^{j+1}]\ldots [\mathcal F,f^{2m+1}][\mathcal F,f^0][\mathcal F,f^1]\ldots[\mathcal F,f^{j-1}] \right)
\end{equation}
Since $2m \geq p-1$ and $(\mathscr H,\mathcal F)$ is a $p$-summable even Fredholm module over $\mathcal C$,  it follows that the operator 
$\epsilon \mathcal  F f^j [\mathcal F,f^{j+1}]\ldots [\mathcal F,f^{2m+1}][\mathcal F,f^0][\mathcal F,f^1]\ldots[\mathcal F,f^{j-1}]$ is trace class. 

\smallskip
We observe that $\tau\psi^j=\psi^{j-1}$ for $1\leq j\leq 2m+1$ and $\tau \psi^0=\psi^{2m+1}$. It follows that  $(1-\lambda)(\psi)=0$. Hence, $\psi \in C^{2m+1}_\lambda(\mathcal{C})=Ker(1-\lambda)$. Using \eqref{defbound}, we have
\begin{equation*}
\begin{array}{ll}
&(b\psi^j)(f^0 \otimes f^1 \otimes \ldots \otimes f^{2m+2})\\
& \quad =\sum\limits_{i=0}^{2m+1} (-1)^i ~\psi^j(f^0 \otimes \ldots \otimes f^if^{i+1} \otimes \ldots \otimes f^{2m+2})+  \psi^j(f^{2m+2}f^0 \otimes f^1 \otimes \ldots \otimes f^{2m+1})\\
& \quad =Tr\left(\epsilon \mathcal Ff^{j+1}[\mathcal F,f^{j+2}]\ldots[\mathcal F,f^{2m+2}]f^0[\mathcal F,f^1]\ldots[\mathcal F,f^j]\right)~+\\& \qquad (-1)^{j-1}Tr\left(\epsilon \mathcal  Ff^{j+1}[\mathcal F,f^{j+2}]\ldots[\mathcal F,f^{2m+2}][\mathcal F,f^0][\mathcal F,f^1]\ldots f^j \right)+\\
& \qquad  Tr\left(\epsilon \mathcal Ff^{j}[\mathcal F,f^{j+1}]\ldots[\mathcal F,f^{2m+2}]f^0[\mathcal F,f^1]\ldots[\mathcal F,f^{j-1}]\right)
\end{array}
\end{equation*}
We now set $\beta^j=[\mathcal F,f^{j+2}]\ldots[\mathcal F,f^{2m+2}]f^0[\mathcal F,f^1]\ldots[\mathcal F,f^{j-1}]$.  Then,  we have
\begin{equation*}
\begin{array}{ll}
[\mathcal F,\beta^j]=\mathcal F\beta^j- (-1)^{2m}\beta^j \mathcal F&=\mathcal F[\mathcal F,f^{j+2}]\ldots[\mathcal F,f^{2m+2}]f^0[\mathcal F,f^1]\ldots[\mathcal F,f^{j-1}]-\\
&\qquad [\mathcal F,f^{j+2}]\ldots[\mathcal F,f^{2m+2}]f^0[\mathcal F,f^1]\ldots[\mathcal F,f^{j-1}]\mathcal F\\
&= (-1)^{j-1}[\mathcal F,f^{j+2}]\ldots [\mathcal F,f^{2m+2}][\mathcal F,f^0][\mathcal F,f^1]\ldots[\mathcal F,f^{j-1}]
\end{array}
\end{equation*}
With $\alpha^j=f^j \mathcal Ff^{j+1}$, we get
\begin{equation}
\begin{array}{ll}
&(-1)^{j-1}Tr\left(\epsilon \mathcal Ff^{j+1}[\mathcal F,f^{j+2}]\ldots[\mathcal F,f^{2m+2}][\mathcal F,f^0][\mathcal F,f^1]\ldots [\mathcal F,f^{j-1}]f^j \right)\\
& \quad =Tr\left(\epsilon \mathcal Ff^{j+1}[\mathcal F,\beta^j]f^j\right)=Tr_s\left(\alpha^j[\mathcal F,\beta^j]\right)=Tr_s([\mathcal F,\alpha^j]\beta^j)
\end{array}
\end{equation} where we have used the fact that $Tr_s$ is a closed graded trace and $Tr_s(T)=Tr(\epsilon T)$ for any operator that is trace class (see \cite[Lemma 2]{C2}).
Thus, we have
\begin{equation*}
\begin{array}{ll}
&(b\psi^j)(f^0 \otimes f^1 \otimes \ldots \otimes f^{2m+2})=-Tr_s\left([\mathcal F,f^j] \mathcal  Ff^{j+1}\beta^j\right)+Tr_s([\mathcal F,\alpha^j]\beta^j)+Tr_s\left( \mathcal Ff^j[\mathcal F,f^{j+1}]\beta^j\right)
\end{array}
\end{equation*}
Since
\begin{equation*}
\mathcal F[\mathcal F,f^jf^{j+1}]=\mathcal F[\mathcal F,f^j]f^{j+1}+\mathcal Ff^j[\mathcal F,f^{j+1}]=-[\mathcal F,f^j]\mathcal Ff^{j+1}+ \mathcal Ff^j[\mathcal F,f^{j+1}],
\end{equation*}
we obtain 
\begin{equation*}
\begin{array}{ll}
(b\psi^j)(f^0 \otimes f^1 \otimes \ldots \otimes f^{2m+2})&=Tr_s\big(\left(\mathcal F[\mathcal F,f^jf^{j+1}] + [\mathcal F,\alpha^j]\right)\beta^j\Big)
\end{array}
\end{equation*}

As 
\begin{equation*}
\mathcal F[\mathcal F,f^jf^{j+1}] +[\mathcal F,\alpha^j]=\mathcal F[\mathcal F,f^jf^{j+1}] +\mathcal F \alpha^j + \alpha^j \mathcal F =[\mathcal F,f^j][\mathcal F,f^{j+1}]+2f^jf^{j+1}
\end{equation*}
 we get
\begin{equation*}
\begin{array}{ll}
(b\psi)(f^0 \otimes f^1 \otimes \ldots \otimes f^{2m+2})&=\sum\limits_{j=0}^{2m+1} (-1)^{j-1} (b\psi^j) (f^0 \otimes f^1 \otimes \ldots \otimes f^{2m+2})\\
&=\sum\limits_{j=0}^{2m+1}(-1)^{j-1} \big( 2~Tr_s\left(f^jf^{j+1}\beta^j\right) + Tr_s\left([\mathcal F,f^j][\mathcal F,f^{j+1}]\beta^j\right) \big)\\
&=\sum\limits_{j=0}^{2m+1} 2~Tr_s \left(f^0[\mathcal F,f^1]\ldots [\mathcal F,f^{j-1}](f^jf^{j+1})[\mathcal F,f^{j+2}]\ldots [\mathcal F,f^{2m+2}] \right)\\
& \quad +\sum\limits_{j=0}^{2m+1}Tr_s\left(f^0[\mathcal F,f^1] \ldots  [\mathcal F,f^{2m+2}]\right)\\
&=\sum\limits_{j=0}^{2m+1}~2~Tr_s\left(f^0[\mathcal F,f^1]\ldots [\mathcal F,f^{j-1}](f^jf^{j+1})[\mathcal F,f^{j+2}]\ldots [\mathcal F,f^{2m+2}] \right)\\
& \quad +(2m+2)Tr\left(f^0[\mathcal F,f^1] \ldots  [\mathcal F,f^{2m+2}]\right)
\end{array}
\end{equation*}
The result now follows by \eqref{exp}.
\end{proof}

\section{Homotopy invariance of the Chern character}

Let $SHilb_2$ be the full subcategory of $SHilb_{\mathbb{Z}_2}$ whose objects are of the form $\mathcal D=\mathcal H\oplus \mathcal H$, for some separable Hilbert space $\mathcal H$. If $\mathcal D=\mathcal H\oplus \mathcal H\in SHilb_2$, we denote by $F(\mathcal D)$ the morphism  in $SHilb_2(\mathcal D,\mathcal D)= \mathcal{B}(\mathcal H \oplus \mathcal H, \mathcal H \oplus\mathcal  H)$ given by the matrix 
$\begin{pmatrix} 0 & 1\\ 1 & 0 \\ \end{pmatrix}$ swapping the two copies of $\mathcal H$.   

\begin{lemma}\label{5.2x}
Let $\mathcal C$ be a small $\mathbb{C}$-category and $\{\mathscr{H}_t:\mathcal{C} \longrightarrow SHilb_2\}_{t \in [0,1]}$ be a family of functors such that
for each $X\in Ob(\mathcal C)$, we have $\mathscr H_t(X)=\mathscr H_{t'}(X)$ for all $t$, $t'\in [0,1]$. We put $\mathscr H(X):=\mathscr H_t(X)$ for all $t \in [0,1]$.
For each $f:X \longrightarrow Y$ in $\mathcal{C}$, we assume that the function
\begin{equation*} p_f:[0,1] \longrightarrow SHilb_{\mathbb{Z}_2}({\mathscr H}_t(X),{\mathscr H}_t(Y))\qquad t \mapsto {\mathscr H}_t
(f)\end{equation*} is strongly $C^1$. Then if $\delta_t(f):=p_f'(t)$, we have
\begin{equation*}
\delta_t(fg)={\mathscr H}_t(f) \circ \delta_t(g) + \delta_t(f) \circ {\mathscr H}_t
(g)
\end{equation*} for composable morphisms $f$, $g$ in $\mathcal C$.
\end{lemma}
\begin{proof} We have
\begin{equation*}
\begin{array}{ll}
&\delta_t(fg)-{\mathscr H}_t(f) \circ \delta_t(g) - \delta_t(f) \circ {\mathscr H}_t
(g)\\
& \quad =p'_{fg}(t)-{\mathscr H}_t(f) \circ p'_g(t) -p'_f(t)  \circ {\mathscr H}_t
(g)\\
& \quad =\lim\limits_{s \to 0} \frac{1}{s} \left(p_{fg}(t+s)-p_{fg}(t) - {\mathscr H}_t(f)~ \circ ~p_g(t+s)+  {\mathscr H}_t(f)~ \circ~ p_g(t)- p_f(t+s)~ \circ~ {\mathscr H}_t
(g)+p_f(t)~ \circ~ {\mathscr H}_t(g)\right)\\
& \quad =\lim\limits_{s \to 0} \frac{1}{s}\left({\mathscr H}_{t+s}(fg)   -{\mathscr H}_{t}(fg)  - {\mathscr H}_t(f){\mathscr H}_{t+s}(g) +  {\mathscr H}_t(f) {\mathscr H}_{t}(g) - {\mathscr H}_{t+s}(f) {\mathscr H}_t
(g)+{\mathscr H}_{t}(f) {\mathscr H}_t(g)\right)\\
& \quad =\lim\limits_{s \to 0} \frac{1}{s}\left({\mathscr H}_{t+s}(f)-{\mathscr H}_{t}(f)\right)\left({\mathscr H}_{t+s}(g)-{\mathscr H}_{t}(g)\right)\\
& \quad =\lim\limits_{s \to 0}\frac{1}{s} \left(p_f(t+s)-p_f(t)\right) \left(p_g(t+s)-p_g(t)\right)\\
& \quad = p_f'(t) \lim\limits_{s \to 0} \left(p_g(t+s)-p_g(t)\right)=0 
\end{array}
\end{equation*}
\end{proof}

For each $n \in \mathbb{Z}_{\geq 0}$, we now define an operator $A:CN^{n}(\mathcal{C}) \longrightarrow CN^n(\mathcal{C})$ given by
\begin{equation*}
A:=1+\lambda+\lambda^2+\ldots+\lambda^n
\end{equation*} where $\lambda$ is the (signed) cyclic operator. 
We observe that if $\psi \in C^n_\lambda(\mathcal{C})=Ker(1-\lambda)$, then $A\psi=(n+1)\psi$. From the relation
\begin{equation*}
(1-\lambda)(1+2\lambda+3\lambda^2+\cdots+(n+1)\lambda^n)=A-(n+1)\cdot 1
\end{equation*}
it is immediate that $Ker(A) \subseteq Im(1-\lambda)$.  Let $B_0: CN^{n+1}(\mathcal{C}) \longrightarrow CN^n(\mathcal{C})$ be the map defined as follows:
\begin{equation*}
(B_0\phi)(f^0 \otimes \ldots \otimes f^n):=\phi(id_{X_0} \otimes f^0 \otimes \ldots \otimes f^n)-(-1)^{n+1}\phi(f^0 \otimes \ldots \otimes f^n \otimes id_{X_0})
\end{equation*}
for any $f^0 \otimes f^1 \otimes \ldots \otimes f^{n} \in Hom_{\mathcal{C}}(X_1,X_0) \otimes Hom_{\mathcal{C}}(X_2,X_1) \otimes \ldots \otimes Hom_{\mathcal{C}}(X_0,X_{n})$.  We now set
\begin{equation*}
B:=AB_0:  CN^{n+1}(\mathcal{C}) \longrightarrow CN^n(\mathcal{C})
\end{equation*}

\begin{lemma}\label{5.3g} We have

(1) $bA=Ab'$.

\smallskip
(2) $bB+Bb=0$.
\end{lemma}
\begin{proof}
(1) This follows from the general fact that the dual $CN^\bullet(\mathcal C)$  of the cyclic nerve of $\mathcal C$ is a cocyclic module (see, for instance, \cite[$\S$ 2.5]{Loday}).

\smallskip
(2) For any $f^0 \otimes f^1 \otimes \ldots \otimes f^{n} \in Hom_{\mathcal{C}}(X_1,X_0) \otimes Hom_{\mathcal{C}}(X_2,X_1) \otimes \ldots \otimes Hom_{\mathcal{C}}(X_0,X_{n})$ and $\phi \in CN^n{\mathcal{C}}$, we have 
\begin{equation*}
\begin{array}{lll}
&(B_0b\phi)(f^0\otimes \ldots \otimes f^{n})\\
&\quad =(b\phi)(id_{X_0} \otimes f^0 \otimes \ldots \otimes f^n)-(-1)^{n+1}(b\phi)(f^0 \otimes \ldots \otimes f^n \otimes id_{X_0})\\
& \quad= \phi(f^0 \otimes \ldots \otimes f^{n}) + \sum\limits_{i=0}^{n-1}(-1)^{i+1} \phi(id_{X_0} \otimes f^0 \otimes \ldots \otimes f^if^{i+1} \otimes \ldots \otimes f^{n}) +(-1)^{n+1}  \phi(f^n \otimes f^0 \otimes \ldots \otimes f^{n-1}) \\
& \qquad -(-1)^{n+1}\left(\sum\limits_{i=0}^{n-1}(-1)^{i} \phi(f^0 \otimes \ldots \otimes f^if^{i+1} \otimes \ldots \otimes f^{n} \otimes id_{X_0})  \right)
\end{array}
\end{equation*}
On the other hand,
\begin{equation*}
\begin{array}{lll}
&(b'B_0\phi)(f^0\otimes \ldots \otimes f^{n})\\
& \quad =\sum\limits_{i=0}^{n-1}(-1)^{i} \phi(id_{X_0} \otimes f^0 \otimes \ldots \otimes f^if^{i+1} \otimes \ldots \otimes f^{n})-(-1)^n\sum\limits_{i=0}^{n-1}(-1)^{i} \phi(f^0 \otimes \ldots \otimes f^if^{i+1} \otimes \ldots \otimes f^{n} \otimes id_{X_0})
\end{array}
\end{equation*}
Thus, we obtain
\begin{equation*}
(B_0b+b'B_0)(\phi)(f^0\otimes \ldots \otimes f^{n})=\phi(f^0 \otimes \ldots \otimes f^{n})+(-1)^{n+1} \phi(f^n \otimes f^0 \otimes \ldots \otimes f^{n-1})
\end{equation*}
Therefore,
\begin{equation}\label{5.1ln}
(B_0b+b'B_0)(\phi)=\phi-\lambda \phi
\end{equation}
Now, by applying the operator $A$ to both sides of \eqref{5.1ln}, we have
\begin{equation*}
AB_0b+Ab'B_0=0
\end{equation*}
The result now follows from part \it(1).
\end{proof}

\begin{prop}\label{Prop4.3gf}
The image of the map $B:CN^{n+1}(\mathcal{C}) \longrightarrow CN^{n}(\mathcal{C})$ is
$C^n_\lambda(\mathcal{C})$.
\end{prop}
\begin{proof}
Let $\phi \in C^n_\lambda(\mathcal{C})$ and let $R:=\bigoplus\limits_{X,Y \in Ob(\mathcal{C})}Hom(X,Y)$. Then $A$ is an algebra with mutiplication given by composition wherever possible and $0$ otherwise. We choose a linear map $\eta: A \longrightarrow \mathbb{C}$ such that 
\begin{equation*}
\begin{array}{ll}
\eta(f)=0 &\qquad \text{for}~ f \in Hom_\mathcal{C}(X,Y), ~X \neq Y\\
\eta(id_X)=1 & \qquad \forall X \in Ob(\mathcal{C})
\end{array}
\end{equation*}

We now define $\psi \in CN^{n+1}(\mathcal{C})$ by setting
\begin{equation*}
\begin{array}{ll}
\psi(f^0\otimes \ldots \otimes f^{n+1}):=&\eta(f^0)\phi(f^1\otimes \ldots \otimes f^{n+1})+\\
&\quad (-1)^n \left(\phi\left(f^0 \otimes f^1 \otimes \ldots \otimes f^{n}\right)\eta(f^{n+1}) -\eta(f^0)\phi\left(id_{X_1} \otimes f^1 \otimes \ldots \otimes f^{n}\right)\eta(f^{n+1})\right)
\end{array}
\end{equation*}
for any $f^0 \otimes f^1 \otimes \ldots \otimes f^{n+1} \in Hom_{\mathcal{C}}(X_1,X_0) \otimes Hom_{\mathcal{C}}(X_2,X_1) \otimes \ldots \otimes Hom_{\mathcal{C}}(X_0,X_{n+1})$. 
We observe that if the tuple $(f^1, \ldots, f^{n+1})$ is not cyclically composable, i.e., $X_0 \neq X_1$, then the first term vanishes as $\eta(f^0)=0$. Similarly, if the tuple $(f^0, \ldots, f^{n})$ is not cyclically composable, i.e., $X_{n+1} \neq X_0$, then the second term vanishes. For the last term, $\eta(f^0)$ and $\eta(f^{n+1})$ will be non zero only if $X_1=X_0$ and $X_0=X_{n+1}$ which means that $X_{n+1}=X_1$ and the tuple $(id_{X_1},f^1, \ldots, f^{n})$ is cyclically composable.

Then, for   any $g^0 \otimes g^1 \otimes \ldots \otimes g^{n} \in Hom_{\mathcal{C}}(Y_1,Y_0) \otimes Hom_{\mathcal{C}}(Y_2,Y_1) \otimes \ldots \otimes Hom_{\mathcal{C}}(Y_0,Y_{n})$,
we have
\begin{equation*}
\begin{array}{ll}
\psi(id_{Y_0} \otimes g^0 \otimes \ldots \otimes g^{n})&=\eta(id_{Y_0})\phi(g^0\otimes \ldots \otimes g^{n})+(-1)^n \big(\phi\left(id_{Y_0} \otimes g^0 \otimes \ldots \otimes g^{n-1}\right)\eta(g^{n}) \\
& \quad -\phi\left(\eta(id_{Y_0})id_{Y_0} \otimes g^0 \otimes \ldots \otimes g^{n-1}\right)\eta(g^{n})\big)\\
&=\phi(g^0\otimes \ldots \otimes g^{n})
\end{array}
\end{equation*}
Also
\begin{equation*}
\begin{array}{ll}
\psi(g^0 \otimes \ldots \otimes g^{n} \otimes id_{Y_0})&=\eta(g^0)\phi(g^1\otimes \ldots \otimes g^{n} \otimes id_{Y_0})+(-1)^n \big(\phi\left(g^0 \otimes \ldots \otimes g^{n}\right)\eta(id_{Y_0})\\
&\quad -\phi\left(\eta(g^0)id_{Y_1} \otimes g^1 \otimes \ldots \otimes g^{n}\right)\eta(id_{Y_0})\big)\\
& =(-1)^n\phi\left(g^0 \otimes \ldots \otimes g^{n}\right) \\
\end{array}
\end{equation*} where the second equality follows from the fact that $\phi\in C_\lambda^n(\mathcal C)$ and that $\eta(g^0)=0$ whenever $Y_1\ne Y_0$. 
Thus,
\begin{equation*}
\begin{array}{ll}
(B_0\psi)(g^0 \otimes \ldots \otimes g^n)&=\psi(id_{Y_0} \otimes g^0 \otimes \ldots \otimes g^n)-(-1)^{n+1}\psi(g^0 \otimes \ldots \otimes g^n \otimes id_{Y_0})\\
&=2\phi(g^0 \otimes \ldots \otimes g^n)
\end{array}
\end{equation*}
Since $\phi \in Ker(1-\lambda)$, we now have $B\psi=2A\phi=2(n+1)\phi$. Thus, $\phi \in Im(B)$. Conversely, let $\phi \in Im(B)$. Then, $\phi=B\psi$ for some $\psi \in CN^{n+1}(\mathcal{C})$. Using the fact that $(1-\lambda)A=0$, we have
\begin{equation*}
\begin{array}{c}
(1-\lambda)(\phi)=(1-\lambda)(B\psi)=((1-\lambda)AB_0)\psi=0\\
\end{array} 
\end{equation*}
This proves the result.
\end{proof}

\begin{prop}\label{prop5.5h}
Let $\psi \in CN^n(\mathcal{C})$ be such that $b\psi \in C^{n+1}_\lambda(\mathcal{C})$. Then,

\smallskip
(1) $B\psi \in Z^{n-1}_\lambda(\mathcal{C})$ i.e., $b(B\psi)=0$ and $(1-\lambda)(B\psi)=0$.

\smallskip
(2) $S(B\psi)=n(n+1)b\psi$ in $H^{n+1}_\lambda(\mathcal{C})$.
\end{prop}
\begin{proof}
{\it(1)} We know that $(1-\lambda)(B \psi)=(1-\lambda)(AB_0)(\psi)=0$. Further, for any $\phi \in Ker(1-\lambda)$, we have $B_0\phi=0$. Therefore, it follows that $bB\psi=-Bb\psi=-AB_0b\psi=0$.

\smallskip
{\it(2)} We have to show that $SB\psi-n(n+1)b\psi=b\zeta$ for some $\zeta \in C^{n}_\lambda(\mathcal{C})$. We set $\phi=B\psi$. Then, $\phi$ is the character of an $(n-1)$-dimensional cycle $(\mathcal{S},\hat{\partial},\hat{T},\rho)$ over $\mathcal{C}$.  By Proposition \ref{Sbound}, we have
$S\phi=b\psi'$, where $\psi' \in CN^n(\mathcal{C})$ is given by
\begin{equation*}
\psi'(f^0 \otimes \ldots \otimes f^{n})=\sum\limits_{j=1}^n (-1)^{j-1}~ \hat{T}\left(f^0\hat{\partial}f^1 \ldots  \hat{\partial}f^{j-1} f^{j} \hat{\partial}f^{j+1} \ldots \hat{\partial}f^n\right)
\end{equation*}
Suppose we have $\psi'' \in CN^n(\mathcal{C})$ such that $\psi''-\psi \in B^n(\mathcal{C})$ and   $\zeta=\psi'-n(n+1)\psi'' \in C^{n}_\lambda (\mathcal{C})$. This would give
\begin{equation*}
b\zeta=b\psi'-n(n+1)b \psi''=SB\psi-n(n+1)b\psi
\end{equation*}
We set $\theta:=B_0\psi$, $\theta':=\frac{1}{n}\phi$ and $\theta'':=\theta-\theta'\in CN^{n-1}(\mathcal C)$. Since $B\psi \in Z^{n-1}_\lambda(\mathcal{C})$, we have
\begin{equation*}
A\theta''=AB_0\psi-\frac{1}{n}A\phi=B\psi-\frac{1}{n}AB\psi=B\psi-\frac{1}{n}nB\psi=0
\end{equation*}
Since $Ker(A) \subseteq Im(1-\lambda)$, we have $\theta''=(1-\lambda)(\psi_1)$ for some $\psi_1 \in CN^{n-1}(\mathcal{C})$. We take $\psi''=\psi-b\psi_1$. We now show that $(1-\lambda)(\zeta)=0$, i.e, $(1-\lambda)(\psi')=n(n+1)(1-\lambda)(\psi'')$ where $\zeta=\psi'-n(n+1)\psi''$.
We see that
\begin{equation*}
(\tau_n\psi')(f^0 \otimes \ldots \otimes f^{n})=\psi'(f^n \otimes f^0 \otimes \ldots \otimes f^{n-1})=\sum\limits_{j=0}^{n-1} (-1)^{j}~ \hat{T}\left(\hat{\partial}f^0\hat{\partial}f^1 \ldots  \hat{\partial}f^{j-1} f^{j} \hat{\partial}f^{j+1} \ldots \hat{\partial}f^{n-1}f^n\right)
\end{equation*} where we have used the fact that $\hat{T}$ is a graded trace. 
For $1\leq j\leq n-1$, we now set
\begin{equation*}
\omega_j:=f^0(\hat{\partial}f^1 \ldots  \hat{\partial}f^{j-1}) f^{j} (\hat{\partial}f^{j+1} \ldots \hat{\partial}f^{n-1})f^n
\end{equation*}
Then,
\begin{equation*}
\begin{array}{ll}
\hat{\partial}\omega_j&=(\hat{\partial}f^0\hat{\partial}f^1 \ldots  \hat{\partial}f^{j-1}) f^{j} (\hat{\partial}f^{j+1} \ldots \hat{\partial}f^{n-1})f^n+(-1)^{j-1}f^0(\hat{\partial}f^1 \ldots  \hat{\partial}f^{j-1} \hat{\partial}f^{j} \hat{\partial}f^{j+1} \ldots \hat{\partial}f^{n-1})f^n +\\
&\quad (-1)^n f^0(\hat{\partial}f^1 \ldots  \hat{\partial}f^{j-1}) f^{j} (\hat{\partial}f^{j+1} \ldots \hat{\partial}f^{n-1} \hat{\partial}f^n)
\end{array}
\end{equation*}
Thus,
\begin{equation*}
\begin{array}{ll}
0=\hat{T}(\hat{\partial}\omega_j)&=\hat{T}\left(\hat{\partial}f^0\hat{\partial}f^1 \ldots  \hat{\partial}f^{j-1} f^{j} \hat{\partial}f^{j+1} \ldots \hat{\partial}f^{n-1}f^n\right)+(-1)^{j-1}\hat{T}\left(f^0\hat{\partial}f^1 \ldots  \hat{\partial}f^{j-1} \hat{\partial}f^{j} \hat{\partial}f^{j+1} \ldots \hat{\partial}f^{n-1}f^n\right) +\\
&\quad (-1)^n \hat{T}\left(f^0\hat{\partial}f^1 \ldots  \hat{\partial}f^{j-1} f^{j} \hat{\partial}f^{j+1} \ldots \hat{\partial}f^{n-1} \hat{\partial}f^n\right)
\end{array}
\end{equation*}
Therefore,
\begin{equation*}
\begin{array}{ll}
&(1-\lambda)(\psi')(f^0 \otimes \ldots \otimes f^{n})\\
&=-\sum\limits_{j=1}^n (-1)^{j}~ \hat{T}\left(f^0\hat{\partial}f^1 \ldots  \hat{\partial}f^{j-1} f^{j} \hat{\partial}f^{j+1} \ldots \hat{\partial}f^n\right)-
(-1)^n \sum\limits_{j=0}^{n-1} (-1)^{j}~ \hat{T}\left(\hat{\partial}f^0\hat{\partial}f^1 \ldots  \hat{\partial}f^{j-1} f^{j} \hat{\partial}f^{j+1} \ldots \hat{\partial}f^{n-1}f^n\right)\\
&=-\big((-1)^n\hat{T}\left(f^0\hat{\partial}f^1 \ldots   \hat{\partial}f^{n-1}f^n\right)+(-1)^n \hat{T}\left(f^0\hat{\partial}f^1 \ldots  \hat{\partial}f^{n-1}f^n\right) +\\
&\quad \sum\limits_{j=1}^{n-1} (-1)^{j}~\left(\hat{T}\left(f^0\hat{\partial}f^1 \ldots  \hat{\partial}f^{j-1} f^{j} \hat{\partial}f^{j+1} \ldots \hat{\partial}f^n\right) +(-1)^n \hat{T}\left(\hat{\partial}f^0\hat{\partial}f^1 \ldots  \hat{\partial}f^{j-1} f^{j} \hat{\partial}f^{j+1} \ldots \hat{\partial}f^{n-1}f^n\right)\right)\big)\\
&=(-1)^{n+1}~(n+1)\hat{T}(f^nf^0\hat{\partial}f^1 \ldots   \hat{\partial}f^{n-1})=(-1)^{n+1} ~(n+1)\phi(f^nf^0\otimes f^1 \otimes \ldots  \otimes f^{n-1})
\end{array}
\end{equation*}
Hence,
\begin{equation}\label{1c}
(1-\lambda)(\psi')(f^0 \otimes \ldots \otimes f^{n})=(-1)^{n+1} ~(n+1)\phi(f^nf^0\otimes f^1 \otimes \ldots  \otimes f^{n-1})
\end{equation}
On the other hand, using  the definition of $\psi''$ and the fact that $(1-\lambda)b=b'(1-\lambda)$, we have
\begin{equation*}
(1-\lambda)(\psi'')=(1-\lambda)(\psi)-(1-\lambda)(b\psi_1)=(1-\lambda)(\psi)-b'(1-\lambda)(\psi_1)=(1-\lambda)(\psi)-b'\theta''
\end{equation*}
Since $b\psi \in C^{n+1}_\lambda(\mathcal{C})$, we have from \eqref{5.1ln} that $(1-\lambda)(\psi)=(B_0b+b'B_0)(\psi)=b'B_0\psi=b'\theta=b'\theta'+b'\theta''$. Hence, 
\begin{equation*}
(1-\lambda)(\psi'')=b'\theta'=\frac{1}{n}~b'\phi
\end{equation*}
Since $\phi=B\psi \in Z^{n-1}_\lambda(\mathcal{C})$, $b\phi=0$ and therefore
\begin{equation}\label{1d}
(1-\lambda)(\psi'')(f^0 \otimes \ldots \otimes f^{n})=\frac{1}{n}~(b'\phi)(f^0 \otimes \ldots \otimes f^{n})= \frac{1}{n} ~(-1)^{n-1}\phi(f^nf^0 \otimes f^1 \otimes \ldots \otimes f^{n-1})
\end{equation}
The result now follows by comparing \eqref{1c} and \eqref{1d}.
\end{proof}

\begin{prop}\label{Conlemma}
Let $\mathcal C$ be a small $\mathbb{C}$-category  and $\{\mathscr{H}_t:\mathcal{C} \longrightarrow SHilb_2\}_{t \in [0,1]}$ be a family of functors such that
for each $X\in Ob(\mathcal C)$, we have $\mathscr H_t(X)=\mathscr H_{t'}(X)$ for all $t$, $t'\in [0,1]$ and $\mathscr H_t(f)$ is of degree zero
for each $f \in Hom_\mathcal{C}(X,Y)$ and $t\in [0,1]$. We put $\mathscr H(X):=\mathscr H_t(X)$ for all $t \in [0,1]$.

\smallskip
Let $\mathcal F$ be the family of operators 
\begin{equation}
\mathcal F=\left\{(F(\mathscr H(X))=\begin{pmatrix} 0 & 1 \\ 1 & 0 \\
\end{pmatrix}\right\}_{X\in Ob(\mathcal C)}
\end{equation}

Let $p=2m$ be an even integer. We assume that 

(1)  for each $f \in Hom_\mathcal{C}(X,Y)$, the association  $ t \mapsto [\mathcal{F},{\mathscr H}_t(f)]$ is a continuous map
\begin{equation*}
\zeta_f:[0,1] \longrightarrow \mathcal{B}^p(\mathscr H(X), \mathscr H(Y)) \qquad t \mapsto [\mathcal{F},{\mathscr H}_t(f)]
\end{equation*}

\smallskip
(2) for each $f \in Hom_\mathcal{C}(X,Y)$, the association 
\begin{equation*}
p_f:[0,1] \longrightarrow SHilb_{\mathbb{Z}_2}({\mathscr H}_t(X),{\mathscr H}_t(Y))\qquad t \mapsto {\mathscr H}_t(f)
\end{equation*}
is piecewise strongly $C^1$. 

\smallskip
Let $({\mathscr H}_t,\mathcal{F})$ be the corresponding $p$-summable Fredholm modules over $\mathcal{C}$. Then, the class in $H^{p+2}_\lambda(\mathcal{C})$ of the $(p+2)$-dimensional character of the Fredholm module $({\mathscr H}_t,\mathcal{F})$ is independent of $t$.
\end{prop}

\begin{proof}
For any $t \in [0,1]$, let $\phi_{t}$ be the $p$-dimensional character of the Fredholm module $({\mathscr H}_{t},\mathcal{F})$. We will show that $S(\phi_{t_1})=S(\phi_{t_2})$ for any $t_1, t_2 \in [0,1]$.

\smallskip
By assumption, we know that there exists a finite set $R=\{0= r_0\leq r_1 < ... <r_k\leq r_{k+1}=1\}\subseteq [0,1]$ such that $p_f:[0,1] \longrightarrow SHilb_{\mathbb{Z}_2}({\mathscr H}_t(X),{\mathscr H}_t(Y))$ is continuously differentiable in each $[r_i,r_{i+1}]$. By abuse of notation,
we set for each $f \in Hom_\mathcal{C}(X,Y)$:
\begin{equation}
\delta_t(f):=p_f'(t)\in SHilb_{\mathbb{Z}_2}({\mathscr H}_t(X),{\mathscr H}_t(Y))
\end{equation} Here, it is understood that if $t=r_i$ for some $1\leq i\leq k$, we use the right hand derivative when $r_i$ is treated as a point
of $[r_i,r_{i+1}]$ and the left hand derivative when $r_i$ is treated as a point of $[r_{i-1},r_i]$. 

\smallskip
Using Lemma \ref{5.2x},  we know that 
\begin{equation}\label{nv2}
\delta_t(fg)={\mathscr H}_t(f) \circ \delta_t(g) + \delta_t(f) \circ {\mathscr H}_t
(g)
\end{equation}
for any $t \in [0,1]$ and for any pair of composable morphisms $f$ and $g$ in $\mathcal{C}$.

\smallskip
For any $t \in [0,1]$ and $1 \leq j \leq p+1$, we set
\begin{equation*}
\psi_t^j(f^0 \otimes \ldots \otimes f^{p+1}):=Tr\left(\epsilon {\mathscr H}_t(f^0)[\mathcal{F}, {\mathscr H}_t(f^1)]\ldots [\mathcal{F}, {\mathscr H}_t(f^{j-1})]\delta_t(f^j)[\mathcal{F}, {\mathscr H}_t(f^{j+1})] \ldots [\mathcal{F}, {\mathscr H}_t(f^{p+1})]\right)
\end{equation*}
Using the expression in \eqref{nv2} and the fact that $\epsilon {\mathscr H}(f)={\mathscr H}(f)\epsilon$ for any morphism $f \in \mathcal{C}$,  it may be easily verified that $b\psi_t^j=0$. For example, when $j=1$,  we have (suppressing the  functor ${\mathscr H}$ )
\begin{equation*}
\begin{array}{ll}
&(b\psi_t^1)(f^0 \otimes \ldots \otimes f^{p+2})\\
& \quad =\sum\limits_{i=0}^{p+1} \psi_t^1(f^0 \otimes \ldots f^if^{i+1} \otimes \ldots \otimes f^{p+2}) + \psi_t^j(f^{p+2}f^0 \otimes f^1 \otimes \ldots \otimes f^{p+2})\\
& \quad = Tr\left(\epsilon f^0f^1\delta_t(f^2)[\mathcal{F},f^3]\ldots[\mathcal{F},f^{p+2}]\right)-Tr\left(\epsilon f^0 \delta_t(f^1f^2)[\mathcal{F},f^3]\ldots[\mathcal{F},f^{p+2}]\right)\\
& \qquad + Tr\left(\epsilon f^0 \delta_t(f^1)[\mathcal{F},f^2f^3]\ldots[\mathcal{F},f^{p+2}]\right)-Tr\left(\epsilon f^0 \delta_t(f^1)[\mathcal{F},f^2][\mathcal{F},f^3f^4]\ldots[\mathcal{F},f^{p+2}]\right) + \ldots\\
& \qquad \ldots - Tr\left(\epsilon f^0 \delta_t(f^1)[\mathcal{F},f^2]\ldots[\mathcal{F},f^{p+1}f^{p+2}]\right) + Tr\left(\epsilon f^{p+2}f^0 \delta_t(f^1)[\mathcal{F},f^2][\mathcal{F},f^3f^4]\ldots[\mathcal{F},f^{p+1}]\right) \\
&\quad =0
\end{array}
\end{equation*}
 We then define
\begin{equation*}
\begin{array}{ll}
\psi_t:=\sum\limits_{j=0}^{p+1}  (-1)^{j-1} \psi_t^j
\end{array}
\end{equation*}
We have  $b\psi_t=0$.

For fixed $f$, it follows from the compactness of $[0,1]$ and the assumptions (1) and (2) that the families $\{ {\mathscr H}_t(f) \}_{t \in [0,1]}$, 
$\{p_f(t)\}_{t\in [0,1]}$ and 
$\{\delta_t(f)\}_{t \in [0,1]}$ are uniformly bounded. For the sake of simplicity, we assume that there is only a single point $r\in R$ such that 
$t_1\leq r\leq t_2$. Then, we form  $\psi \in CN^{p+1}(\mathcal{C})$ by setting
\begin{equation*}
\psi(f^0 \otimes \ldots \otimes f^{p+1}):=\int_{t_1}^{r}\psi_t(f^0 \otimes \ldots \otimes f^{p+1})dt+\int_{r}^{t_2}\psi_t(f^0 \otimes \ldots \otimes f^{p+1})dt
\end{equation*}
We now have
\begin{equation*}
\begin{array}{ll}
&\psi(id_{X_0} \otimes f^0 \otimes \ldots \otimes f^{p})\\
&=\int_{t_1}^{r}\psi_t(id_{X_0} \otimes f^0 \otimes \ldots \otimes f^{p})dt+\int_{r}^{t_2}\psi_t(id_{X_0} \otimes f^0 \otimes \ldots \otimes f^{p})dt\\
&=\int_{t_1}^{r} \big(\sum\limits_{j=0}^{p}(-1)^{j}Tr\left(\epsilon[\mathcal{F}, {\mathscr H}_t(f^0)]\ldots [\mathcal{F}, {\mathscr H}_t(f^{j-1})]\delta_t(f^j)[\mathcal{F}, {\mathscr H}_t(f^{j+1})] \ldots [\mathcal{F}, {\mathscr H}_t(f^{p})]\right)\big)dt\\&\textrm{ }+\int_{r}^{t_2} \big(\sum\limits_{j=0}^{p}(-1)^{j}Tr\left(\epsilon[\mathcal{F}, {\mathscr H}_t(f^0)]\ldots [\mathcal{F}, {\mathscr H}_t(f^{j-1})]\delta_t(f^j)[\mathcal{F}, {\mathscr H}_t(f^{j+1})] \ldots [\mathcal{F}, {\mathscr H}_t(f^{p})]\right)\big)dt\\
\end{array}
\end{equation*}
Let $\phi:[0,1] \longrightarrow Z^p_\lambda(\mathcal{C})$ be the map given by $t \mapsto \phi_t$. We now claim that 
\begin{equation*}
\psi(id_{X_0} \otimes f^0 \otimes \ldots \otimes f^{p})=\int_{t_1}^{r}\phi'(t)(f^0 \otimes \ldots \otimes f^p) ~dt+\int_{r}^{t_2}\phi'(t)(f^0 \otimes \ldots \otimes f^p) ~dt
\end{equation*}
 Indeed, we have
\begin{equation*}
\begin{array}{ll}
\phi'(t)(f^0 \otimes \ldots \otimes f^p)&=\lim\limits_{s \to 0}\frac{1}{s}(\phi_{t+s}-\phi_t)(f^0 \otimes \ldots \otimes f^p)\\
&=\lim\limits_{s \to 0}\big(Tr\big(\epsilon \frac{1}{s}\left( {\mathscr H}_{t+s}(f^0)-{\mathscr H}_t(f^0)\right)  [\mathcal{F}, {\mathscr H}_{t+s}(f^{1})] \ldots  [\mathcal{F}, {\mathscr H}_{t+s}(f^{p})]\big)\\
& \quad +Tr\big(\epsilon {\mathscr H}_{t}(f^0) [\mathcal{F},\frac{1}{s}\left({\mathscr H}_{t+s}(f^1)-{\mathscr H}_t(f^1)\right)]  [\mathcal{F}, {\mathscr H}_{t+s}(f^{2})] \ldots  [\mathcal{F}, {\mathscr H}_{t+s}(f^{p})]\big)+ \ldots\\
& \quad +Tr\big(\epsilon {\mathscr H}_{t}(f^0)[\mathcal{F}, {\mathscr H}_{t}(f^1) ] \ldots [\mathcal{F},\frac{1}{s}\left({\mathscr H}_{t+s}(f^p)-{\mathscr H}_t(f^p)\right)]\big)\big)
\end{array}
\end{equation*}
By {\it(1)}, we know that the association $t \mapsto [\mathcal{F},{\mathscr H}_{t}(f)]$ is a continuous map for each morphism $f \in \mathcal{C}$. Therefore, we have
\begin{equation*}
\begin{array}{ll}
\lim\limits_{s \to 0}\big(Tr\big(\epsilon {\mathscr H}_{t}(f^0)[\mathcal{F}, {\mathscr H}_{t}(f^1)]\ldots [\mathcal{F}, {\mathscr H}_{t}(f^{j-1})] [\mathcal{F},\frac{1}{s}\left({\mathscr H}_{t+s}(f^{j})-{\mathscr H}_t(f^j)\right)] \ldots  [\mathcal{F}, {\mathscr H}_{t+s}(f^{p})]\big)\big)=\\
=\lim\limits_{s \to 0} (-1)^j \big(Tr\big(\epsilon [\mathcal F,{\mathscr H}_{t}(f^0)]\ldots [\mathcal{F}, {\mathscr H}_{t}(f^{j-1})] \frac{1}{s}\left({\mathscr H}_{t+s}(f^{j})-{\mathscr H}_t(f^j)\right)[\mathcal{F}, {\mathscr H}_{t+s}(f^{j+1})]  \ldots  [\mathcal{F}, {\mathscr H}_{t+s}(f^{p})]\big)\big)\\
= (-1)^j Tr\big(\epsilon [\mathcal F,{\mathscr H}_{t}(f^0)][\mathcal{F}, {\mathscr H}_{t}(f^1)]\ldots [\mathcal{F}, {\mathscr H}_{t}(f^{j-1})] \delta_t(f^j) [\mathcal{F}, {\mathscr H}_{t}(f^{j+1})] \ldots  [\mathcal{F}, {\mathscr H}_{t}(f^{p})]\big)
\end{array}
\end{equation*}
From this, we obtain
\begin{equation*}
\begin{array}{ll}
&\int_{t_1}^{r}\phi'(t) (f^0 \otimes \ldots \otimes f^p) ~dt+\int_{r}^{t_2}\phi'(t)  (f^0 \otimes \ldots \otimes f^p)~dt\\
&=\int_{t_1}^{r} \sum\limits_{j=0}^p (-1)^j Tr\big(\epsilon [\mathcal F,{\mathscr H}_{t}(f^0)][\mathcal{F}, {\mathscr H}_{t}(f^1)]\ldots [\mathcal{F}, {\mathscr H}_{t}(f^{j-1})] \delta_t(f^j)[\mathcal{F}, {\mathscr H}_{t}(f^{j+1})]  \ldots  [\mathcal{F}, {\mathscr H}_{t}(f^{p})]\big)~ dt\\
&\textrm{ }+\int_{r}^{t_2} \sum\limits_{j=0}^p (-1)^j Tr\big(\epsilon [\mathcal F,{\mathscr H}_{t}(f^0)][\mathcal{F}, {\mathscr H}_{t}(f^1)]\ldots [\mathcal{F}, {\mathscr H}_{t}(f^{j-1})] \delta_t(f^j) [\mathcal{F}, {\mathscr H}_{t}(f^{j+1})] \ldots  [\mathcal{F}, {\mathscr H}_{t}(f^{p})]\big)~ dt\\
&=\psi(id_{X_0} \otimes f^0 \otimes \ldots \otimes f^{p})
\end{array}
\end{equation*}
Hence
\begin{equation*}
\begin{array}{ll}
\psi(id_{X_0} \otimes f^0 \otimes \ldots \otimes f^{p})&=\phi_{t_2}(f^0 \otimes \ldots \otimes f^{p})-\phi_{r}(f^0 \otimes \ldots \otimes f^{p})+\phi_{r}(f^0 \otimes \ldots \otimes f^{p})-\phi_{t_1}(f^0 \otimes \ldots \otimes f^{p})\\
&=\phi_{t_2}(f^0 \otimes \ldots \otimes f^{p})-\phi_{t_1}(f^0 \otimes \ldots \otimes f^{p})\\
\end{array}
\end{equation*}
Since $\psi(f^0 \otimes \ldots \otimes f^{p} \otimes id_{X_0} )=0$, we now have
\begin{equation*}
\begin{array}{ll}
(B_0\psi) (f^0 \otimes \ldots \otimes f^{p})&=\psi(id_{X_0} \otimes f^0 \otimes \ldots \otimes f^{p})-\psi(f^0 \otimes \ldots \otimes f^{p} \otimes id_{X_0} )\\
&=(\phi_{t_2}-\phi_{t_1})(f^0 \otimes \ldots \otimes f^{p})
\end{array}
\end{equation*}
Since $b\psi=0$, using Proposition \ref{prop5.5h} and the fact that $\phi_{t_2}-\phi_{t_1} \in Ker(1-\lambda)$, we have
\begin{equation*}
0=S(B\psi)=S(AB_0\psi)=(p+1)S(\phi_{t_2}-\phi_{t_1})
\end{equation*}
This proves the result.
\end{proof}

\begin{theorem}\label{Thm5.6t}
Let $\mathcal C$ be a small $\mathbb{C}$-category  and $\{\rho_t:\mathcal{C} \longrightarrow SHilb_2\}_{t \in [0,1]}$ be a family of functors such that
for each $X\in Ob(\mathcal C)$, we have $\rho_t(X)=\rho_{t'}(X)$ for all $t$, $t'\in [0,1]$. We put $\rho(X):=\rho_t(X)$ for all $t \in [0,1]$.
 Further, for each $t \in [0,1]$,
let 
\begin{equation}
\mathcal{F}_t:=\left\lbrace \mathcal F_t(X):=\begin{pmatrix} 0 & \mathcal{Q}_t(X) \\ \mathcal{P}_t(X) & 0 \\
\end{pmatrix}:\rho(X)\longrightarrow  \rho(X) \right\rbrace_{X\in Ob(\mathcal C)}
\end{equation}
with $\mathcal{P}_t(X)=\mathcal{Q}_t^{-1}(X)$ be such that $({\rho}_t,\mathcal{F}_t)$ is a $p$-summable Fredholm module over the category $\mathcal{C}$.
We set $\rho(X)=\rho'(X)\oplus \rho'(X)\in SHilb_2$. We further assume that for some even integer $p$ and for any $f \in Hom_\mathcal{C}(X,Y)$, we have

\smallskip
(1) $t \mapsto {\rho}_t^+(f)-\mathcal{Q}_t{\rho}_t^-(f)\mathcal{P}_t$ is a continuous map from $[0,1]$ to $\mathcal{B}^p({\rho}'(X),{\rho}'(Y))$, where $\rho_t^\pm$ are the two components of the morphism $\rho_t$ of degree zero.

\smallskip
(2) $t \mapsto {\rho}_t^+(f)$ and $t \mapsto \mathcal{Q}_t{\rho}_t^-(f)\mathcal{P}_t$ are piecewise strongly $C^1$ maps from $[0,1]$ to $SHilb(\rho'(X),\rho'(Y))$.

\smallskip
Then,  the $(p+2)$-dimensional character  $\text{ch}^{p+2}({\rho}_t,\mathcal{F}_t) \in H^{p+2}_\lambda(\mathcal{C})$ is independent of $t \in [0,1]$.
\end{theorem}
\begin{proof}
For each $t \in [0,1]$, we set $\mathcal{T}_t:=\begin{pmatrix} 1 & 0 \\ 0 & \mathcal{Q}_t \\
\end{pmatrix}$. Then, $\mathcal{T}_t^{-1}=\begin{pmatrix} 1 & 0 \\ 0 & \mathcal{P}_t \\
\end{pmatrix}$ and $\mathcal{F}'_t:=\mathcal{T}_t\mathcal{F}_t\mathcal{T}_t^{-1}=\begin{pmatrix} 0 & 1\\ 1& 0 \\
\end{pmatrix}.$

\smallskip
For each $t \in [0,1]$, we also define a functor ${\mathscr H}_t:\mathcal{C} \longrightarrow SHilb_{\mathbb{Z}_2}$ given by
\begin{equation*}
{\mathscr H}_t(X):={\rho}(X) \qquad {\mathscr H}_t(f):=\mathcal{T}_t{\rho}_t(f)\mathcal{T}_t^{-1}
\end{equation*}
Then, we have
\begin{equation*}
[\mathcal{F}'_t,{\mathscr H}_t(f)]=\begin{pmatrix} 0 & \mathcal{Q}_t{\rho_t}^-(f)\mathcal{P}_t -{\rho}_t^+(f)\\ {\rho}_t^+(f)- \mathcal{Q}_t{\rho_t}^-(f)\mathcal{P}_t  & 0 
\end{pmatrix}
\end{equation*}
Therefore, using assumption {\it(1)}, we see that the map $t \mapsto [\mathcal{F}',{\mathscr H}_t(f)]$ from $[0,1]$ to $\mathcal{B}^p({\mathscr H}_t(X),{\mathscr H}_t(Y))$ is continuous for each $f \in Hom_\mathcal{C}(X,Y)$. Further,
\begin{equation*}
{\mathscr H}_t(f)=\mathcal{T}_t{\rho}_t(f)\mathcal{T}_t^{-1}=\begin{pmatrix} {\rho}_t^+(f) & 0\\ 0  &  \mathcal{Q}_t{\rho_t}^-(f)\mathcal{P}_t
\end{pmatrix}
\end{equation*}
Therefore, by applying assumption {\it(2)}, we see that the map $t \mapsto {\mathscr H}_t(f)$ is piecewise strongly $C^1$. Since trace is invariant under similarity, the result now follows using Proposition \ref{Conlemma}.
\end{proof}

\begin{theorem}\label{TfinalT}
Let $\mathcal C$ be a small $\mathbb{C}$-category  and $\{\rho_t:\mathcal{C} \longrightarrow SHilb_2\}_{t \in [0,1]}$ be a family of functors such that
for each $X\in Ob(\mathcal C)$, we have $\rho_t(X)=\rho_{t'}(X)$ for all $t$, $t'\in [0,1]$. We put $\rho(X):=\rho_t(X)$ for all $t \in [0,1]$.
 Further, for each $t \in [0,1]$ and $X\in Ob(\mathcal C)$, 
let 
\begin{equation*}
\mathcal F_t(X):=\begin{pmatrix} 0 & \mathcal{Q}_t(X) \\ \mathcal{P}_t(X) & 0 \\
\end{pmatrix}:\rho(X)\longrightarrow  \rho(X)
\end{equation*}
with $\mathcal{Q}_t^{-1}=\mathcal{P}_t$ be such that $({\rho}_t,\mathcal{F}_t)$ is a $p$-summable Fredholm module over the category $\mathcal{C}$.
We further assume that for some even integer $p$, we have

\smallskip
(1) For any $f \in Hom_\mathcal{C}(X,Y)$, $t \mapsto {\rho}_t(f)$ is a strongly $C^1$-map from $[0,1]$ to $SHilb_{\mathbb Z_2}({\rho}(X),{\rho}(Y))$.

\smallskip
(2) For any $X\in \mathcal C$, $t \mapsto \mathcal F_t(X)$   is a strongly $C^1$-map from $[0,1]$ to $SHilb_{\mathbb Z_2}({\rho}(X),{\rho}(X))$.

\smallskip
Then, the $(p+2)$-dimensional character $\text{ch}^{p+2}({\rho}_t,\mathcal{F}_t) \in H^{p+2}_\lambda(\mathcal{C})$ is independent of $t \in [0,1]$.
\end{theorem}

\begin{proof} By definition, ${\rho}_t(f)=\begin{pmatrix} \rho^+(f) & 0 \\ 
0 & \rho^-(f)\\ \end{pmatrix}$ and $\mathcal F_t(X)=\begin{pmatrix} 0 & \mathcal{Q}_t(X) \\ \mathcal{P}_t(X) & 0 \\
\end{pmatrix}$. As such, it is clear that a system satisfying the assumptions (1) and (2) above also satisfies the assumptions in Theorem \ref{Thm5.6t}. This proves the result.
\end{proof}

\end{document}